\documentclass[11pt]{article}

\usepackage[margin=1in]{geometry}
\usepackage{amsmath,amsthm,amssymb,amsfonts,float}
\usepackage{comment}
\usepackage[utf8]{inputenc}
\usepackage[T1]{fontenc}

\usepackage[colorlinks=true, pdfstartview=FitV, linkcolor=blue,citecolor=blue, urlcolor=blue]{hyperref}

\usepackage[abbrev,lite,nobysame]{amsrefs}
\usepackage{times}
\usepackage[usenames,dvipsnames]{color}

\usepackage{esint}

\usepackage{mathtools,enumitem,mathrsfs}

\usepackage[compact]{titlesec}

\usepackage[title]{appendix}

\usepackage{comment}

\newcommand{\eps}{\epsilon}

\newcommand{\leqc}{\lesssim}

\newcommand{\grad}{\nabla}

\newcommand{\norm}[1]{\left|\left| #1 \right|\right|}
\newcommand{\abs}[1]{\left| #1 \right|}
\newcommand{\set}[1]{\left\{ #1 \right\}}
\newcommand{\brak}[1]{\left\langle #1 \right\rangle} 

\newcommand{\R}{\mathbb{R}}
\newcommand{\N}{\mathbb{N}}

\newcommand{\T}{\mathbb{T}}

\usepackage{bbm}

\renewcommand{\P}{\mathbf{P}}

\newcommand{\E}{\mathbf{E}}

\newcommand{\PP}{\mathbf P}

\newcommand{\Pt}{\mathcal{P}}

\newcommand{\Hhyp}{H^1_{\text{hyp}}}
\newcommand{\Hhypd}{H^1_{\text{hyp},\delta}}

\newtheorem{theorem}{Theorem}[section]

\newtheorem{corollary}[theorem]{Corollary}
\newtheorem{lemma}[theorem]{Lemma}
\newtheorem*{lemma*}{Lemma}

\newtheorem{assumption}{Assumption}

\theoremstyle{definition}
\newtheorem{definition}[theorem]{Definition}
\newtheorem{remark}{Remark}

\DeclareUnicodeCharacter{2BC}{'}

\numberwithin{equation}{section}

\begin{document}

\title{Quantitative spectral gaps and uniform lower bounds in the small noise limit for Markov semigroups generated by hypoelliptic stochastic differential equations} 
\author{Jacob Bedrossian\thanks{\footnotesize Department of Mathematics, University of Maryland, College Park, MD 20742, USA \href{mailto:jacob@math.umd.edu}{\texttt{jacob@math.umd.edu}}.    Both J.B. and K.L were supported by NSF CAREER grant DMS-1552826 and NSF RNMS \#1107444 (Ki-Net)}  \and Kyle Liss\thanks{\footnotesize Department of Mathematics, University of Maryland, College Park, MD 20742, USA \href{mailto:kliss@umd.edu}{\texttt{kliss@umd.edu}}}}

\maketitle

\begin{abstract}
  We study the convergence rate to equilibrium for a family of Markov semigroups $\{\Pt_t^{\epsilon}\}_{\epsilon > 0}$ generated by a class of hypoelliptic stochastic differential equations on $\R^d$, including Galerkin truncations of the incompressible Navier-Stokes equations, Lorenz-96, and the shell model SABRA. 
  In the regime of vanishing, balanced noise and dissipation, we obtain a sharp (in terms of scaling) quantitative estimate on the exponential convergence in terms of the small parameter $\epsilon$.
  By scaling, this regime implies corresponding optimal results both for fixed dissipation and large noise limits or fixed noise and vanishing dissipation limits. 
As part of the proof, and of independent interest, we obtain uniform-in-$\epsilon$ upper and lower bounds on the density of the stationary measure.
Upper bounds are obtained by a hypoelliptic Moser iteration, the lower bounds by a de Giorgi-type iteration (both uniform in $\epsilon$).
The spectral gap estimate on the semigroup is obtained by a weak Poincar\'e inequality argument combined with quantitative hypoelliptic regularization of the time-dependent problem. 
\end{abstract}

\setcounter{tocdepth}{2}
{\small\tableofcontents}


\section{Introduction}\label{sec:Intro}

In this paper, we obtain optimal (in terms of scaling in $\epsilon$) quantitative estimates on the exponential convergence to equilibrium for a class of hypoelliptic PDEs on $\R^d$ of the form 
\begin{align} 
  \partial_t \mu_t & = L^*_\epsilon \mu_t  := \epsilon \left( \sum_{j=1}^r (Z_j \cdot \grad)^2 + Ax\cdot \grad + \text{Tr}(A)\right) \mu_t + \epsilon^\alpha Bx \cdot \grad \mu_t + N \cdot \grad \mu_t \label{eq:L*}
\end{align}
for parameters $0 < \epsilon \ll 1$ and $\alpha \geq 0$. Here, $\{Z_j\}_{j=1}^r$ is a collection of vector fields on $\R^d$ assumed to be constant in $x$, $A \in \mathbb{M}^{d\times d}$ (the vector space of $d\times d$ matrices with real entries) is a positive definite matrix that plays the role of dissipation, $B \in \mathbb{M}^{d\times d}$ is skew-symmetric (possibly zero), and $N:\R^d \to \R^d$ is a smooth, nonlinear drift such that $N(x): = N(x,x,\ldots,x)$ for a multlinear function $N(x_1,x_2,\ldots,x_p)$ of $p \ge 2$ arguments. We assume moreover that $N(x)$ is divergence free and obeys the energy conservation property
\begin{equation} \label{eq:energycons}
N(x) \cdot x = 0 \quad \forall x\in \R^d.
\end{equation}
The skew-symmetry of $B$ implies that $Bx$ satisfies (\ref{eq:energycons}) and $\grad \cdot Bx = \text{Tr}(B) = 0$, so that the term $\epsilon^\alpha Bx$ plays the role of a linear (and lower order when $\alpha > 0$) conservative drift. Complete statements of all our main assumptions are given at the beginning of Section~\ref{sec:IntroResults}.

The equation \eqref{eq:L*} is the forward Kolmogorov equation for the diffusion process on $\R^d$
\begin{equation}\label{eq:SDE}
dx_t^\epsilon = -\epsilon A x^\epsilon_t dt - \epsilon^\alpha B x^\epsilon_t dt -  N(x^\epsilon_t)dt + \sum_{j=1}^r \sqrt{2\epsilon} Z_j dW_t^{(j)}, 
\end{equation}
where $\set{W_t^{(j)}}_{j=1}^r$ are independent one-dimensional Wiener processes on a common filtered probability space. Under our assumptions on the vector fields in \eqref{eq:L*}, for any initial condition $x_0^\epsilon = x \in \R^d$ the SDE \eqref{eq:SDE} admits a unique, global solution $(x_t^{\epsilon})_{t \ge 0}$ (see Lemma~\ref{lem:StochasticWP} for a precise statement) that defines a Markov process with generator 
\begin{equation} \label{eq:L}
L_\epsilon : = \epsilon \sum_{j=1}^r (Z_j \cdot \grad)^2 - Ax\cdot \grad - \epsilon^\alpha Bx \cdot \grad - N\cdot \grad.
\end{equation}
The form of (\ref{eq:SDE}) is fairly general and captures a number of fundamental dynamical systems driven by a white-in-time forcing, such as Lorenz-96 \cite{Lorenz1996} and Galerkin truncations of the Navier-Stokes or the complex Ginzburg-Landau equations (see \cite{Majda16} and Section~\ref{sec:examples} below). Notice that we have chosen the scaling for which one can hope to prove bounds on the equilibrium density that do not depend on $\epsilon$ (often called \emph{fluctuation-dissipation scaling}).
Due to the homogeneity of $N$, by rescaling time and $x_t^\epsilon$, treating this scaling also implies corresponding statements for both the large forcing and the small dissipation cases. 

By \textit{hypoelliptic}, we mean that while $\set{Z_j}_{j=1}^r$ may not span $\R^d$, we assume that the set of vector fields $\set{\epsilon Ax + \epsilon^\alpha Bx + N,Z_1,...,Z_r}$ satisfies the classical \emph{parabolic H\"ormander condition} (see discussions in e.g. \cite{H11} and the references therein). In the remainder of the paper we identity vector fields on $\R^d$ and first-order differential operators, and for an open set $\Omega \subseteq \R^d$ we write $T(\Omega)$ for the collection of all smooth vector fields defined on $\Omega$. For $X,Y \in T(\Omega)$ we denote by $[X,Y] \in T(\Omega)$ the vector field
$$[X,Y] = XY - YX.$$
\begin{definition}[Locally uniform parabolic H\"ormander] \label{def:uniformHor}
For an open set $\Omega \subseteq \R^d$ and $\{X_0,X_1,\ldots,X_k\} \subseteq T(\Omega)$, let $V_0 = \{X_1, \ldots, X_k\}$ and
\begin{equation} \label{eq:Vn}
 V_n = V_{n-1} \cup \{[X_j,Y]: 0 \le j \le k, Y \in V_{n-1}\} \quad \forall n\ge 1.
\end{equation}
We say that the family  $\{X_0,\dots,X_k\}$ satisfies the \textit{parabolic H\"{o}rmander condition} on $\Omega$ if $\cup_{n=0}^\infty V_n$ spans $\R^d$ at every point $x \in \Omega$. We say that $\{X_0,\dots,X_k\}$ satisfies the \textit{uniform parabolic H\"{o}rmander condition on $\Omega$ with constants $(N_0,C_0) \in \N \times (0,\infty)$} if for every $x \in \Omega$ there exists a set $\{Y_i\}_{i=1}^d \subseteq V_{N_0}$ such that 
\begin{equation}
\label{eq:span}
|\text{det}(Y_1(x),Y_2(x),\ldots,Y_d(x))|^{-1} \le C_0.
\end{equation}
In many settings, especially time-independent ones, it is natural to allow $X_0$ in the definition of $V_0$ above. In this case, if $\cup_{n=0}^\infty V_n$ spans $\R^d$ at every point $x \in \Omega$, then $\{X_0,\ldots,X_k\}$ is said to satisfy \textit{H\"{o}rmander's condition} on $\Omega$. The notion of uniformity extends in the obvious way.
\end{definition}

Physically, Definition \ref{def:uniformHor} describes how the randomness injected into the system by the stochastic forcing spreads to all degrees of freedom through the action of the drift term. If $\{\epsilon A x + \epsilon^\alpha Bx + N, Z_1,\ldots, Z_r\}$ satisfies the parabolic H\"{o}rmander condition on $\R^d$ then the semigroup ($\Pt_t^\epsilon)^*$ generated by $L_\epsilon^*$ is instantly smoothing, despite the fact that $L_\epsilon^*$ is not elliptic. Moreover, it is well-known that in this case there is a unique probability measure $\mu_\epsilon$ solving $L_\epsilon^* \mu_\epsilon = 0$, and that for all $\epsilon > 0$, $\mu_\epsilon$ is absolutely continuous with respect to Lebesgue measure with a smooth density $f_\epsilon$ (see Lemma~\ref{lem:MarkovSemigroup}).
This measure is also the unique stationary measure for the Markov semigroup $\Pt_t^\epsilon$ generated by \eqref{eq:SDE}.
Denote the transition probabilities for a Markov process $x_t$ on a Polish space $\mathcal{X}$ as $\Pt_t(x,A) = \PP \set{x_t \in A | x_0 = x}$ for all $A \in \mathcal{B}(\mathcal{X})$ (the set of Borel sets) and $x \in \R^d$.
Recall that the Markov semigroup $\Pt_t$ is defined to act on bounded measurable functions $f: \mathcal{X} \to \R$ by
\begin{equation}
\Pt_t f(x) := \int_{\R^d} f(y)\Pt_t(x,dy),
\end{equation}
and a measure $\mu \in \mathcal{M}(\mathcal{X})$ (the space of Borel probability measures on $\mathcal{X}$) is called \textit{stationary} (or \textit{invariant}) for $\Pt_t$ if for every $A \in \mathcal{B}(\mathcal{X})$ one has
\begin{equation} \label{eq:stationary}
\Pt_t^*\mu(A) := \int_{\R^d}\Pt_t(x,A)\mu(dx) = \mu(A).
\end{equation} 
Finally, it is known that $\forall \epsilon > 0$ the semigroup $\Pt_t^\epsilon$ converges exponentially on suitable weighted spaces. In particular, let $V:\mathcal{X} \to [0,\infty)$ be continuous and for measurable functions $f: \mathcal{X} \to \R$ define 
\begin{equation} \label{eq:defCV}
\|f\|_{V} = \sup_{x \in \mathcal{X}} \frac{|f(x)|}{1+V(x)}.
\end{equation}
Then, it is known that if $V \in C^2(\R^d)$ satisfies 
 $$L_\epsilon V \le -\theta V + \lambda$$
for some constants $\lambda, \theta > 0$, i.e., $V$ satisfies a drift condition, then there exists $C_\epsilon, \gamma_\epsilon > 0$ (both depending on $\epsilon$) such that for all measurable $f:\R^d \to \R$ with $\|f\|_{V} < \infty$ there holds
\begin{align} 
\norm{\Pt_t^\epsilon f - \mu_\epsilon(f)}_{V} \leq C_\epsilon e^{-\gamma_\epsilon t} \norm{f - \mu_\epsilon(f)}_{V}, \label{ineq:PteCV}
\end{align}
where we have written $\mu_\epsilon(f) = \int f d\mu_\epsilon$. Even for infinite-dimensional analogs of \eqref{eq:SDE} (for e.g. complex Ginzburg-Landau in $\T^d$, Navier-Stokes in $\T^2$, or hyper-viscous Navier-Stokes in $\T^3$), the existence and uniqueness of a stationary measure $\mu_\epsilon$ (see e.g. \cite{FM95,HM06, HM11, KNS20,KNS20II,KZ20}) with smooth finite-dimensional projections \cite{MattPard06, HM11} is known for $\epsilon> 0$, as are exponential convergence results similar to \eqref{ineq:PteCV} (see e.g. \cite{HM08,GM05,GM06,KNS20} and the references therein).     

In general, it is a very important question to understand the limit $\epsilon \to 0$, both to characterize properties of $\mu_\epsilon$ and to quantify $C_\epsilon, \gamma_\epsilon$ as functions of $\epsilon$.
In the case of e.g. the (infinite dimensional) 3D stochastic Navier-Stokes equations, characterizing $\mu_\epsilon$ as $\epsilon \to 0$ is equivalent to understanding many properties of turbulence in the statistically stationary regime, whereas quantifying $\gamma_\epsilon$ amounts to estimating the convergence rate to statistical equilibrium with respect to the fluid viscosity in the inviscid limit, also a question of fundamental importance to the theory of turbulence (see discussions in \cite{Kup}).
In spite of its importance, little is known about quantifying $\gamma_\epsilon$, $C_\epsilon$.
In finite dimensions, it is not difficult to deduce (see Theorem \ref{thrm:optimal} below) that $\gamma_\epsilon \lesssim \epsilon$, but lower bounds are much harder to come by.
In infinite dimensions, the methods of e.g. \cite{HM08} yield a lower bound on $\gamma_\epsilon$ that is exponentially bad in $\epsilon$, even if one takes non-degenerate stochastic forcing. The situation in finite dimensions is not significantly better, as standard proofs of convergence to equilibrium for (\ref{eq:SDE}) similarly yield a spectral gap that is not even polynomial in $\epsilon$. See Remark \ref{rmk:Turb} below for further discussion on the differences between finite and infinite dimensions, especially in the context of fluid mechanics.

A key reason that previously proven lower bounds on $\gamma_\epsilon$ (in either finite or infinite dimensions) are far from optimal is a lack of quantitative irreducibility results. It is well-known that unique ergodicity and the convergence rate to equilibrium for a Markov semigroup is in part determined by its irreducibility properties \cite{MTBook,HMHarris,MT94}, in particular the extent to which the support of transition probabilities arising from distinct points either overlap (see e.g. [Assumption 2, \cite{HMHarris}]) or become arbitrarily close to each other (see e.g. \cite{HairerMattinglyScheutzow2011} and [Assumption 6, \cite{HM08}]). The former is common for finite-dimensional diffusions, while in degenerate, infinite-dimensional settings one often must resort to the latter.
In most of the previous works, the irreducibility is obtained by taking advantage of rare events in the forcing.  
Previous works such as e.g. \cite{HM08, EM}, use that for any initial condition, there is a small probability that the energy input from the noise is low enough that the dissipation causes the process to drift back to any neighborhood of the origin. Along with some regularity of transition probabilities and a suitable Lyapunov structure, this is a strong enough irreducibility statement to prove exponential convergence statements such as \eqref{ineq:PteCV}.  However, rare excursions to the origin are clearly not the actual mechanism for irreducibility in high-dimensional, chaotic systems, and as such the estimates on the mixing time one obtains from such an analysis are sub-optimal \cite{Kup}. 
More sophisticated approaches to irreducibility rely on H\"{o}rmander's condition and optimal control theory (see e.g. \cite{AgrachevSachkov, ScalSatGHM} and the references therein).
However, these arguments similarly rely on rare events where the diffusion completely dominates the drift and hence still do not yield any type of uniform-in-$\epsilon$ irreducibility for the transition probabilities of (\ref{eq:SDE}), nor do they capture true mechanisms behind mixing in the fluctuation dissipation limit.

While improving estimates on $\gamma_\epsilon$ in infinite dimensions seems to be an extremely challenging problem and currently out of reach, in this paper we rectify the above issue in finite dimensions and obtain the optimal estimates $\gamma_\epsilon \approx \epsilon$ and $C_\epsilon$ independent of $\epsilon$. As a necessary step, we also obtain a uniform-in-$\epsilon$, pointwise Gaussian upper bound on $f_\epsilon$ (the density of the stationary measure) and moreover obtain a uniform-in-$\epsilon$, pointwise lower bound on every compact set. 

\subsection{Main results and discussion}\label{sec:IntroResults}

In the statements of our results below, and throughout the remainder of the entire paper, we write $a \leqc b$ to mean that $a \le Cb$ for a constant $C$ depending possibly on $A$, $B$, $N$, $\{Z_j\}_{j=1}^r$, and the dimension $d$, but not on $\epsilon$. Any $\epsilon$ dependence in estimates or constants we define will always be made explicit. Also, throughout the entire paper we write $B_R$ for the open ball of radius $R$ centered at the origin. 

\subsubsection{Statement of main assumptions}
We now state precisely our main assumptions. They consist of the dissipative and conservative structures of $A$, $B$, and $N$, a uniform-in-$\epsilon$ nondegeneracy condition, and strict (but \textit{qualitative}) positivity of the stationary density $f_\epsilon$ for every $\epsilon > 0$. 
\begin{assumption} \label{Assumption:VF}
The matrix $A \in \mathbb{M}^{d\times d}$ is positive definite, $B \in \mathbb{M}^{d\times d}$ is skew-symmetric, and $N(x):=N(x,\ldots,x)$ for a smooth, multilinear function of $p \ge 2$ arguments satisfying the conservation properties $\grad \cdot N(x) = 0$ and \eqref{eq:energycons}.
\end{assumption}

\begin{assumption} \label{Assumption:spanning}
For every $R > 0$ there exists $M \in \N$ and $C > 0$ (depending possibly on $R$) so that for every $\epsilon_1, \epsilon_2 \in [0,1]$ the collection of vector fields
$$\{N + \epsilon_1 Ax + \epsilon_2 B x,Z_1, \ldots, Z_r\}$$ 
satisfies the uniform parabolic H\"{o}rmander condition on $B_R$ with constants $(M,C)$. 
\end{assumption}

\begin{assumption} \label{Assumption:positivity}
	For every $\epsilon > 0$ the smooth density $f_\epsilon$ of the unique stationary measure $\mu_\epsilon$ for $\Pt_t^\epsilon$ is strictly positive. 
\end{assumption}

We discuss specific examples that satisfy Assumptions~\ref{Assumption:VF}-\ref{Assumption:positivity} in Section~\ref{sec:examples}. For now, we remark that the uniform-in-$\epsilon$ spanning condition of Assumption~\ref{Assumption:spanning} is quite natural for \eqref{eq:SDE}. Indeed, one usually verifies the parabolic H\"{o}rmander condition by showing that the collection 
$$ \{Z_1,\ldots,Z_r\} \cup\{Y: Y = [\ldots[N,Z_{i_1}],Z_{i_2}],\ldots],Z_{i_{p-1}}],Z_{i_p}], 1\le i_j \le r\} $$
satisfies H\"{o}rmander's condition on $\R^d$; see for example \cite{EM}. Since $p > 1$, in this situation the linear drift terms $\epsilon_1 Ax$ and $\epsilon_2 Bx$ do not change the bracket structure and so Assumption~\ref{Assumption:spanning} is satisfied. Regarding Assumption~\ref{Assumption:positivity}, strict positivity of the stationary density is typically proven by combining a suitable hypoellipticity assumption (in our setting implied by Assumptions~\ref{Assumption:VF} and~\ref{Assumption:spanning}) with exact controllability (see e.g. \cite{ScalSatGHM,HerMat15} and the references therein). 

\begin{remark} \label{rem:Assumptions}
	Beyond stating our main results in Sections~\ref{sec:IntroHypoBounds} and~\ref{sec:IntroGeo} below, throughout the entire paper we will always assume that Assumptions~\ref{Assumption:VF} and~\ref{Assumption:spanning} both hold unless remarked otherwise. On the other hand, Assumption~\ref{Assumption:positivity} is only needed in select locations, and so we will always indicate explicitly when it is required. 	
\end{remark}

\subsubsection{Uniform-in-$\epsilon$ hypoelliptic estimates} \label{sec:IntroHypoBounds}
In this section we state quantitative hypoelliptic estimates, in particular uniform-in-$\epsilon$ pointwise bounds on the equilibrium density $f_\epsilon$, and a long-time $L^2 \to L^\infty$ regularization estimate for $\Pt_t^\epsilon$. We are motivated to obtain such bounds mostly to use as lemmas in the proof that $\gamma_\epsilon \approx \epsilon$. However, they are of independent interest, as estimates on hypoelliptic equations that are uniform in a small parameter are a delicate matter.

\begin{theorem}[Uniform-in-$\epsilon$ estimates on $f_\epsilon$]\label{thrm:stationary}
Under Assumptions~\ref{Assumption:VF} and~\ref{Assumption:spanning} there exists $\lambda_- > 0$ so that the smooth density $f_\epsilon$ of the unique stationary measure $\mu_\epsilon$ for (\ref{eq:SDE}) satisfies the pointwise estimate
\begin{equation} \label{eq:Gaussubd}
\sup_{\epsilon \in (0,1)} f_\epsilon(x)  \leqc e^{-\lambda_- |x|^2}.
\end{equation}
If in addition Assumption~\ref{Assumption:positivity} is satisfied, then we also have the lower bound
\begin{equation}
\label{eq:lbd}
\inf_{\epsilon \in (0,1)}\inf_{|x| \le R}f_\epsilon(x)  \gtrsim_R 1.
\end{equation}
\end{theorem}

\begin{remark}\label{rem:lbdWOpositivity}
	Without Assumption~\ref{Assumption:positivity}, our proof of \eqref{eq:lbd} shows that for every $R > 0$ there exists $\epsilon_*(R) \in (0,1)$ such that 
	$$ \inf_{\epsilon \in (0,\epsilon_*)}\inf_{|x| \le R} f_\epsilon(x) \gtrsim_R 1. $$
\end{remark}

\begin{remark}
	A consequence of the proof of Theorem~\ref{thrm:stationary} is that we have uniform-in-$\epsilon$ control on $\|f_{\epsilon}\|_{H^s_{\text{loc}}}$ for some $s$ sufficiently small (independent of $\epsilon$). We do not know how to obtain uniform-in-$\epsilon$ bounds in higher regularity, nor do we even necessarily expect such bounds to be true. In particular, we do not have continuity in any sense that is uniform as $\epsilon \to 0$.
\end{remark}

\begin{lemma}[Quantitative $L^2_{\mu_\epsilon} \to L^\infty$ regularization]\label{lem:ParReg} 
	Write $\|f\|_{L^2_\mu}^2 = \int_{\R^d}|f(x)|^2 \mu(dx)$ for a measure $\mu \in \mathcal{M}(\R^d)$ and measurable function $f$. Under Assumptions~\ref{Assumption:VF}-\ref{Assumption:positivity}, for every bounded and measurable observable $f:\R^d \to \R$ and $R \ge 1$ there holds, uniformly in $\epsilon \in (0,1)$, 
	\begin{equation} \label{eq:ParRegThrm}
	\|\Pt_{\epsilon^{-1}}^\epsilon f\|_{L^\infty(B_R)} \leqc_R \|f\|_{L_{\mu_\epsilon}^2}. 
	\end{equation}
\end{lemma}

There have been many methods put forward for obtaining $L^2_\text{loc} \to L^\infty_{\text{loc}}$ type bounds and pointwise estimates of hypoelliptic equations or their time-evolution counterparts.\footnote{sometimes called ``ultraparabolic'' in the literature, however, we find this name misleading and so we do not use it.}
To our knowledge, there has not been any work that is quantitative in a small parameter such as $\epsilon$ here; instead, much of the work is focused on reducing regularity requirements on the coefficients. One general set of methods is focused on obtaining pointwise upper and lower bounds on fundamental solutions,
often using an explicit local approximation combined with optimal control arguments; see e.g.~\cite{LanconelliPolidoro1994,Polidoro1997global,Pascucci2004,DiFranPolidori2006,CintiEtAl2010,Polidoro16,Anceschi2019,LanconelliEtAl2020,FarhanGiulio19} and the references therein. Such upper and lower bounds then provide a relatively straightforward path towards adapting many classical parabolic and elliptic methods. There are also the related works \cite{WangZhang2011,WangZhang2009}, which prove H\"{o}lder regularity for a class of degenerate parabolic equations with measurable coefficients.
Another set of methods, at least specifically in the context of kinetic theory, have recently been proposed which focus on adapting de Giorgi-Nash-Moser methods, with the key starting point being a local gain of integrability available from velocity averaging lemmas \cite{Bouchut2002}; see e.g.~\cite{GIMV16,Mouhot2018} and the earlier preprints \cite{GV15,IM15}. Recently on kinetic Fokker-Planck there is also \cite{AM19}, which is closer in spirit to the original work of H\"ormander \cite{H67}, and isolates the natural functional framework in which the variational treatment of elliptic equations extends to kinetic Fokker-Planck. 

Obtaining uniform-in-$\epsilon$ estimates is a somewhat different problem than lowering regularity requirements on coefficients. In fact, it is quantifying a priori estimates in the original work of H\"{o}rmander \cite{H67} combined with the adaptation of certain ideas from \cite{GIMV16} for extending de Giorgi-Nash-Moser theory to hypoelliptic settings that form the basis for our proof of Theorem~\ref{thrm:stationary} and Lemma~\ref{lem:ParReg}. We obtain both upper bounds (\ref{eq:Gaussubd}) and (\ref{eq:ParRegThrm}) with hypoelliptic Moser iterations.
With proper use of the structural assumptions $\grad \cdot N = N \cdot x = 0$, for both results, the main difficulty is obtaining a local gain of integrability that does not depend on $\epsilon$.
For this we derive suitable uniform H\"ormander inequalities (see Section \ref{sec:outline} for discussion).
The parabolic version (Lemma \ref{lem:Horparabolic}) is the more delicate of the two, and while the proof does not require deep modifications to H\"{o}rmander's original methods, to our knowledge nothing quite analogous can be found in the literature. 
An additional challenge for \eqref{eq:Gaussubd} as compared to previous works such as \cite{GIMV16,GV15, IM15}, is that to close the iteration scheme we must deduce a uniform-in-$\epsilon$ bound on $\|f_\epsilon\|_{L^2}$, which requires using a H\"ormander inequality that is quantitative also in the diameter of the set (Lemma \ref{lem:HineqBall/Ann}) and moment bounds coming from the drift condition that $x\cdot N = 0$ allows to deduce.   

The main difficulty in obtaining the lower bound \eqref{eq:lbd} is to adapt a compactness-rigidity argument from \cite{GV15,IM15} used to prove an isoperimetic inequality for subsolutions of a kinetic Fokker-Planck to our setting. In particular, we show that subsolutions to (\ref{eq:L*}) obeying some additional conditions satisfy a uniform-in-$\epsilon$ isoperimetric inequality (see Lemma~\ref{lem:IVT}). This is then combined with a quantitative hypoelliptic $L^2 \to L^\infty$ estimate for solutions to $L_\epsilon^* f = 0$ (Lemma~\ref{lem:Moser}) and classical ideas from the de Giorgi elliptic theory \cite{Vasseur2016}. It remains an interesting question to obtain quantitative H\"{o}lder regularity in the small parameter limits for PDEs with the form of (\ref{eq:L*}). This remains out of reach with our current techniques and a direction of future interest, since our present methods rely crucially on $N(x) \cdot x = 0$ and are not invariant under translations. 

\subsubsection{Quantitative geometric ergodicity and consequences} \label{sec:IntroGeo}
In this section we give precise statements of our results on the geometric ergodicity of (\ref{eq:SDE}) and some immediate consequences when combined with the results from Section~\ref{sec:IntroHypoBounds}. First, we need to define an appropriate notion of a uniform drift condition.

\begin{definition} \label{def:uniformLyapunov}
We say that a nonnegative function $V \in C^2(\R^d)$ is a \textit{uniform Lyapunov function for $\{\Pt_t^\epsilon\}_{\epsilon \in (0,1)}$} if there exists $\kappa, b > 0$ so that for all $\epsilon \in (0,1)$ and $\delta \in [0,1]$ there holds
\begin{equation} \label{eq:LyapFuncGen}
\epsilon \delta \Delta V + L_\epsilon V \le - \epsilon \kappa V + \epsilon b.
\end{equation}
\end{definition} 
We include the term $\epsilon \delta \Delta V$ on the left-hand side since at times it will be convenient to work with the regularized operator $\epsilon \delta \Delta + L_\epsilon^*$. It is easy to check that $V(x) = e^{\gamma x^2}$ is a uniform Lyapunov function provided that $\gamma$ is chosen sufficiently small. Along with the notations from Section~\ref{sec:Intro} we then have the following.

\begin{theorem}[Quantitative geometric ergodicity]\label{thrm:GeoErg}
Let $V$ be a uniform Lyapunov function for $\{\Pt_t^\epsilon\}_{\epsilon \in (0,1)}$. Under Assumptions~\ref{Assumption:VF}-\ref{Assumption:positivity} there exists $K, \delta > 0$ that do not depend on $\epsilon$ such that for all $\epsilon \in (0,1)$, $t > 0$, and measurable $f:\R^d \to \R$ satisfying $\|f\|_{V} < \infty$ there holds
	\begin{equation} \label{eq:GeoErg}
	\|\Pt_t^\epsilon f - \mu_\epsilon(f)\|_{V} \le Ke^{-\epsilon \delta t}\|f - \mu_\epsilon(f)\|_{V}.
	\end{equation}
\end{theorem}

\begin{remark}
As $\Pt_t:C_V \to C_V$ (where $C_V$ denotes the closure of $C^\infty_0$ under the norm $\max_{x} f(x)/(1 + V(x)) < \infty$) defines a $C_0$-semigroup for all $V$ which are uniform Lyapunov functions, Theorem \ref{thrm:GeoErg} implies that $f \equiv 1$ is an isolated, dominant of eigenvalue of $\Pt_t$ and provides a quantitative estimate on the spectral gap separating it from the rest of the spectrum (in particular $\sigma(\Pt_t) \subset \set{1} \cup \set{z \in \mathbb C : \abs{z} \leq e^{-\epsilon \delta' t}}$ for all $\delta' < \delta$).
Using standard semigroup theory (e.g. [Theorem 3.6, \cite{EN01}]) a corresponding estimate holds also on the generator $L_\epsilon$, i.e. $\sigma(L_\epsilon) \subset \set{0} \cup \set{z \in \mathbb C : \textbf{Re} \, z < -\epsilon \delta}$. 
\end{remark}

By duality, Theorem~\ref{thrm:GeoErg} also implies a corresponding statement on the convergence of the law of $x_t^\epsilon$ as a measure on $\R^d$ to $\mu_\epsilon$ in a weighted total variation space. For a Polish space $\mathcal{X}$, a continuous function $V:\mathcal{X} \to [0,\infty)$, and $\mu, \nu \in \mathcal{M}(\mathcal{X})$ we write 
\begin{equation} \label{eq:defTVV}
\|\mu - \nu\|_{TV,V} = \sup_{\|f\|_{V} \le 1}\int f(x)(\mu(dx) - \nu(dx)).
\end{equation} 

\begin{corollary} \label{cor:GeoErg}
	Under the assumptions and notations of Theorem~\ref{thrm:GeoErg}, there exists $K, \delta > 0$ that do not depend on $\epsilon$ such that for all $\epsilon \in (0,1)$, $t > 0$, and measures $\mu \in \mathcal{M}(\R^d)$ satisfying $\int V(x) \mu(dx) < \infty$ there holds
	\begin{equation}
	\|(\Pt_t^{\epsilon})^* \mu - \mu_\epsilon\|_{TV,V} \le K e^{-\epsilon \delta t}\|\mu - \mu_\epsilon\|_{TV,V}.
	\end{equation}	
\end{corollary}

Many techniques exist for studying exponential convergence to equilibrium of a Markov process. Perhaps the most well-known and flexible methods are Harris type theorems, which combine drift towards a ``small set'' and a type of local irreducibility there to yield an explicitly computable rate of convergence in weighted total variation or Wasserstein distances; see e.g.~\cite{HairerMattinglyScheutzow2011,HMHarris,MT94}. Related criterion for subgeometric rates of convergence have also been studied \cite{Butkovsky2014, Durmus2016, Douc2009}. For examples of works using a Harris theorem framework in the setting of (\ref{eq:SDE}) we refer to \cite{EM} and \cite{HM08} mentioned above. In finite-dimensional situations, an entirely different class of techniques exist that use the Kolmogorov equation (i.e., PDE approaches) and functional inequalities involving the equilibrium density; see e.g.~\cite{Villani2009, Bak08, Arnold2001} and the references therein. Most directly, for elliptic generators with the form $L = \Delta + X \cdot \grad$, a Poincar\'{e} inequality in $L^2_\mu$ ($\mu$ being the stationary measure) implies exponential convergence to equilibrium in the same space (for related results and discussion see e.g. \cite{Bak08}). This is a consequence of the a priori estimate
\begin{equation} \label{eq:L2muApriori}
\frac{d}{dt}\|\Pt_t f - \mu(f)\|_{L^2_\mu}^2 = - 2 \|\grad(\Pt_t f - \mu(f))\|_{L^2_\mu}^2.
\end{equation}
 Poincar\'{e} inequalities also play a crucial role in degenerate settings; see for example [Theorem 35, \cite{Villani2009}], which shows that exponential convergence to equilibrium for the kinetic Fokker-Planck equation with a $C^2$ confining potential $V:\R^d \to \R$ (satisfying a natural upper bound) is implied by an $L^2$ Poincar\'{e} inequality for $e^{-V(x)}dx$. Methods for proving convergence to equilibrium based on weaker functional inequalities are also known. For example, weak Poincar\'{e} inequalities, which trace back to \cite{Liggett1991} and were extended to a more general form in \cite{RockWang01} (see also \cite{Grothaus2019, Hu2019} for applications in degenerate settings). The key feature of weak Poincar\'{e} inequalities is that they allow for a small loss of a norm stronger than $L^2_\mu$ on the right-hand side, the most common example being
\begin{equation} \label{eq:WeakPoincEx}
\|f - \mu(f)\|_{L^2} \le \beta(s)\|\grad f\|_{L^2_\mu} + s \|f\|_\infty,
\end{equation}
where $\|f\|_\infty:=\sup_{x\in \R^d}|f(x)|$ and the inequality is required to hold for every $s > 0$ and some nonincreasing function $\beta:(0,\infty) \to [1,\infty)$ that possibly blows up as $s \to 0$ (see [Theorem 1.4, \cite{Bak08}] for a related but more general inequality). As such, they are much more forgiving to prove than standard Poincar\'{e} inequalities, but when applied in \eqref{eq:L2muApriori} only result in a subgeometric rate of convergence and from a stronger norm to a weaker norm.

The only uniform-in-$\epsilon$ information on $f_\epsilon$ that we currently have are the pointwise bounds stated in Theorem~\ref{thrm:stationary}, which are far from enough to imply a Poincar\'{e} inequality (see e.g.~\cite{BakryCattGuill2008, Bak08, Villani2009} for common conditions on a measure that yield a Poincar\'{e} inequality). Moreover, as discussed in Section~\ref{sec:Intro}, uniform-in-$\epsilon$ irreducibility statements are not forthcoming from standard methods. As such, it is not clear what the starting point for a proof of (\ref{eq:GeoErg}) should be. Our idea is to extend to the hypoelliptic setting the interesting fact that any measure $\mu(dx) = Ce^{-V(x)}dx$ for $V: \R^d \to \R$ that is merely locally bounded satisfies an (elliptic) \emph{weak} Poincar\'{e} inequality. This is done in Lemmas~\ref{lem:poincinterp} and~\ref{lem:weakpoinc}, where we prove a hypoelliptic version of (\ref{eq:WeakPoincEx}) that implies the decay estimate 
\begin{equation} \label{eq:IntroWeakPoinc}
\|\Pt_t^\epsilon f - \mu_\epsilon(f)\|_{L^2_{\mu_\epsilon}} \le \psi(\epsilon t)\|f - \mu_\epsilon(f)\|_{L^\infty}
\end{equation}
for some function $\psi: [0,\infty) \to [0,\infty)$ with $\lim_{t\to \infty}\psi(t) = 0$ and every bounded, Borel measurable function $f: \R^d \to \R$. One of our main insights is that the hypoelliptic regularization of Lemma~\ref{lem:ParReg}, when combined with a uniform Lyapunov function $V$ and the local equivalence of $\|\cdot\|_{L^2_{\mu_\epsilon}}$ and $\|\cdot\|_{L^2}$ given by Theorem~\ref{thrm:stationary}, allows to upgrade (\ref{eq:IntroWeakPoinc}) to exponential decay in $\|\cdot\|_V$. This is done by using \eqref{eq:lbd} and \eqref{eq:ParRegThrm} to show that for every $R \ge 1$ there is a $T(R) > 0$ such that
\begin{equation} \label{eq:IntroLinftyDecay}
\|\Pt_{\epsilon^{-1}T}^\epsilon f - \mu_\epsilon(f)\|_{L^\infty(B_R)} \le \frac{1}{2}\|f - \mu_\epsilon(f)\|_{L^\infty},
\end{equation}
which is then applied as the ``small set" condition in a standard Harris theorem; see Section~\ref{sec:OutlineTD} for details of the argument. To our knowledge, this particular scheme for obtaining exponential convergence by combining a weak Poincar\'{e} inequality with a local regularization estimate and drift condition has not appeared in the literature. We believe that this approach is of general interest and could be useful in other related problems. 

The fact that Theorem~\ref{thrm:GeoErg} is optimal with respect to the scaling of $\gamma_\epsilon, C_\epsilon$ is described in the following. 
\begin{theorem}[Optimal $\epsilon \to 0$ scaling of Theorem~\ref{thrm:GeoErg}]\label{thrm:optimal}
Let $\gamma$ be small enough so that $V = e^{\gamma x^2}$ is a uniform Lyapunov function and suppose that Assumptions~\ref{Assumption:VF}-\ref{Assumption:positivity} are satisfied. For every $s < 1$ and $K,\delta > 0$ there exists a measurable function $f:\R^d \to \R$ satisfying $\|f\|_{V} < \infty$ and an $\epsilon_0 > 0$ so that for all $\epsilon \in (0,\epsilon_0)$ there exists $t_*(\epsilon)$ such that
$$ \|\Pt_{t_*}^\epsilon f - \mu_\epsilon(f)\|_{V} \ge Ke^{-\epsilon^s \delta t_*}\|f - \mu_\epsilon(f)\|_{V}.$$
Similarly, if $\delta > 0$ and $\{K_\epsilon\}_{\epsilon \in (0,1)}$ are such that 
$$\|\Pt_t^\epsilon f - \mu_\epsilon(f)\|_{V} \le K_\epsilon e^{-\epsilon \delta t}\|f-\mu_\epsilon(f)\|_{V}$$ 
for every $t > 0$ and measurable $f:\R^d \to f$ with $\|f\|_{V} < \infty$, then $\liminf_{\epsilon \to 0}K_\epsilon > 0.$
\end{theorem}

\begin{remark} \label{rmk:Turb}
Note that fixing dimension and sending $\epsilon \to 0$ will yield very different results from sending $\epsilon \to 0$ in infinite dimensional problems, due to the possible development of turbulence. Unlike in the finite dimensional case, due to anomalous dissipation, different balances of dissipation vs forcing are possible and what one can see in each scaling depends on whether one has a direct cascade and/or inverse cascade of conserved quantities (see \cite{Nazarenko2011} for a discussion on inverse and direct cascades). 

It is clear that any results will be deeply tied to the topology, for example, for Batchelor-regime passive scalar turbulence, in fluctuation dissipation scaling as in \eqref{eq:SDE}, $\mu_\epsilon \rightharpoonup \delta_0$ in $H^s$ for $s < 1$ and $\mu_\epsilon \rightharpoonup 0$ for $H^s$ for $s>1$ (losing all mass to infinity); one requires a different scaling to capture non-trivial dynamics \cite{BBPS18,BBPS19}.  
In the hypoelliptic setting at least, there is no known, reasonable reference measure with respect to which one can study the stationary density, and even in the case of non-degenerate forcing, it is unclear what estimates could be expected for systems such as \eqref{eq:SDE} and the methods for proving any such estimates are essentially non-existent at the current time. 
As far as $\gamma_\epsilon$ is concerned, it is not clear what could be expected, for example, the presence of anomalous dissipation could conceivably result in $\gamma_\epsilon$ being larger than is possible in finite dimensions. 
\end{remark}

\subsection{Examples} \label{sec:examples}
A wide variety of systems fall under the general form \eqref{eq:SDE} that satisfy Assumptions \ref{Assumption:VF}--\ref{Assumption:positivity}. Working in vorticity form, Galerkin truncations (of arbitrary dimension) of the 2D Navier-Stokes equations in a periodic box can be written in the form \eqref{eq:SDE} with Assumption \ref{Assumption:VF}.
Minimal conditions on the forcing to obtain Assumption \ref{Assumption:spanning} were obtained in \cite{EM,HM06}.
The verification of Assumption \ref{Assumption:positivity} follows from the geometric control theory discussions in \cite{ScalSatGHM,HerMat15}. 
In \cite{Lorenz1996}, Lorenz put forward the following model (now known as \emph{Lorenz-96}) for $n$ real-valued oscillators $u_1,...,u_n$ in a periodic ensemble $u_{i+kn} = u_i$ (after rescaling to  match \eqref{eq:SDE}) 
\begin{align}
\partial_t u_m = (u_{m+1}-u_{m-2})u_{m-1} - \epsilon u_m + \sqrt{2 \epsilon} q_m dW_t^{(m)}, \label{def:L96intro}
\end{align}
where $\set{W_t^{(m)}}$ are independent Brownian motions and $\set{q_m}$ are fixed parameters. This model has been studied as a prototypical chaotic, high dimensional system (see e.g. \cite{Majda16,KP10,LK98}).
It is not hard to check the structural Assumption \ref{Assumption:VF}, and moreover, it is not hard to check that the uniform parabolic H\"ormander condition, Assumption \ref{Assumption:spanning}, is satisfied provided that $q_1, q_2 \neq 0$.
See discussions in \cite{ScalSatGHM,HerMat15} for verification of Assumption \ref{Assumption:positivity}, as like 2D Navier-Stokes, Lorenz-96 satisfies the structural assumptions sufficient to use geometric control arguments despite the even nonlinearity.   
The SABRA shell model was introduced in \cite{LvovEtAl98} to mimic many of the properties of turbulence; truncated to finite dimensional $u \in \mathbb C^J$, which we will regard as evolving on $\mathbb R^{2J}$, the model becomes (with the obvious convention that $u_k = 0$ if $k \not\in\set{1,...,J}$)
\begin{align}
\partial_t u_m & = i 2^m\left(\overline{u_{m+1}} u_{m+2} - \frac{\delta}{2}\overline{u_{m-1}}u_{m+1} + \frac{\delta-1}{4} u_{m-2}u_{m-1}\right) \\ & \quad - \epsilon 2^{2m} u_m + \sqrt{\epsilon} q_m dW_t^{(m;R)} + i \sqrt{\epsilon} p_m dW_t^{(m;I)},
\label{def:SABRA}  
\end{align}
for real parameters $q_m,p_m$ and a fixed parameter $\delta \in (0,2) \setminus \{ 1 \}$ (for $\delta \in (0,1)$ the model is meant to capture some properties of the energy/enstrophy cascades in 2D Navier-Stokes and for $\delta \in (1,2)$, the energy/helicity cascade in 3D Navier-Stokes). See \cite{Ditlevsen2010} for more discussion on this model.
Assumption~\ref{Assumption:VF} is straightforward as for Lorenz-96. Writing in real variables $u = a + i b$ we see that 
\begin{align*}
Z_{0} & = -\sum_{\ell=1}^J 2^\ell \left(\left(a_{\ell+1} b_{\ell+2}  - b_{\ell+1} a_{\ell+2}\right) + \frac{\delta}{2}\left(a_{\ell-1} b_{\ell+1} - b_{\ell-1} a_{\ell+1} \right)  + \frac{\delta-1}{4}\left(b_{\ell-2}a_{\ell-1} + a_{\ell-2} b_{\ell-1} \right)  \right)\partial_{a_\ell} \\
& \quad + \sum_{\ell=1}^J 2^\ell \left( \left(a_{\ell+1} a_{\ell+2} + b_{\ell+1} b_{\ell+2}\right) + \frac{\delta}{2}\left(a_{\ell-1}a_{\ell+1} + b_{\ell-1}b_{\ell+1}\right) + \frac{\delta-1}{4}\left(a_{\ell-2}a_{\ell-1} - b_{\ell-2} b_{\ell-1}\right) \right)\partial_{b_\ell}.  
\end{align*}
Despite the appearance, it is actually straightforward to verify Assumption \ref{Assumption:spanning} for this model under the condition that $q_1,q_2,p_1,p_2$ are all non-zero.
The local coupling allows for a relatively easy proof by induction, supposing first that the Lie algebra contains $\set{\partial_{a_j},\partial_{b_j}}_{j=1}^m$ and then using this to deduce that it also contains the directions $\set{\partial_{a_{m+1}},\partial_{b_{m+1}}}$.
The argument is essentially dictated by the terms containing the $b_{\ell-2}$ or $a_{\ell-2}$ factors, as one sees when computing the brackets of the form $[\partial_{a_{m-1}},Z_0]$.
As in the case of Lorenz-96,  Assumption \ref{Assumption:positivity} follows from discussions in \cite{ScalSatGHM,HerMat15}. 

\section{Outline} \label{sec:outline}
In this section we discuss the main steps in the proofs of Theorems~\ref{thrm:stationary} and~\ref{thrm:GeoErg}. Recall the convention from Remark~\ref{rem:Assumptions}.

\subsection{Time-stationary problem: upper bounds} 
Our goal in this section is to sketch the proof of the uniform-in-$\epsilon$ local upper bound
\begin{equation} \label{eq:ubd}
\sup_{\epsilon  \in (0,1)}\|f_\epsilon\|_{L^\infty(B_R)} \leqc_R 1.
\end{equation}
From there, the Gaussian upper bound (\ref{eq:Gaussubd}) follows from a comparison principle argument; see Section \ref{sec:StationaryGaussian} for details.

The natural starting point for a proof of (\ref{eq:ubd}) is to determine the uniform-in-$\epsilon$ a priori estimates that are available for solutions to the problem 
\begin{equation} \label{eq:statprob} 
\begin{cases}
L^*_\epsilon f_\epsilon = 0, \\ 
f_\epsilon \ge 0,\\ 
 \int f_\epsilon = 1.
\end{cases}
\end{equation}
From a probabilistic point-of-view, the most immediate estimate is a moment bound following from the fact that $V(x) = e^{\gamma x^2}$ is a uniform Lyapunov function (see Definition~\ref{def:uniformLyapunov}) whenever $\gamma \ll 1$. Integrating $0 = V L_\epsilon^* f_\epsilon$ and using that $\int f_\epsilon = 1$ yields, in the notation of \eqref{eq:LyapFuncGen},
$$
\int e^{\gamma x^2} f_\epsilon \le \frac{b}{\kappa} \leqc 1. 
$$
On the other hand, from the perspective of elliptic PDEs, the natural a priori bound is the energy estimate that holds for sufficiently regular, nonnegative subsolutions obtained by pairing $L_\epsilon^* f \ge 0$ with $f$ and integrating by parts.
Let $\mathscr{X}$ denote the natural energy norm defined by 
$$\|f\|_{\mathscr{X}}:= \|f\|_{L^2} + \sum_{j = 1}^r \|Z_j f\|_{L^2}.$$
The contributions from $N$ and $B$ to the energy estimate both vanish due to $\grad \cdot Bx = \grad \cdot N = 0$, and so we obtain 
\begin{equation}\label{eq:apriorienergy}
L_\epsilon^* f \ge 0 \implies \|f\|_{\mathscr{X}} \leqc \|f\|_{L^2}.
\end{equation}
Since the collection $\{Z_j\}_{j=1}^r$ does not alone satisfy H\"{o}rmander's condition, \eqref{eq:apriorienergy} should be supplemented with some type of uniform estimate on the drift vector field 
\begin{equation} \label{eq:Z_0eps}
Z_{0,\epsilon} := \epsilon Ax + \epsilon^\alpha B x + N.
\end{equation}
This bound is far more subtle than (\ref{eq:apriorienergy}) and comes in the form of an estimate in the norm dual to $\mathscr{X}$. In particular, define the norm
$$\|f\|_{\mathscr{X}^*}:=\sup_{\varphi \in C_0^\infty, \|\varphi\|_{\mathscr{X}} \le 1} \int \varphi f.$$
Then, using $L^*_\epsilon f_\epsilon = 0$ and (\ref{eq:apriorienergy}), we see that
\begin{align}
\|Z_{0,\epsilon} f_\epsilon\|_{\mathscr{X}^*} &  = \sup_{\varphi \in C_0^\infty, \|\varphi\|_{\mathscr{X}} \le 1} \epsilon \int \varphi \left(\sum_{j=1}^rZ_j^2 + \text{Tr}(A)\right)f_\epsilon  \leqc \epsilon \|f_\epsilon\|_{\mathscr{X}} \leqc \epsilon \|f_\epsilon\|_{L^2}. \label{eq:aprioriX*}
\end{align}
We summarize the a priori bounds discussed above in the following lemma. 
\begin{lemma}[Uniform-in-$\epsilon$ a priori estimates] \label{lem:AprioriBounds}
	Let $f\ge 0$ be a sufficiently smooth and well localized solution to $L_\epsilon^* f \ge 0$. Then, $f$ satisfies the energy estimate
	\begin{equation} \label{eq:apriorienergy2}
	\|f\|_{\mathscr{X}} \leqc \|f\|_{L^2}.
	\end{equation}
	If in addition $L_\epsilon^* f = 0$, then there exists $\gamma \ll 1$ (independent of $\epsilon$) such that	 
	\begin{align}
    \int e^{\gamma x^2}f & \leqc \|f\|_{L^1},  \label{eq:momentbound}\\ 
    \|Z_{0,\epsilon}f\|_{\mathscr{X}^*} &\leqc \|f\|_{L^2}. \label{eq:aprioridual}
	\end{align}
	All of the implicit constants above do not depend on $\epsilon$. 
\end{lemma}

\begin{remark} \label{rem:ellipticL2}
Notice that Lemma \ref{lem:AprioriBounds} does \emph{not} contain a uniform-in-$\epsilon$ a priori estimate on $\norm{f}_{L^2}$. When $L_\epsilon^*$ is elliptic; i.e., there exists $c > 0$ such that
\begin{align}
\sum_{j = 1}^r |Z_j \cdot \xi|^2 \ge c|\xi|^2 \quad \forall\xi \in \R^d,  \label{def:ellip}
\end{align}
	then (\ref{eq:apriorienergy2}) implies an a priori $L^2$ bound. Indeed, by the Gagliardo-Nirenberg inequality and $\int f_\epsilon = 1$ there exists $\theta \in (0,1)$ such that 
	$$ \|f_\eps\|_{L^2} \leqc \|f_\eps\|_{L^1}^{1-\theta}\|\grad f_\eps\|_{L^2}^\theta =  \|\grad f_\eps\|_{L^2}^\theta \leqc \|f_\eps\|_{\mathscr{X}}^\theta \leqc \|f_\eps\|_{L^2}^\theta, $$
	which immediately yields
	\begin{equation} \label{eq:aprioriL2}
	\sup_{\epsilon \in (0,1)}\|f_\epsilon\|_{L^2} \leqc 1.
	\end{equation}
However, in the hypoelliptic case, a more complicated argument is required to deduce \eqref{eq:aprioriL2} (see Section~\ref{sec:OutlineL2}).  
\end{remark}

\subsubsection{H\"{o}rmander's inequality and Moser iteration} \label{sec:OutlineMoser}
The proof of (\ref{eq:ubd}) is based on combining a local gain of integrability with a uniform-in-$\epsilon$ bound on the $L^2$ norm. In this section we discuss the gain of integrability estimate, and in Section~\ref{sec:OutlineL2} below we describe how to obtain uniform-in-$\epsilon$ $L^2$ control (see Remark \ref{rem:ellipticL2} above).  
\begin{lemma} \label{lem:Moser}
Let $\delta \in (0,1)$ and suppose that $f \in C^\infty(\R^d)$ satisfies $f \ge 0$ and
\begin{equation} \label{eq:deltainequation}
(\epsilon \delta \Delta + L_\epsilon^*) f \ge 0.
\end{equation}
Then, for any $R \ge 1$, uniformly in $\epsilon, \delta \in (0,1)$ there holds
$$\|f\|_{L^\infty(B_R)} \leqc_R \|f\|_{L^2(B_{2R})}.$$
\end{lemma}

\begin{remark} \label{rem:deltareg}
	The purpose of regularizing $L_\epsilon^*$ with $\epsilon \delta \Delta$ is to make formal computations easier to justify. This does not cause any difficulties in extracting information about $f_\epsilon$, since if $f_{\epsilon,\delta}$ denotes the unique solution to the problem
\begin{equation} \label{eq:regprob}
\begin{cases} 
(\epsilon \delta \Delta + L_\epsilon^*)f_{\epsilon,\delta} = 0, \\ 
f_{\epsilon,\delta} \ge 0, \\ 
\int f_{\epsilon,\delta} = 1,
\end{cases}
\end{equation}
then $\lim_{\delta \to 0}f_{\epsilon,\delta} = f_\epsilon$ in $H^k_{\text{loc}}$ for each fixed $\epsilon > 0$ and $k \in \N$ (see Lemma~\ref{lem:deltatozero}).
\end{remark}

The first step in the proof of Lemma~\ref{lem:Moser} is to derive localized versions of \eqref{eq:apriorienergy2} and \eqref{eq:aprioridual} that preserve the uniformity in $\epsilon$. Fix $0 < r_1 < r_2$ and let $\chi \in C_0^\infty(B_{r_2})$ be a radially symmetric cutoff with $\chi(x) = 1$ for all $|x| \le r_1$. The localized estimates take the form
\begin{align} 
\|\chi f\|_{\mathscr{X}} &\leqc \left(\frac{1+r_2}{r_2 - r_1}\right)\|f\|_{L^2(B_{r_2})},  \quad L_\epsilon^* f\ge 0 \label{eq:LocalizedX}, \\ 
\|Z_{0,\epsilon}(\chi f)\|_{\mathscr{X}^*} &\leqc \left(\frac{1+r_2}{r_2 - r_1}\right)\|f\|_{L^2(B_{r_2})}, \label{eq:LocalizedDual}  \quad L_\epsilon^* f = 0.
\end{align}
An important point is that we are forced to use a radially symmetric localization to preserve the structures that made our estimates uniform in $\epsilon$. Indeed, the energy structure assumed on the conservative drifts $B$ and $N$ is that they leave spherical shells invariant, so only for radially symmetric $\chi$ do we have $[\chi,Bx\cdot \grad] = [\chi,N\cdot \grad] = 0$.

In the elliptic setting (i.e.~\eqref{def:ellip}), Lemma~\ref{lem:Moser} follows from (\ref{eq:LocalizedX}) and the classical Moser iteration method (see e.g.~\cite{GT}). Recall that the main idea is that, in the elliptic case, the left-hand side of (\ref{eq:LocalizedX}) controls $\|\grad(\chi f)\|_{L^2}$, so by Sobolev embedding, $\exists \alpha > 1$ such that 
\begin{equation} \label{eq:intgain}
\|f\|_{L^{2\alpha}(B_{r_1})} \leqc \left(\frac{1+r_2}{r_2 - r_1}\right)\|f\|_{L^2(B_{r_2})}.
\end{equation}
By the convexity of $z \mapsto z^\beta$ for $\beta \ge 1$, if $f$ is a nonnegative subsolution, then $f^\beta$ is a subsolution to essentially the same equation (see \eqref{eq:MoserConvex}). The integrability gain (\ref{eq:intgain}) is then iterated with $f \to f^{\alpha^n}$, $n \in \N$ along a sequence of decreasing radii, ultimately yielding an $L^\infty$ bound. 

However, the iteration just described does not directly apply in the hypoelliptic setting because $\|\chi f\|_{\mathscr{X}}$ does not control $\|\chi f\|_{H^1}$.
Instead, one requires a Sobolev inequality that uses both (\ref{eq:LocalizedX}) and (\ref{eq:LocalizedDual}); i.e., defining
\begin{equation} \label{eq:defH1hyp}
\|f\|_{\Hhyp}:= \|f\|_{\mathscr{X}} + \|Z_{0,\epsilon} f\|_{\mathscr{X}^*}
\end{equation}
one wishes to prove the following.
\begin{lemma}[H\"{o}rmander inequality for $\Hhyp$]\label{lem:H1hypembed}
	There exists $s > 0$ such that for all $R \ge 1$ and $f \in C_0^\infty(B_R)$ there holds, uniformly in $\epsilon \in (0,1)$,
	\begin{equation} \label{eq:H1hypembed}
	\|f\|_{H^s} \leqc R^{1-s}\|f\|_{\Hhyp}. 
	\end{equation}
\end{lemma} 
The idea that (\ref{eq:defH1hyp}) is the natural norm for extending elliptic regularization to hypoelliptic operators with the general form of $L_\epsilon^*$ dates back to H\"{o}rmander's seminal paper on hypoellipticity \cite{H67} (see also discussions in \cite{AM19}). For $R=1$, Lemma~\ref{lem:H1hypembed} follows from a careful reading of \cite{H67} with the goal of quantifying [(3.4), \cite{H67}]. What we will actually need is a version of Lemma~\ref{lem:H1hypembed} that is uniform also in $\delta \in (0,1)$ and adapted to the a priori estimates provided by the regularized operator $\epsilon \delta \Delta + L_\epsilon^*$.
The $R$ dependence is deduced using a rescaling argument and the homogeneity of the drift term. For a precise statement and proof sketch of the H\"{o}rmander inequality that allows us to 
obtain the needed generalization of Lemma~\ref{lem:H1hypembed}, we refer to Lemmas~\ref{lem:HineqOrigQuant} and \ref{lem:HineqBall/Ann} in Section~\ref{sec:PrelimHineqStat}.

\begin{remark} 
The recent paper \cite{AM19} is the first to use the notation $\Hhyp$, wherein the authors develop a well-posedness theory in the complete space associated with a norm analogous to $\|\cdot\|_{\Hhyp}$ for the kinetic Fokker-Planck equation that mimics the classical $H^1$ variational theory for elliptic PDEs. Here, we mostly only need to be concerned with a priori estimates, but the terminologies remain quite natural nonetheless. 
\end{remark}

Using (\ref{eq:LocalizedX}), (\ref{eq:LocalizedDual}), and Lemma~\ref{lem:H1hypembed}, the proof of Lemma~\ref{lem:Moser} is similar to the classical Moser iteration scheme. A difficulty that prevents one from directly applying the iteration method as described after (\ref{eq:intgain}) is that (\ref{eq:LocalizedDual}) does not hold for subsolutions.
To remedy this, we adapt an argument from \cite{GIMV16}. Namely, since the ultimate goal is only to upgrade integrability (and not regularity), we instead estimate the exact solution to a closely related PDE, and then use a weak elliptic maximum principle; see Section \ref{sec:StationaryMoser} for details.

\subsubsection{Uniform in $\epsilon$ bound on $\|f_\eps\|$ in hypoelliptic case} \label{sec:OutlineL2}
As discussed in Remark~\ref{rem:ellipticL2}, Lemma~\ref{lem:AprioriBounds} does not directly imply the required uniform-in-$\epsilon$ estimate on $\|f_\epsilon\|_{L^2}$ to obtain (\ref{eq:ubd}) from Lemma~\ref{lem:Moser}. The estimate from Remark~\ref{rem:ellipticL2} would generalize to the hypoelliptic setting if we had a uniform estimate on any $\dot{H}^s$ seminorm in terms of $L^2$. However, this is not immediately possible (even formally) with Lemma~\ref{lem:H1hypembed} since the constant in (\ref{eq:H1hypembed}) scales like $R^{1-s}$. The natural thing we do here is to interpolate against the moment bound \eqref{eq:momentbound}. 
First, we write 
$$ f_\epsilon = \bar{\chi} f_\epsilon + \sum_{R = 2^j:j\ge 0}\chi_R f_\epsilon$$ 
for radially symmetric $\bar{\chi}, \chi_R$ with $\bar{\chi} \in C_0^\infty(B_1)$ and $\chi_R \in C_0^\infty(B_{2R}\setminus B_{R/2})$.
Next, we will use \eqref{eq:H1hypembed} on each annulus and interpolate (using the Gagliardo-Nirenberg inequality) against $L^1$, employing (\ref{eq:momentbound}) to absorb the loss that occurs when applying Lemma~\ref{lem:H1hypembed}. By (\ref{eq:LocalizedX}) and (\ref{eq:LocalizedDual}) there holds
\begin{equation} \label{eq:H1hypByL2LocBound}
\|\chi_R f_\epsilon\|_{\Hhyp} \leqc \|f_\epsilon\|_{L^2},
\end{equation}
so from (\ref{eq:H1hypembed}) and the Gagliardo-Nirenberg inequality, $\exists \theta \in (0,1)$ such that 
\begin{align*} 
 \sum_{R = 2^j: j \ge 0} \|\chi_R f_\epsilon\|_{L^2} & \leqc \|f_\epsilon\|_{L^2}^\theta \sum_{R = 2^j: j \ge 0} R^\theta\|\chi_R f_\epsilon\|_{L^1}^{1-\theta} \\ 
& \leqc \|f_\epsilon\|_{L^2}^\theta \left(1+\sum_{R = 2^j: j \ge 0}\left\|R^\frac{2\theta}{1-\theta}\chi_R f_\epsilon \right\|_{L^1}\right) \\ 
&   \leqc \|f_\epsilon\|_{L^2}^\theta \left(1 + \int e^{\gamma x^2}f_\epsilon \right) \leqc \|f_\epsilon\|_{L^2}^\theta.
\end{align*}
We can estimate $\|\bar{\chi} f_\epsilon\|_{L^2}$ in a similar fashion. These computations form the basis of the following lemma, which we prove in detail in Section~\ref{sec:StationaryL2}.
\begin{lemma}\label{lem:L2bound}
	With the notation of (\ref{eq:regprob}), there holds
	$$\sup_{\epsilon,\delta \in (0,1)} \|f_{\epsilon,\delta}\|_{L^2} \leqc 1.$$
	As a consequence,
	\begin{equation}\label{eq:H1hypuniform}
	\sup_{\epsilon \in (0,1)}\|f_\epsilon\|_{\Hhyp} \leqc \sup_{\epsilon \in (0,1)}\|f_\epsilon\|_{L^2} \leqc 1.
	\end{equation}
\end{lemma}

\subsection{Time-stationary problem: lower bounds} \label{sec:Outlinelower}

In this section we discuss the key ideas that go into proving the lower bound (\ref{eq:lbd}).
The first observation is that (\ref{eq:ubd}) and (\ref{eq:momentbound}) together imply that there exist constants $c_1, c_2 > 0$ such that for every $R \gtrsim 1$ there holds, uniformly in $\epsilon \in (0,1)$, 
\begin{equation} \label{eq:Outlinelbd1}
|\{x \in B_{R}: f_\epsilon/c_1 \ge 1\}| \ge c_2.
\end{equation}
In other words, $f_\epsilon$ stays uniformly bounded away from zero on a set of positive measure. A classical idea in the H\"{o}lder regularity theory for second-order elliptic equations with rough coefficients is that weak solutions ``cannot oscillate too much'' in the sense that if $-1 \le u \le 1$ solves $Lu = 0$ on $B_2$ for a suitable elliptic operator $L$, then $|\{x \in B_1: u(x) \le 0\}| > 0$ implies that $u$ must remain uniformly bounded away from $1$ \textit{everywhere} on the smaller set $B_{1/2}$; see for example the review \cite{Vasseur2016} (in particular, Proposition 9) and the references therein. For nonnegative solutions and $L_\epsilon^*$ elliptic, (\ref{eq:Outlinelbd1}) would then imply the estimate
\begin{equation} \label{eq:Outlinelbd2}
\inf_{x \in B_R}f_\epsilon(x) \gtrsim_{R,\epsilon, c_1,c_2} 1. 
\end{equation}
We have indicated that the implicit constant a priori depends on $\epsilon$, although we will see below that this is not the case. 

Our strategy is to extend the argument that yields (\ref{eq:Outlinelbd2}) from (\ref{eq:Outlinelbd1}) (with a constant independent of $\epsilon$) to the hypoelliptic setting.
We begin by recalling a clever trick from de Giorgi's approach to H\"{o}lder regularity for elliptic PDEs with rough coefficients (see e.g. \cite{Vasseur2016}).
For $\theta \in (0,1)$ consider the rescaled functions
\begin{equation}
w_k = \left(1 - \theta^{-k} (f_\epsilon/c_1)\right)_+, \end{equation}
structured so that 
\begin{align}
|\{x \in B_R:w_k = 0\}| & \ge c_2, \label{eq:Outlinelbd3}\\ 
|\{x \in B_R: w_k \ge 1-\theta\}| &\ge \int_{B_R} |w_{k+1}|^2,\label{eq:Outlinelbd4}
\end{align}
and observe that a uniform lower bound on $f_\epsilon$ follows provided that $\liminf_{k \to \infty}\|w_k\|_{L^2(B_{R})} = 0$ uniformly in $\epsilon \in (0,1)$. Indeed, $L_\epsilon^* w_k \ge 0$ by the convexity of $z \mapsto z_+$,  and hence Lemma~\ref{lem:Moser} (suppose for the sake of discussion that it is true for $f \in H^1$ and $\delta = 0$) gives
$$ \inf_{x \in B_{R/2}}f_\epsilon(x) \ge c_1 \theta^k\left(1-\|w_k\|_{L^\infty(B_{R/2})}\right) \ge c_1 \theta^k\left(1 - C(R)\|w_k\|_{L^2(B_{R})}\right)\quad \forall k\in \N.$$

The proof that $\|w_k\|_{L^2}$ eventually gets small uses an iteration argument hinging on an isoperimetric inequality, which in the elliptic case controls the amount of possible oscillation of nonnegative subsolutions $L^*_\epsilon f \geq 0$ in terms of the natural $H^1$ energy norm.  
The essential point is that if $\|w_{k+1}\|^2_{L^2(B_R)} \ge \kappa(R) > 0$, then due to (\ref{eq:Outlinelbd4}) there holds
\begin{equation} \label{eq:Outlinelbd5}
|\{x \in B_R: w_k \ge 1-\theta\}| \ge \kappa \quad \& \quad |\{x \in B_R: w_k = 0\}| \ge c_2, 
\end{equation}
and so since the collection $\{\{0 < w_k < 1-\theta\}\}_{k=1}^\infty$ is pairwise disjoint, $\|w_{k_*}\|_{L^2} \ge \kappa$ must fail for some large enough $k_* = k_*(\kappa,c_1,c_2,R)$ provided that (\ref{eq:Outlinelbd5}) and $L_\epsilon^* w_k \ge 0$ together imply 
\begin{equation} \label{eq:Outlinelbd6}
|\{x\in B_R: 0 < w_k < 1-\theta\}| \gtrsim_{\kappa,c_1,c_2,R} 1;
\end{equation}
see the proof of Lemma~\ref{lem:wkiteration} for more explanation. It is important to note that proving (\ref{eq:lbd}) requires the constant in (\ref{eq:Outlinelbd6}) to be independent of $\epsilon$.  

In the elliptic setting, (\ref{eq:Outlinelbd6}) is provided by (\ref{eq:LocalizedX}) and the classical De Giorgi isoperimetric inequality (see e.g. [Lemma 10, \cite{Vasseur2016}]), which explicitly quantifies a lower bound for $|\{x\in B_R: 0 < w_k < 1-\theta\}|$ in terms of $\|w_k\|_{H^1(B_R)}$ and the quantities in (\ref{eq:Outlinelbd5}). This approach does not apply in the hypoelliptic setting because $L_\epsilon^* w_k \ge 0$ is not sufficient to provide a uniform-in-$\epsilon$ bound on $\|w_k\|_{H^1(B_R)}$ (see also discussions in \cite{GIMV16}). Nevertheless, for each fixed $R \ge 1$ we are able to prove an isoperimetric inequality that holds uniformly in $0 < \epsilon \ll 1$. It is stated as follows (recall the notation from (\ref{eq:regprob})).

\begin{lemma}[An intermediate value lemma] \label{lem:IVT}
	Fix $R \ge 1$ and $\alpha_1$, $\alpha_2 > 0$. There exists $\epsilon_0 > 0$, $\mu > 0$, and $\theta \in (0,1/2)$ such that if $\epsilon \le \epsilon_0$, $\delta \in (0,1)$, and $w \in C^\infty(B_{2R})$ with $0 \le w \le 1$ satisfies
	 \begin{equation} \label{eq:IVTassumption}
	 0 \le \delta \Delta w + \frac{1}{\epsilon}L_\epsilon^* w  \le \frac{1}{\sqrt{\epsilon}}\left(1+ \delta|\grad f_{\epsilon,\delta}|^2 +\sum_{j=1}^r |Z_j f_{\epsilon,\delta}|^2\right)
	 \end{equation} 
	  on $B_{2R}$, then the inequalities 
	$$|\{w  = 0\} \cap B_R| \ge \alpha_1$$
	and
	$$|\{w  \ge 1-\theta\} \cap B_R| \ge \alpha_2$$
	together imply
	$$|\{0<w <1-\theta\} \cap B_R| \ge \mu.$$
\end{lemma}

The proof of Lemma~\ref{lem:IVT} follows a compactness-rigidity argument motivated by [Lemma 14, \cite{GIMV16}]. The desired compactness is deduced with a uniform H\"{o}rmander inequality. However, we cannot directly apply Lemma~\ref{lem:H1hypembed} to obtain $\|w\|_{H^s} \leqc 1$ because the upper bound in (\ref{eq:IVTassumption}) is too weak to provide a uniform estimate on $\|Z_{0,\epsilon}w\|_{\mathscr{X}^*}$. Instead, for radially symmetric $\chi \in C_0^\infty(B_{2R})$ with $\chi(x) = 1$ for $|x| \le R$ and a test function $\varphi \in C_0^\infty$, the natural $\epsilon$-independent estimate is (supposing $\delta = 0$ for simplicity)
\begin{align*}
\left|\int \varphi Z_{0,\epsilon} (\chi w) \right| &\leqc \epsilon \|\varphi\|_{\mathscr{X}} \|w\|_{L^2(B_{2R})} + \sqrt{\epsilon}\|\varphi\|_{L^\infty}\left(1 + \|f_\epsilon\|_{\Hhyp}^2 \right)  \leqc \sqrt{\epsilon} \left(\|\varphi\|_{L^\infty} + \|\varphi\|_{\mathscr{X}}\right),
\end{align*}
where in the second inequality we used (\ref{eq:H1hypuniform}) to control $\|f_\epsilon\|_{\Hhyp}$ and $0\le w \le 1$ to control $\|w\|_{L^2(B_{2R})}$. Since $w$ is a priori bounded, this estimate suggests that we obtain the needed compactness with a H\"{o}rmander inequality that includes $L^\infty$ in the norm pair. We derive such an inequality in Lemma~\ref{lem:HineqOrigQuant} by making some small modifications to the original arguments in \cite{H67}.

The rigidity step is resolved by passing to the limit and deriving a contradiction with the supposed counter-example obtained at $\theta = \epsilon = \mu = 0$ satisfying the $\epsilon$-independent estimates provided by (\ref{eq:IVTassumption}). This only requires one to know that there cannot exist a non-constant characteristic function $\xi$ satisfying (in the sense of distributions) $N \cdot \grad \xi = 0$ and $Z_j  \cdot \grad \xi = 0$ for $j=1,\ldots,r$. While this can be achieved directly from H\"{o}rmander's theorem, we have the following stronger rigidity statement, which plays a key role in the proof of our hypoelliptic weak Poincar\'{e} inequality (see Lemmas~\ref{lem:poincinterp} and~\ref{lem:weakpoinc}).
\begin{lemma}[Rigidity lemma]\label{lem:rigidity}
	Let $\Omega \subset \R^d$ be an open, connected set. Suppose that $\{X_j\}_{j=0}^k \subseteq T(\Omega)$ satisfies H\"{o}rmander's condition on $\Omega$. Let $f \in L^2_{loc}(\R^d)$ be a distributional solution to
	$$X_0  f = 0.$$
	If $X_j f = 0$ for each $j = 1, \ldots, k$, then $f$ is constant on $\Omega$. 
\end{lemma}

\begin{proof}
	The assumptions of the lemma imply that $f \in L^2_{\text{loc}}(\Omega)$ is a distributional solution to 
	$\sum_{j=1}^k X_j^2 f + X_0 f = 0,$
	and so by H\"{o}rmander's theorem [Theorem 1.1, \cite{H67}] it follows that $f \in C^\infty(\Omega)$. Thus, it suffices to prove the result for smooth functions satisfying $X_j f = 0$ for $j = 0,1,\ldots, k$ in the classical sense on $\Omega$.
	
	Let $B(x,r)$ denote the ball of radius $r$ centered at $x$ with the usual Euclidean metric, and for $c \in \R$, let
	$$S_c = \{x \in \Omega: \exists r > 0 \text{ such that }f \equiv c\text{ on } B(x,r)\}.$$
	We will prove that there is some $c$ so that $S_c$ is open, nonempty, and relatively closed in $\Omega$. The fact that $S_c$ is open for each $c$ follows by its definition. To prove that $\exists c$ such that $S_c$ is both relatively closed and nonempty, it suffices to show that for every $x \in \Omega$ $\exists r_x > 0$ so that $f$ is constant on $B(x,r_x)$. Fix $x_0 \in \Omega$ and let $\mathcal{U} \subset \Omega$ be an open ball containing $x_0$. Define the \textit{$\mathcal{U}$-reachable set at $x_0$} to be the points $x_1 \in \mathcal{U}$ such that there exist bounded, measurable functions $\{c_j:[0,1] \to \R\}_{j=0}^k$ and a curve $\gamma: [0,1] \to \mathcal{U}$ such that $\gamma(0) = x_0$, $\gamma(1) = x_1$, and 
	\begin{equation} \label{eq:gammacurve}
	\gamma'(t) = \sum_{j=0}^k c_j(t)X_j(\gamma(t)) \quad \text{a.e. } t \in [0,1].
	\end{equation}
	A well-known fact in the theory of local controllability is that the $\mathcal{U}$-reachable set at $x_0$ is an open neighborhood of $x_0$ as soon as $\{X_j\}_{j=0}^k$ satisfies H\"{o}rmander's condition on $\mathcal{U}$; see e.g. [Theorem 2.2, \cite{Hermann1977}]. Since $f$ is constant along any curve $\gamma$ satisfying \eqref{eq:gammacurve} we conclude that it must be constant on some open ball containing $x_0$, completing the proof. 	
\end{proof}

We have already sketched the main ideas in using a uniform-in-$\epsilon$ intermediate value lemma to obtain a local lower bound, though due to the complexity of Lemma~\ref{lem:IVT} there are some additional details to fill in. This is done in Section~\ref{sec:StatIVT}, wherein we prove the following. 
\begin{lemma} \label{lem:IVTtolbd}
	Suppose that Lemma~\ref{lem:IVT} holds. Then, for all $R \ge 1$ there exists $\epsilon_*(R) > 0$ such that 
	\begin{equation} \label{eq:IVTtolbd}
	\inf_{\epsilon \in (0,\epsilon_*)}\inf_{|x|\le R} f_\epsilon(x) \gtrsim_R 1.
	\end{equation}
	If Assumption~\ref{Assumption:positivity} is satisfied, then \eqref{eq:lbd} also holds.
\end{lemma}
To obtain \eqref{eq:lbd} from \eqref{eq:IVTtolbd}, we use that since $\epsilon_*$ in Lemma~\ref{lem:IVTtolbd} depends only on $R$ it suffices to show
\begin{equation}
\inf_{\epsilon \in [\epsilon_*, 1)} \inf_{|x| \le R}f_\epsilon(x) \gtrsim_{R,\epsilon_*} 1.
\end{equation}
Such a bound follows in a straightforward way provided that Assumption~\ref{Assumption:positivity} is satisfied (see Lemma~\ref{lem:largeps}).

\subsection{Time-dependent problem} \label{sec:OutlineTD}

In this section we prove Theorem~\ref{thrm:GeoErg}, with the needed intermediate results stated as lemmas to be proven in Section \ref{sec:GeoErg}. 

We quantify the spectral gap for $\Pt_t^\epsilon$ using a Harris theorem that follows easily from the techniques in \cite{HMHarris}. For a Polish space $\mathcal{X}$ and any bounded, Borel measurable function $f:\mathcal{X} \to \R$ we write 
\begin{equation}
\|f\|_{\infty} := \sup_{x \in \mathcal{X}}|f(x)|.
\end{equation}
Recall also the notations $\|\cdot\|_{V}$ and $\|\cdot\|_{TV,V}$ defined in (\ref{eq:defCV}) and (\ref{eq:defTVV}), respectively. The Harris theorem then reads as follows.

\begin{lemma}[Harris]\label{thrm:Harris}
	Let $\Pt$ be a Markov semigroup on a Polish space $\mathcal{X}$. Suppose that there exists a continuous function $V:\mathcal{X} \to [0,\infty)$ such that: 
	\begin{itemize}
		\item There exists $\bar{b} \ge 1$ such that for all $x \in \mathcal{X}$ there holds 
		\begin{equation} \label{eq:Harrisdrift}
		\Pt V(x) \le \frac{1}{2}V(x) + \bar{b}.
		\end{equation}
		\item There exists $\eta \in (0,2)$ such that if $x,y \in \mathcal{X}$ satisfy $V(x) + V(y) \le 10\bar{b}$ and $f:\mathcal{X} \to \R$ is a bounded, Borel measurable function with $\|f\|_\infty \le 1$, then
		\begin{equation}  \label{eq:irreducible0}
	    |\Pt f(x) - \Pt f(y)| \le 2-\eta.
		\end{equation}
	\end{itemize}
	Then, there exists $c \in (0,1)$ satisfying $c \gtrsim \eta$ and $K \ge 1$ depending only on $\bar{b}$ such that for any two measures $\mu, \nu \in \mathcal{M}(\mathcal{X})$ with 
	$\int_\mathcal{X} V(x) \mu(dx) + \int_\mathcal{X} V(x)\nu(dx) < \infty$
	and any $n \in \N$ there holds 
	$$ \|(\mathcal{P}^*)^{n}(\mu - \nu)\|_{TV,V} \le K (1-c)^n \|\mu - \nu\|_{TV,V}. $$
	As a consequence, $\mathcal{P}$ can admit only one stationary measure, and if $\mu_\infty$ is the unique stationary measure there exists $\bar{K}\ge 1$ depending only on $\bar{b}$ so that for any $n \in \N$ and measurable function $f: \R^d \to \R$ with $\|f\|_{V} < \infty$ there holds 
	$$ \|\Pt^{n} f - \mu_\infty(f)\|_{V} \le \bar{K}(1-c)^n \|f - \mu_\infty(f)\|_{V}. $$
\end{lemma}

Let $V$ be a uniform Lyapunov function for $\Pt_t^\epsilon$ (see Definition~\ref{def:uniformLyapunov}). By Lemma~\ref{thrm:Harris}, to prove (\ref{eq:GeoErg}) for discrete times $t = nt_*$ it is enough to show that there exists $C_* > 0$ so that if $t_* = C_*\epsilon^{-1}$ then (\ref{eq:Harrisdrift}) and (\ref{eq:irreducible0}) hold with $\Pt = \Pt_{t_*}^\epsilon$ and constants $\bar{b}$, $\eta$ that do not depend on $\epsilon$. For the drift condition, directly from \eqref{eq:AppMoment} we have
\begin{equation} \label{eq:drift}
\Pt_t^\epsilon V(x) = e^{-\kappa \epsilon t}V(x) + \frac{b}{\kappa} : = e^{-\kappa \epsilon t}V(x) + \bar{b},
\end{equation}
and so $\Pt_{t_*}^\epsilon$ satisfies (\ref{eq:Harrisdrift}) uniformly in $\epsilon$ for $C_*$ sufficiently large. As alluded to in Section~\ref{sec:Intro}, establishing (\ref{eq:irreducible0}) is much more difficult. First, notice that by duality and the fact that $V$ and $\bar{b}$ do not depend on $\epsilon$, it would suffice to show that for every $R \ge 1$ there exists $\eta(R) \in (0,2)$ such that whenever $|x|, |y| \le R$ and $C_*$ is sufficiently large depending only on $R$ there holds
\begin{equation} \label{eq:irreducible}
\|(\Pt^\epsilon_{t_*})^*\delta_x - (\Pt^\epsilon_{t_*})^*\delta_y\|_{TV} \le 2 - \eta.
\end{equation}
If $C_*$ could be chosen depending on $\epsilon$, then using the dissipative structure of \eqref{eq:SDE}, the regularizing properties of $(\Pt_t^{\epsilon})^*$, and the fact that there is a nonzero probability that the driving Wiener process remains small over $[0,t_*]$, one can prove that \eqref{eq:irreducible} is satisfied for some $\eta > 0$ that depends badly on $\epsilon$. What makes Theorem~\ref{thrm:GeoErg} and its proof novel is that we are forced to show that $\eta$ and $C_*$ can actually be chosen independently of $\epsilon$. To this end we will prove that
\begin{equation} \label{eq:irreducible2}
\lim_{t\to \infty}\sup_{\epsilon \in (0,1),\|f\|_{\infty}\le 1}\|\Pt_{t \epsilon^{-1}}^\epsilon f - \mu_\epsilon(f)\|_{L^\infty(B_R)} = 0. 
\end{equation}

Our proof of (\ref{eq:irreducible2}) is based on a two step procedure that makes crucial use of Theorem~\ref{thrm:stationary}. First, we show that $\Pt_t^\epsilon$ satisfies the quantitative $L^2_{\mu_\epsilon} \to L^\infty$ ``parabolic'' regularization estimate stated in Lemma~\ref{lem:ParReg}. For fixed $\epsilon$, similar regularization estimates can be obtained by several known methods, hence the main challenge in proving Lemma~\ref{lem:ParReg} is obtaining uniformity in $\epsilon \in (0,1)$.
As mentioned earlier, this requires us to derive a slightly subtle space-time H\"{o}rmander inequality better adapted to a parabolic framework (Lemma~\ref{lem:Horparabolic}). With such an inequality in hand, one can apply a suitably adapted Moser iteration.
However, the Moser iteration controls the $L^\infty(B_R)$ norm in terms of the time integral of the $L^2(B_{2R})$ norm. In order to estimate this in terms of the initial condition, we crucially apply the uniform lower bounds in Theorem \ref{thrm:stationary} and the monotonicity of $\Pt_t^\epsilon$ with respect to $L^2_{\mu_\epsilon}$; see Section \ref{sec:ParRegul} for details.    

Given Lemma~\ref{lem:ParReg}, the proof of (\ref{eq:irreducible2}) reduces to showing that
\begin{equation} \label{eq:Outlineweakpoinc1}
\lim_{t \to \infty} \sup_{\epsilon \in (0,1), \|f\|_{\infty} \le 1}\|\Pt_{t \epsilon^{-1}}^\epsilon f - \mu_\epsilon(f)\|_{L^2_{\mu_\epsilon}} = 0.
\end{equation}
In other words, we need to prove that $\Pt_t^\epsilon$ satisfies a $\|\cdot\|_\infty \to \|\cdot\|_{L^2_{\mu_\epsilon}}$  decay estimate with timescale $\epsilon^{-1}$ on functions which are mean zero with respect to $\mu_\epsilon$. This task is much more tractable than proving, say, an $L^2_{\mu_\epsilon} \to L^2_{\mu_\epsilon}$ exponential decay estimate for $\Pt_t^\epsilon$ (which would require a Poincar\'{e} inequality for $\mu_\epsilon$; see e.g. \cite{Bak08}), and is provided by the following lemma.  

\begin{lemma}\label{lem:Linfty2L2decay}
Under Assumption~\ref{Assumption:positivity}, there exists a function $\psi:[0,\infty) \to (0,\infty)$ with $\lim_{t\to \infty} \psi(t) = 0$ such that for every bounded, Borel measurable function $f:\R^d \to \R$ and $\epsilon \in (0,1)$ there holds
	\begin{equation} \label{eq:Linfty2L2decay}
	\|\Pt_t f-\mu_\epsilon(f)\|_{L^2_{\mu_\epsilon}}^2 \le \psi(\epsilon t)\|f-\mu_\epsilon(f)\|_{L^\infty}^2.
	\end{equation}
\end{lemma}

From essentially the ODE computation in [Theorem 2.1, \cite{RockWang01}], we can prove Lemma~\ref{lem:Linfty2L2decay} provided we can show that $\Pt_t^\epsilon$ satisfies a type of uniform-in-$\epsilon$, hypoelliptic weak Poincar\'{e} inequality. This is accomplished in Lemmas~\ref{lem:poincinterp} and \ref{lem:weakpoinc} using a compactness-rigidity argument that relies on Lemma~\ref{lem:rigidity}.

We now prove Theorem~\ref{thrm:GeoErg} assuming Lemmas~\ref{lem:ParReg} and~\ref{lem:Linfty2L2decay}.
\begin{proof}[Proof of Theorem~\ref{thrm:GeoErg}]
	We first prove the result for discrete times. That is, we show that there exists $K,\delta, C_* > 0$ that do not depend on $\epsilon$ so that if $t_* = C_*\epsilon^{-1}$ then for every $n \in \N$, $\epsilon \in (0,1)$, and measurable $f:\R^d \to \R$ with $\|f\|_{V} < \infty$ there holds
	\begin{equation} \label{eq:GeoErgDisc}
	\|\Pt_{nt_*}^\epsilon f - \mu_\epsilon(f)\|_{V} \le Ke^{-\delta \epsilon nt_*}\|f-\mu_\epsilon(f)\|_{V}.
	\end{equation} 
	From Lemma~\ref{thrm:Harris} and the discussion proceeding it we just need to prove (\ref{eq:irreducible2}). Fix $R \ge 1$. By Lemma~\ref{lem:ParReg} and the monotonicity of $\Pt_t^\epsilon$ with respect to $L^2_{\mu_\epsilon}$ we have, for any $t \ge 2$, 
	$$ 
	\sup_{\epsilon \in (0,1),\|f\|_{\infty}\le 1}\|\Pt_{t \epsilon^{-1}}^\epsilon f - \mu_\epsilon(f)\|_{L^\infty(B_R)} \leqc_R \sup_{\epsilon \in (0,1), \|f\|_{\infty}\le 1}\|\Pt^\epsilon_{\frac{t}{2}\epsilon^{-1}}f-\mu_\epsilon(f)\|_{L^2_{\mu_\epsilon}}.
	$$
Hence, applying Lemma~\ref{lem:Linfty2L2decay} we obtain
\begin{align*}
	\limsup_{t\to \infty} \sup_{\epsilon \in (0,1), \|f\|_{\infty}\le 1} \|\Pt^\epsilon_{t\epsilon^{-1}}f-\mu_\epsilon(f)\|_{L^\infty(B_R)} 	& \leqc \limsup_{t\to \infty}\sup_{\epsilon \in (0,1), \|f\|_{\infty}\le 1}\sqrt{\psi(t/2)}\|f - \mu_\epsilon(f)\|_{L^\infty} \\
    & \leqc \limsup_{t\to\infty}\sqrt{\psi(t/2)} = 0,
	\end{align*}
	as desired.	
	
	It only remains to upgrade (\ref{eq:GeoErgDisc}) to continuous time. Let $t \ge 0$ and choose $n \in \N$ so that $t \in [nt_*, (n+1)t_*)$. Using the semigroup property, there exists $s \in [0,t_*)$ so that 
	$$\|\Pt_t^\epsilon f - \mu_\epsilon(f)\|_{V} = \|\Pt_s^\epsilon(\Pt_{nt_*}^\epsilon f - \mu_\epsilon(f))\|_{V} \le Ke^{-\delta \epsilon n t_*}\|\Pt_s^{\epsilon}\|_{V \to V}\|f- \mu_\epsilon(f)\|_{V}. $$
	Whenever $n \ge 1$ one has 
	$ \delta n \epsilon t_* \ge \epsilon (\delta/2) t$, and so (\ref{eq:GeoErg}) follows provided
	\begin{equation} \label{eq:PtVbound}
	\sup_{0 \le s \le C_* \epsilon^{-1}}\|\Pt_s^\epsilon\|_{V \to V} \leqc_{C_*} 1.
	\end{equation}
	This is proven at the beginning of Section~\ref{sec:GeoErg}; see Lemma~\ref{lem:PtVbound}.
\end{proof}

\section{Uniform H\"{o}rmander inequalities} \label{sec:Prelim}
\label{sec:PrelimHineq}

In this section we state and sketch proofs of various uniform H\"{o}rmander type inequalities. Throughout the entire section, $\Omega \subset \R^d$ denotes an open, bounded set and $K \subset \Omega$ is compact. The proof techniques in this section are not used elsewhere in the paper, and the reader interested only in statements of the H\"{o}rmander inequalities to be used can safely skip to Section~\ref{sec:PrelimHineqStat}. 

\subsection{Notation and basic facts} \label{sec:HineqDef}
We begin with some notation and basic facts that will be needed in the proof sketches to follow. 

We use the notation in \cite{H67} for the $L^2$-based H\"{o}lder regularity of a function $u$ along a vector field $X \in T(\Omega)$. For any $0 < t_0 \ll 1$ and $s \in (0,1]$ we write 
\begin{equation} \label{eq:DefFiniteDiff1}
|u|^{t_0}_{X,s} = \sup_{|t| \le t_0}|t|^{-s}\|e^{tX}u - u\|_{L^2}, \quad u \in C_0^\infty(K).
\end{equation} 
This is well defined since $e^{tX}$ maps $C_0^\infty(K)$ into $C_0^\infty(\Omega)$ provided that $|t|$ is sufficiently small depending only on $K$, $\Omega$, and the derivatives of $X$. We also define an isotropic $s$-norm by
\begin{equation}  \label{eq:DefFiniteDiff2}
|u|_s^{t_0} = \sup_{|h| \le t_0}|h|^{-s}\|u(\cdot + h) - u(\cdot)\|_{L^2}, \quad u \in C_0^\infty(K).
\end{equation}
In Section~\ref{sec:HineqTime} we will need to consider differential operators of the form $\epsilon \partial_t + X$ for $\epsilon > 0$ and functions that depend on time. In this situation, we write 
\begin{equation} \label{eq:DefFiniteDiff3}
|u|_{\epsilon\partial_t + X,s}^{t_0} = \sup_{|\tau| \le t_0}|\tau|^{-s}\|e^{\tau X}u(\cdot + \epsilon \tau,\cdot) - u(\cdot,\cdot)\|_{L^2}, \quad u \in C_0^\infty((a,b)\times K).
\end{equation}

The seminorm $|\cdot|^{t_0}_{s}$ is related to the usual homogeneous Sobolev spaces by the equivalence
\begin{equation} \label{eq:Besov}
\|u\|_{\dot{B}_{p,r}^s} \approx_{s,p,r} \left\|\frac{\|u(\cdot - y) - u(\cdot)\|_{L^p}}{|y|^s}\right\|_{L^r(\R^d;|y|^{-d}dy)}, 
\end{equation}
which holds for any $s \in (0,1)$ and $(p,r) \in [1,\infty]^2$; see e.g. [Theorem 2.36, \cite{BCD11}]. Here, $\dot{B}_{p,r}^s$ denotes the usual homogeneous Besov space. We refer to \cite{BCD11} for definitions and basic results. A straightforward consequence of (\ref{eq:Besov}) and $\|u\|_{\dot{B}^s_{2,\infty}} \le \|u\|_{\dot{B}^s_{2,2}} \approx \|u\|_{\dot{H}^s}$ is that for any $s \in (0,1)$ and $s' > s$ there holds
\begin{equation} \label{eq:finitediff2}
|u|_{s}^{t_0} \leqc_{s} \|u\|_{H^{s}} \leqc_{s'-s} C(t_0,s) \left(\|u\|_{L^2} + |u|^{t_0}_{s'} \right), \quad u \in C_0^\infty(K),
\end{equation}
where $C$ is nonincreasing in $|t_0|$.

As usual, define $\text{ad}X (Y): = [X,Y]$ for $X, Y \in T(\Omega)$. Then, for $\{(X_j,s_j)\}_{j=0}^r \subseteq T(\Omega) \times (0,1]$ and a multi-index $I = (i_1,\ldots,i_k)$, $0 \le i_j \le r$ we write 
\begin{equation} \label{eq:indexnotation}
\begin{aligned} 
X_I &= \text{ad}X_{i_k} \text{ad}X_{i_{k-1}}\ldots \text{ad}X_{i_2} X_{i_1}, \\ 
\frac{1}{s(I)} & = \sum_{j=1}^k \frac{1}{s_{i_j}}, \quad m(I) = \frac{1}{s(I)}.
\end{aligned}
\end{equation}

\subsection{Time-independent H\"{o}rmander inequalities} \label{sec:PrelimHineqStat}

We begin by defining the regularized H\"{o}rmander norm pairs natural for studying the operator $\epsilon \delta \Delta + L_\epsilon^*$. For $\{X_j\}_{j=1}^r\subseteq T(\Omega)$, an open set $\Omega' \subseteq \Omega$, and $\delta \in [0,1]$ we define 
\begin{equation} \label{eq:defXdelta}
\begin{aligned} 
\|g\|_{\mathscr{X}_\delta(\Omega')}& : = \|g\|_{L^2(\Omega')} + \sum_{j=1}^r \|X_j g\|_{L^2(\Omega')} + \sqrt{\delta}\|\grad g\|_{L^2(\Omega')}, & \\ \|g\|_{\mathscr{X_\delta^*}(\Omega')} &: = \sup_{\varphi \in C_0^\infty(\Omega'), \|\varphi\|_{\mathscr{X_\delta}(\Omega')}\le 1}\int_{\Omega'} \varphi g;
\end{aligned}
\end{equation}
\begin{equation}\label{eq:defXdeltainf}
\begin{aligned} 
 \|g\|_{\mathscr{\tilde{X}_\delta}(\Omega')}& : = \|g\|_{\mathscr{X_\delta}(\Omega')} + \|g\|_{L^\infty(\Omega')}, \\
 \|g\|_{\mathscr{\tilde{X}_\delta^*}(\Omega')}&:=\sup_{\varphi \in C_0^\infty(\Omega'), \|\varphi\|_{\mathscr{\tilde{X}_\delta}(\Omega')}\le 1}\int_{\Omega'} \varphi g.
\end{aligned}
\end{equation}
Typically, the functions $g$ we consider have compact support in $\Omega$ and $\Omega' = \Omega$. In this case, we do not indicate any domain in the notation. Also, by an abuse of notation, throughout the paper we will write ($\mathscr{X}_\delta$, $\mathscr{X_\delta^*}$), ($\mathscr{X}$, $\mathscr{X}^*$), etc., regardless of whether the vector fields involved are a general collection $\{X_j\}_{j=1}^r \subseteq T(\Omega)$ or the specific vector fields $\{Z_j\}_{j=1}^r \subseteq T(\R^d)$ from \eqref{eq:L*} since the meaning will always be clear from context. 

The following lemma is a generalized and quantitative version of Theorem~\ref{lem:H1hypembed}. It holds uniformly in the regularization parameter $\delta$ and is indifferent to whether or not $L^\infty$ is included in the H\"{o}rmander norm. The proof we give is a straightforward adaptation of the techniques from \cite{H67}. Recall the terminology from Definition~\ref{def:uniformHor}. 
\begin{lemma}[Quantitative H\"{o}rmander inequality] \label{lem:HineqOrigQuant}
	Let $\Omega \subset \R^d$ be an open, bounded set, $K \subset \Omega$ be compact, and $\bar{\mathscr{X}}$ be either $\mathscr{X_\delta}$ or $\mathscr{\tilde{X}_\delta}$. Suppose that $\{X_j\}_{j=0}^r \subseteq T(\Omega)$ satisfies the uniform H\"{o}rmander condition on $\Omega$ with constants $(N_0,C_0) \in \N \times (0,\infty)$. There exists $s(N_0) > 0$ and a constant $C$ such that for all $u \in C_0^\infty(K)$ and $\delta \in [0,1]$ there holds 
	$$ \|u\|_{H^s} \le C\left(\|u\|_{\bar{\mathscr{X}}} + \|X_0 u\|_{\bar{\mathscr{X}^*}}\right).$$ 
    The constant $C$ depends on $\{X_j\}_{j=0}^r$ only through $r$, $N_0$, $C_0$, and an upper bound on
    $ \sum_{j=0}^r\|X_j\|_{C^k(\Omega)}$ for some $k(N_0) > 0$ sufficiently large.
\end{lemma}

\begin{proof}[Proof sketch]
	Recall the definition of the norm $\mathscr{X}$ from Section~\ref{sec:OutlineMoser}. For $\delta_1, \delta_2 \in [0,1]$, and functions $g \in C_0^\infty(\Omega)$ we define the H\"{o}rmander norm pair 
	\begin{align}
	\|g\|_{\mathscr{X}_{\delta_1, \delta_2}} &:= \|g\|_{\mathscr{X}} + \delta_1\|\grad g\|_{L^2} + \delta_2 \|g\|_{L^\infty}, \quad  \|g\|_{\mathscr{X}_{\delta_1, \delta_2}^*} := \sup_{\varphi \in C_0^\infty(\Omega), \|\varphi\|_{\mathscr{X}_{\delta_1, \delta_2}}\le 1}\int \varphi g.
	\end{align}	
	Our goal is to show that uniformly in $\delta_1, \delta_2 \in [0,1]$ there holds 
	\begin{equation} \label{eq:lemA2goal}
	\|u\|_{H^s} \le C\left(\|u\|_{\mathscr{X}_{\delta_1, \delta_2}} + \|X_0u\|_{\mathscr{X}_{\delta_1, \delta_2}^*}\right), \quad u\in C_0^\infty(K),
	\end{equation}	
	where $s$ and $C$ are as in the statement of the lemma. In the remainder of this proof, $C > 0$ denotes any such constant and $u$ denotes an arbitrary function in $C_0^\infty(K)$.
	
	Let $s_j = 1$ for $j = 1,\ldots,r$, and $s_0 = 1/2$. Let $\sigma, s' > 0$ satisfy 
	\begin{equation} \label{eq:defsigma}
	N_0^{-1} \leqc \sigma < s' < \min_{X_I \in V_{N_0}} s(I),
	\end{equation}
	where $X_I$ and $V_{N_0}$ are as in (\ref{eq:indexnotation}) and Definition~\ref{def:uniformHor}, respectively.
	Then, let $\mathcal{J}$ be the set of multi-indices with $\sigma m(I) \le 1$ that contain both zero and nonzero indices, and for $t > 0$ to be taken sufficiently small define
	\begin{equation} 
	\bar{M}(u) = \|u\|_{\mathscr{X}} + \sum_{I \in \mathcal{J}}|u|_{X_I,s(I)} + |u|^{t}_\sigma.
	\end{equation}
	Observe that $\bar{M}(u)$ is nothing more than the quantity $M(u)$ defined in the equation preceding [(5.6), \cite{H67}] but with the dual norm removed. It is clear from a reading of \cite{H67} that [Lemma 5.2, \cite{H67}] and [(5.16), \cite{H67}] both hold with $M(u)$ replaced $\bar{M}(u)$.
	
	Let $S_t$ denote the regularizer defined in the paragraph directly after the statement of [Theorem 5.1, \cite{H67}], and set $v_{t,\tau} = (e^{\tau X_0}S_t)^*(e^{\tau X_0}S_t u - S_t u)$ for $0 \le \tau \le t^2$. To prove (\ref{eq:lemA2goal}), we follow the proof of [(3.4), \cite{H67}] exactly, except we replace the estimate of the second term in [(5.15), \cite{H67}] with 
	\begin{align*} 
	\left|\int_\Omega (X_0 u) v \right| \le \|X_0 u\|_{\mathscr{X}_{\delta_1, \delta_2}^*} \|v_{t,\tau}\|_{\mathscr{X}_{\delta_1, \delta_2}} \leqc \|X_0 u\|_{\mathscr{X}_{\delta_1, \delta_2}^*}^2 + \|v_{t,\tau}\|_{\mathscr{X}}^2 +  \delta_1^2\|\grad v_{t,\tau}\|_{L^2}^2 + \delta_2^2\|v_{t,\tau}\|_{L^\infty}^2.
	\end{align*}
	Bounding $\|v_{t,\tau}\|_{\mathscr{X}}$ using [(5.16), \cite{H67}] with $M$ replaced by $\bar{M}$, and then proceeding as in the computations after [(5.6), \cite{H67}] results in the following modified version of [(3.4), \cite{H67}]:
	\begin{align} \label{eq:mod3.4}
	|u|^{t}_{s'} + \|u\|_{L^2} \leqc \|u\|_{\mathscr{X}} + \|X_0 u\|_{\mathscr{X}_{\delta_1, \delta_2}^*} +\delta_1 \sup_{0<|\tau| \le t^2}\|\grad v_{t,\tau}\|_{L^2} + \delta_2\sup_{0<|\tau| \le t^2}\|v_{t,\tau}\|_{L^\infty}.
	\end{align}	
	The estimate is uniform in $\delta_1$, $\delta_2$ and holds for $t$ sufficiently small. Moreover, a careful reading of \cite{H67} shows that in addition to depending of course on $K$ and $\Omega$, both the implicit constant and the smallness requirement on $t$ in (\ref{eq:mod3.4}) depend only on $r$, $C_0$, $N_0$, and an upper bound on $\sum_{j=0}^r \|X_j\|_{C^k(\Omega)}$ for some $k(N_0) > 0$. Applying (\ref{eq:finitediff2}), we thus obtain that for $\sigma < s < s'$ there holds
	\begin{equation} \label{eq:LemQuantHor1}
	\|u\|_{H^s} \le C(\|u\|_{\mathscr{X}} + \|X_0 u\|_{\mathscr{X}_{\delta_1, \delta_2}^*} +\delta_1 \sup_{0<|\tau| \le t^2}\|\grad v_{t,\tau}\|_{L^2} + \delta_2 \sup_{0<|\tau| \le t^2}\|v_{t,\tau}\|_{L^\infty}).
	\end{equation}
	
	It remains to bound the latter two terms of (\ref{eq:LemQuantHor1}) in terms of $\|u\|_{\mathscr{X}_{\delta_1, \delta_2}}$. Let $T_t = S_t$, $e^{tX_0}$,  $S_t^*$, or $(e^{tX_0})^*$. From the definition of $S_t$ (it is a finite product of operators that smooth along the vector fields $X_I$, $I \in \mathcal{J}$) it is straightforward to check that if $V_1$ and $V_2$ are open sets with $V_1 \subset \subset V_2 \subset \subset \Omega$, then for $t$ sufficiently small depending only on $r$, $N_0$, $V_1$, $V_2$, and $\sum_{j=0}^r \|X_j\|_{C^k(\Omega)}$, for any $g \in C_0^\infty(V_1)$ there holds 
	\begin{align} 
	& T_t g \in C_0^\infty(V_2), \\
	& \|\grad T_t g\|_{L^2} \le C\|g\|_{H^1}, \\ 
	&  \|T_t g\|_{L^\infty} \le C\|g\|_{L^\infty}.
	\end{align}
	Combining this with (\ref{eq:LemQuantHor1}) and recalling the definition of $v_{t,\tau}$ completes the proof.
\end{proof}

For $\mathscr{X}_\delta, \mathscr{X_\delta^*}$ as in \eqref{eq:defXdelta} with $X_j$ replaced by $Z_j$ and $g \in C_0^\infty(\R^d)$, let
\begin{equation}
\|g\|_{\Hhypd} : = \|g\|_{\mathscr{X_\delta}} + \|Z_{0,\epsilon} g\|_{\mathscr{X_\delta^*}},
\end{equation}
which is nothing more than the natural $\delta$-regularization of the $\Hhyp$ norm defined in \eqref{eq:defH1hyp}. Our main application of Lemma~\ref{lem:HineqOrigQuant} is a H\"{o}rmander inequality for $\Hhypd$ that is uniform in both $\delta \in [0,1]$ and $\epsilon \in (0,1)$. It is one of the key ingredients in the proofs of Lemmas~\ref{lem:Moser} and~\ref{lem:L2bound} carried out in Section~\ref{sec:Stationary}.

\begin{lemma}[H\"{o}rmander inequality for $\Hhypd$] \label{lem:HineqBall/Ann}
	 Let $R \ge 1$. There exists $s > 0$ such that for any $g \in C_0^\infty(B_R)$ there holds, uniformly in $\epsilon \in (0,1)$ and $\delta \in [0,1]$, 
	 $$ \|g\|_{H^s} \leqc R^{1-s} \|g\|_{H^1_{\text{hyp},\delta}}.$$
\end{lemma}

\begin{proof}
Let $\bar{g}(x) = g(Rx)$ so that $\bar{g} \in C_0^\infty(B_1)$. Define
$$Z_0 = N + \epsilon R^{-p+1} Ax + \epsilon^{\alpha}R^{-p+1}Bx,$$
where $p$ is the homogeneity degree of $N$. By Assumption~\ref{Assumption:spanning}, $\{Z_0,Z_1,\ldots,Z_r\}$ satisfies H\"{o}rmander's condition on $B_{2}$ with constants $(N_0,C_0) \in \N \times (0,\infty)$ that do not depend on $\epsilon$, and so by Lemma~\ref{lem:HineqOrigQuant} there exists $s(N_0) > 0$ such that
\begin{equation} \label{eq:HineqBall/Ann1}
\|\bar{g}\|_{H^s} \leqc \|\bar{g}\|_{\mathscr{X_\delta}} + \|Z_0 \bar{g}\|_{\mathscr{X_\delta}^*}.
\end{equation}
The implicit constant in (\ref{eq:HineqBall/Ann1}) depends on $r$, $N_0$, and $C_0$, but not on $\epsilon$ or $\delta$. Now, if $\varphi \in C_0^\infty$ with $\|\varphi\|_{\mathscr{X_\delta}} \le 1$, then by rescaling we have
$$ \int_{B_1} \varphi Z_0 \bar{g} \le \left\|\varphi\left(\frac{\cdot}{R}\right)\right\|_{\mathscr{X_\delta}}R^{-d - p + 1}\|Z_{0,\epsilon}g\|_{\mathscr{X_\delta}^*} \le R^{-d/2 - p + 1}\|Z_{0,\epsilon}g\|_{\mathscr{X_\delta}^*}.$$
Combining with $\|\bar{g}\|_{\mathscr{X_\delta}} \le R^{-d/2 + 1}\|g\|_{\mathscr{X_\delta}}$ we obtain
$$ \|\bar{g}\|_{\mathscr{X_\delta}} +  \|Z_0\bar{g}\|_{\mathscr{X_\delta}^*} \le R^{-d/2 + 1}\|g\|_{\Hhypd}.$$
Since $\|\bar{g}\|_{H^s} = R^{-d/2 + s}\|g\|_{H^s}$ it follows then from (\ref{eq:HineqBall/Ann1}) that
$$R^{-d/2 + s}\|g\|_{H^s} \leqc R^{-d/2 + 1}\|g\|_{\Hhypd}, $$
as desired.
\end{proof}

\begin{remark}
	Since Lemma~\ref{lem:HineqOrigQuant} does not use the parabolic H\"{o}rmander condition, Lemma~\ref{lem:HineqBall/Ann} holds just as well when the uniform spanning condition in Theorem~\ref{thrm:stationary} is replaced with the analogous statement requiring only H\"{o}rmander's condition.
\end{remark}

\subsection{H\"{o}rmander inequality for spaces involving time} \label{sec:HineqTime}

In this section, we discuss a parabolic H\"{o}rmander inequality that is natural for proving uniform-in-$\epsilon$ $L^2 \to L^\infty$ regularization estimates for the semigroup generated by $\epsilon^{-1}L_\epsilon$.

We begin with some notation. For an open, bounded set $\Omega \subset \R^d$, an open set $\Omega' \subseteq \Omega$, $\{X_j\}_{j=1}^r \subseteq T(\Omega)$, $t_0 \in \R$, and $t > 0$ we define the H\"{o}rmander norm pair

\begin{align}\label{eq:defXparabolic}
\|g\|_{L^2((t_0,t_0+t); \mathscr{X}(\Omega'))}& : = \left(\int_{t_0}^{t_0+t} \|g(\tau)\|_{\mathscr{X}(\Omega')}^2 d\tau \right)^{1/2},  \\
\|g\|_{L^2((t_0,t_0+t); \mathscr{X}^*(\Omega'))} &: = \sup_{\varphi \in C_0^\infty((t_0,t_0+t)\times \Omega'), \|\varphi\|_{L^2((t_0,t_0+t); \mathscr{X}(\Omega'))} \le 1} \int_{\R \times \Omega'} \varphi g. \label{eq:defXstarparabolic}
\end{align}
The notation for the  dual norm is motivated by the fact that it is possible to show $(L^2 \mathscr{X})^* \cong L^2 \mathscr{X}^*$ though we will not require this fact\footnote{As $\mathscr{X}$ is a Hilbert space, the Radon-Nikodym theorem extends to Bochner integrals of the form $L^2(0,T;\mathscr{X})$, which allows to show that every continuous linear functional on $L^2(0,T;\mathscr{X})$ can be represented in the form $\Lambda(f) = \int_0^T \brak{f,g}_{\mathscr{X} \times \mathscr{X}^*} dt$ for $g \in L^2(0,T;\mathscr{X}^*)$.}. 
We use all of the same notations when $(\mathscr{X},\mathscr{X}^*)$ is replaced with a different H\"{o}rmander norm pair. 

In the above setting and notations, the parabolic H\"{o}rmander inequality is as follows. 
\begin{lemma}[Uniform parabolic H\"{o}rmander inequality]\label{lem:Horparabolic}
	Let $K \subset \Omega$ be compact and suppose that $\{X_j\}_{j=0}^r \subseteq T(\Omega)$ satisfies the uniform parabolic H\"{o}rmander condition on $\Omega$ with constants $(N_0,C_0) \in \N \times (0,\infty)$. Fix $t_0 \in \R$, $t \in [1,10]$, and $0 < \eta \le 1/4$. There exists $s(N_0) > 0$ and a constant $C$ such that for all  $u \in C_0^\infty((t_0 + \eta,t_0 + t - \eta) \times K)$ there holds, uniformly in $\epsilon > 0$ and $\delta \in [0,1]$,
	$$\int_{t_0}^{t_0 + t}\|u(\tau,\cdot)\|^2_{H^s(\R^d)}d\tau \le C\left(\|u\|^2_{L^2((t_0,t_0+t); \mathscr{X_\delta})} + \|(\epsilon \partial_t + X_0)u\|_{L^2((t_0,t_0+t); \mathscr{X_\delta^*})}^2\right). $$ 
	The constant $C$ is uniformly bounded with respect to $\eta$ varying over compact time intervals away from the origin, and depends on $\{X_j\}_{j=0}^r$ only through $r$, $N_0$, $C_0$, and an upper bound on
	$ \sum_{j=0}^r\|X_j\|_{C^k(\Omega)}$ for some $k = k(N_0)$ sufficiently large.
\end{lemma}

The remainder of this section is devoted to the proof of Lemma~\ref{lem:Horparabolic}. We will assume throughout that $t=1$, $t_0 = 0$, and $\eta = 1/4$ since the general case is no different. Moreover, by the same arguments from the proof of Lemma~\ref{lem:HineqOrigQuant} it suffices to consider the case when $\delta = 0$. We will also suppress the superscript notation in (\ref{eq:DefFiniteDiff1})--(\ref{eq:DefFiniteDiff3}), with the understanding that the increment is always taken sufficiently small depending only on $K$ and finitely many derivatives of $\{X_j\}_{j=0}^r$. Lastly, unless otherwise stated, all implicit constants in this section satisfy the same properties as $C$ from the lemma statement. 

The proof of Lemma~\ref{lem:Horparabolic} is again based very closely on \cite{H67}. However, the generalization is a little more subtle than in Lemma~\ref{lem:HineqOrigQuant} because to make use of the parabolic H\"{o}rmander condition we need to give the time direction a distinguished role, something not done in \cite{H67}. The first step is to generalize [Theorem 4.3, \cite{H67}], which says that a function $u: \R^d \to \R$ with some regularity along the vector fields $\{X_j\}_{j=0}^r$ must have some regularity in all directions. Recall from Definition~\ref{def:uniformHor} and Section~\ref{sec:HineqDef} the definitions of $V_j$ and the seminorms $|\cdot|_{X,s}$, $|\cdot|_{\epsilon \partial_t + X,s}$ for $X \in T(\Omega)$ and $s \in (0,1]$. 

\begin{lemma} \label{lem:Hor4.6/4.3}
	Let $\{(X_j, s_j)\}_{j=0}^r \subseteq T(\Omega) \times [1/2,1]$ and $\gamma \gtrsim N_0^{-1}$. For $|\tau|$ sufficiently small and any $X_I \in V_j$ with $j \leqc_{N_0} 1$ there holds, uniformly in $\epsilon \in (0,1)$,
	\begin{equation} \label{eq:lemPar1}
	\int_0^1\|e^{\tau^{m(I)}X_I}u(t,\cdot) - u(t,\cdot)\|_{L^2}^2 dt \leqc \tau^2|u|_{\epsilon \partial_t + X_0,s_0}^2 + \tau^2\int_0^1 \left(\sum_{j=1}^r|u(t)|_{X_j,s_j}^2 + |u(t)|_\gamma^2\right)dt
	\end{equation}
	for every $u \in C_0^\infty((1/4,3/4) \times K)$.
	As a consequence, when $\{X_j\}_{j=0}^r \subseteq T(\Omega)$ satisfies the parabolic H\"{o}rmander condition on $\Omega$ with constants $(N_0,C_0)$ there exists $s(N_0) > 0$ so that uniformly in $\epsilon \in (0,1)$ there holds
	\begin{equation} \label{eq:lemPar2}
	\int_0^1 |u(t)|_s^2 dt \leqc \int_0^1\|u(t)\|_{L^2}^2 dt +  \sum_{j=1}^r \int_0^1 |u(t)|_{X_j,s_j}^2 dt +  |u|_{\epsilon \partial_t + X_0,s_0}^2, \quad u \in C_0^\infty((1/4,3/4) \times K).
	\end{equation}
\end{lemma}

\begin{proof}[Proof sketch]
	Let $I$ be a multi-index and $N \in \N$ be such that $N > |I|$, where as usual we write $|I|$ to denote the length of $I$. There exists a finite product decomposition (see [(4.13), \cite{H67}] and the discussion leading up to it)
	\begin{equation} \label{eq:productdecomp1}
	e^{\tau^{m(I)}X_I} = \left(\Pi e^{\pm \tau^{m_j}X_j}\right)e^{\tau^{m(I_1)}X_{I_1}} \ldots e^{\tau^{m(I_\ell)}X_{I_\ell}}H_N^\tau,
	\end{equation}
	where each multi-index $I_j$, $1 \le j \le \ell$ satisfies $|I| < |I_j| \le N$, and 
	$$H_N^\tau v(x) = v(g(x,\tau)), \quad v \in C_0^\infty(K)$$ 
	for a smooth mapping $g: K \times (-t_0,t_0) \to \Omega$ satisfying 
	$$\sup_{x\in K}|g(x,\tau)-x| = \mathcal{O}(|\tau|^N), \quad |\tau| \le t_0$$
	for $t_0$ sufficiently small. The decomposition~(\ref{eq:productdecomp1}) is obtained by iteratively using that from the Cambell-Baker-Hausdorff formula one has
	\begin{equation} \label{eq:CBH1}
	e^{-\tau X}e^{-\tau Y}e^{\tau X}e^{\tau Y} = e^{\frac{\tau^2}{2}[X,Y] + \ldots}, \quad X,Y \in T(\Omega)
	\end{equation}
	in the sense of formal power series, where $+\ldots$ denotes a series of iterated commutators of length at least three formed with $\tau X$ and $\tau Y$; see e.g. [pg.~162, \cite{H67}]. Since $[\partial_t, X] = [\partial_t,Y] = 0$ for $X,Y \in T(\Omega)$ viewed as constant in time vector fields on $\R^{d+1}$, it is clear that (\ref{eq:CBH1}) remains true with $Y$ in the left-hand side replaced by $\epsilon \partial_t + Y$ for any $\epsilon > 0$. It follows that, when lifted to an operator on functions of spacetime, (\ref{eq:productdecomp1}) holds	with every occurrence of $X_0$ on the right-hand side replaced by $\epsilon \partial_t + X_0$. In particular, the error $H_N^\tau$ in the Taylor expansion acts only on the spatial variables: 
	$$H_N^\tau u(t,x) = u(t,g(x,\tau)), \quad u \in C_0^\infty((1/4,3/4) \times K).$$
	 By choosing $N \leqc N_0$ such that $N\gamma \ge 1$, it follows then from [(4.11), \cite{H67}] and [Lemma 4.2, \cite{H67}] that for any $X_I \in V_j$ with $j \leqc_{N_0} 1$ and $u \in C_0^\infty((1/4,3/4)\times K)$ there holds
	\begin{equation} \label{eq:Hor4.6.1}
	\begin{aligned}\int_0^1 \|e^{\tau^{m(I)}X_I}u(t,\cdot) &- u(t,\cdot)\|_{L^2}^2 dt  \leqc |\tau|^2|u|_{\epsilon \partial_t + X_0}^2 + |\tau|^2\sum_{j=1}^r \int_0^1|u(t)|_{X_j,s_j}^2dt \\ 
	& + \sum_{j=1}^\ell \int_0^1 \|e^{\tau^{m(I_j)}X_{I_j}}u(t,\cdot) - u(t,\cdot)\|_{L^2}^2 dt  + |\tau|^2 \int_0^1 |u(t)|^2_\gamma  dt,
	\end{aligned}
	\end{equation}
	where each $I_j$, $1 \le j \le \ell$ satisfies $|I| < |I_j| \le N$. Using (\ref{eq:Hor4.6.1}), the proof of (\ref{eq:lemPar1}) follows from the induction argument in [Lemma 4.6, \cite{H67}]. 
	
	Now we turn to (\ref{eq:lemPar2}). Applying (\ref{eq:lemPar1}) and the arguments that lead to [(4.14), \cite{H67}] yields that for $\sigma$ and $s'$ as in (\ref{eq:defsigma}) there holds
	\begin{equation} \label{eq:lemPar3}
	\sup_{0<|h|\ll 1}|h|^{-2s'}\int_0^1 \|u(t,\cdot + h) - u(t,\cdot)\|_{L^2}^2 dt \leqc |u|_{\epsilon \partial_t + X_0,s_0}^2 + \sum_{j=1}^r \int_0^1 |u(t)|_{X_j,s_j}^2 dt + \int_0^1 |u(t)|_{\sigma}^2 dt.
	\end{equation}
	Let now $\sigma < s < s'$. By (\ref{eq:Besov}) and Fubini's theorem we have the bound 
	$$ \int_0^1 \|u(t)\|_{\dot{H}^{s}}^2 dt \approx \int_0^1 \|u(t)\|_{\dot{B}^{s}_{2,2}}^2dt \leqc \sup_{0 < |h| \ll 1} |h|^{-2s'}\int_0^1 \|u(t,\cdot + h) - u(t,\cdot)\|_{L^2}^2 dt + \int_0^1 \|u(t)\|_{L^2}^2 dt. $$
	Since $s > \sigma$, for every $\delta > 0$ there exists $C_\delta$ such that
	$$|u(t)|_{\sigma} \leqc \|u(t)\|_{\dot{B}^\sigma_{2,\infty}} \leqc \|u(t)\|_{\dot{B}^\sigma_{2,2}} \le \delta\|u(t)\|_{\dot{H}^{s}} + C_\delta\|u(t)\|_{L^2}. $$
	The previous two estimates together with (\ref{eq:lemPar3}) yield (\ref{eq:lemPar2}) with $|u(t)|_s$ replaced by $\|u(t)\|_{H^s}$. The proof is then complete since $|u(t)|_s \leqc \|u(t)\|_{H^s}$.
\end{proof}

With Lemma~\ref{lem:Hor4.6/4.3} at our disposal, the proof of Lemma~\ref{lem:Horparabolic} is a straightforward generalization of [section 5, \cite{H67}]. Throughout the entire proof we write $L^2\mathscr{X}$ and $L^2 \mathscr{X}^*$ to mean the norms taken on the time interval $(0,1)$. Also, for convenience we define $X_{0,\epsilon} = \epsilon \partial_t + X_0$.

\begin{proof}[Proof sketch of Lemma~\ref{lem:Horparabolic}]
	For $\sigma > 0$, $\{s_j\}_{j=0}^r$, and $\mathcal{J}$ all as in the proof of Lemma~\ref{lem:HineqOrigQuant}, let 
	\begin{align} 
	\tilde{M}(u) = \|u\|_{L^2\mathscr{X}}^2 + \sum_{I \in \mathcal{J}}\sup_{0 < |\tau| \ll 1}|\tau|^{-2}\int_0^1 \|e^{\tau^{m(I)}X_I}u(t) - u(t)\|_{L^2}^2 dt + \int_0^1 |u(t)|^2_\sigma dt.
	\end{align}
	Note that $\tilde{M}(u)$ is not equivalent to $\int_0^1 |\bar{M}(u(t))|^2 dt$ because in the second term the supremum over the increment is outside of the time integral. Using Lemma~\ref{lem:Hor4.6/4.3}, it follows from the arguments between [(5.6), \cite{H67}] and [(5.11), \cite{H67}] that to complete the proof of Lemma~\ref{lem:Horparabolic} it suffices to show that for $\tau > 0$ sufficiently small there holds	
	\begin{equation} \label{eq:par5.11}
	\int_0^1\|e^{\tau^2 X_{0,\epsilon}}S_\tau u(t) - S_\tau u(t)\|_{L^2}^2 dt \leqc \tau^2 \tilde{M}(u) + \tau^2 \|X_{0,\epsilon}u\|^2_{L^2\mathscr{X}^*},
	\end{equation}
	where $S_\tau$ denotes the same regularizer introduced in the proof of Lemma~\ref{lem:HineqOrigQuant}.
	
	To prove (\ref{eq:par5.11}) we proceed as in \cite{H67} and define 
	$$f(s) = \left(\int_0^1\|e^{s X_{0,\epsilon}}S_\tau u(t) - S_\tau u(t)\|_{L^2}^2 dt\right)^{1/2} \quad 0 < |s| \le \tau^2 $$
	with the goal of showing that $f(\tau) \leqc |\tau|\sqrt{\tilde{M}(u)} + |\tau| \|X_{0,\epsilon}\|_{L^2 \mathscr{X}^*}$. Since $S_\tau$ does not regularize in the time variable we clearly have $[S_\tau,X_{0,\epsilon}] = [S_\tau,X_0]$, and so differentiating $f^2$ with respect to $s$ gives
	\begin{equation} \label{eq:ff'}
	\frac{1}{2}\frac{d}{ds}f^2(s) = \brak{e^{s X_{0,\epsilon}}[X_0,S_\tau]u, e^{s X_{0,\epsilon}}S_\tau u - S_\tau u} + \brak{X_{0,\epsilon}u,(e^{s X_{0,\epsilon}}S_\tau)^*(e^{s X_{0,\epsilon}}S_\tau u - S_\tau u)}, 
	\end{equation}
	where $\brak{\cdot, \cdot}$ denotes the $L^2$-inner product over $(0,1) \times \Omega$. To estimate the right-hand side of the expression above we need the following lemma, which is a variation of [Lemma 5.2, \cite{H67}]. 
	\begin{lemma} \label{lem:parLemma5.2}
		Let $V$ be an open set with $V\subset \subset \Omega$. For $\tau > 0$ sufficiently small and every $v \in C_0^\infty((0,1)\times V)$ there holds 
		\begin{align}
		\int_0^1\|\tau^{1/\sigma}\grad S_\tau v(t)\|_{L^2}^2 dt  &\leqc  \tau^2 \tilde{M}(v) \label{eq:parLemma5.2.1} \\ 
		\sum_{I \in \mathcal{J}} \int_0^1 \|\tau^{m(I)}X_I S_\tau v(t)\|_{L^2}^2 dt &\leqc \tau^2 \tilde{M}(v) \label{eq:parLemma5.2.2}\\ 
		\sum_{j=0}^r \int_0^1[\tau^{m_j}X_j,S_\tau]v(t)\|_{L^2}^2 dt & \leqc \tau^2 \tilde{M}(v).\label{eq:parLemma5.2.3}
		\end{align}	
	\end{lemma}
A remark on the proof is in order because the second term in $\tilde{M}$ is weaker than 
$$ \sum_{I \in \mathcal{J}} \int_0^1 |u(t,\cdot)|_{X_I, s(I)}^2 dt, $$
and so the lemma does not follow simply by integrating the estimates in [Lemma 5.2, \cite{H67}].
\begin{proof}[Proof sketch of Lemma~\ref{lem:parLemma5.2}]
	The left-hand side of (\ref{eq:parLemma5.2.1}) is bounded using the latter term in the definition of $\tilde{M}$. The estimate follows in the same manner as [(5.12), \cite{H67}] because in this term the supremum over the increment is inside the time integral. 
	
	Now we turn to (\ref{eq:parLemma5.2.2}). Let $\varphi_{X}$ denote the regularizer defined in [section 5, \cite{H67}]. By Minkowski's inequality, Jensen's inequality, and the equation preceding [(5.1), \cite{H67}] there holds
	$$
	\int_0^1 \|\tau^{m(I)} X_I \varphi_{\tau^{m(I)} X_I} v(t)\|_{L^2}^2 dt  \le \int_0^1 \int_{-1}^1 \|e^{s\tau^{m(I)}X_I}v(t) - v(t)\|_{L^2}^2 |\varphi'(s)|^2 ds dt.
	$$
	Employing also Fubini's theorem we then get
	\begin{align*}
	\int_0^1 \|\tau^{m(I)} & X_I \varphi_{\tau^{m(I)} X_I} v(t)\|_{L^2}^2 dt \\ 
	&  \le \tau^2 \int_{-1}^1 |s|^{2/m(I)} |\varphi'(s)|^2 \left((|s|^{1/m(I)}|\tau|)^{-2} \int_0^1 \|e^{(|s|^{1/m(I)}\tau)^{m(I)}X_I}v(t) - v(t)\|_{L^2}^2 dt\right) ds \\ 
	& \le \tau^2 \tilde{M}(v) \int_{-1}^1 |s|^{2/m(I)} |\varphi'(s)|^2 ds \leqc \tau^2 \tilde{M}(v).	
	\end{align*}
	Hence, (\ref{eq:parLemma5.2.2}) holds with the summation replaced by a fixed $I \in \mathcal{J}$ and $S_\tau$ replaced by the individual regularizer $\varphi_{\tau^{m(I)}X_I}$. 
	The induction trick in the proof of [Lemma 5.2, \cite{H67}] then works to upgrade to (\ref{eq:parLemma5.2.2}). 
	Adapting the methods from [Lemma 5.2, \cite{H67}] to obtain (\ref{eq:parLemma5.2.3}) is done similarly.
\end{proof}
	
	Using (\ref{eq:parLemma5.2.3}) in (\ref{eq:ff'}) we have 
	$$f(s) f'(s) \leqc  \tau^{-1}\sqrt{\tilde{M}(u)}f(s) + \|X_{0,\epsilon}u\|_{L^2 \mathscr{X}^*}^2 + \|(e^{s X_{0,\epsilon}}S_\tau)^*(e^{s X_{0,\epsilon}}S_\tau u - S_\tau u)\|_{L^2\mathscr{X}}^2.$$
	From the elementary ODE computation proceeding [(5.15), \cite{H67}], it follows that to complete the proof of (\ref{eq:par5.11}) it suffices to show that the latter term above can be controlled by $\tilde{M}(u)$. To this end, for $v \in C_0^\infty((0,1)\times \Omega)$ we define 
	$$\tilde{N}_\tau(v) = \|v(t)\|_{L^2\mathscr{X}}^2 + \sum_{I \in \mathcal{J}}\int_0^1 \|\tau^{m(I)-1}X_I v(t)\|_{L^2}^2 dt + \int_0^1 \|\tau^{1/\sigma - 1}\grad v(t)\|_{L^2}^2 dt. $$ 
	Applying Lemma~\ref{lem:parLemma5.2} and the arguments that lead to [(5.17), \cite{H67}] and [(5.18), \cite{H67}] gives that for any open set $V \subset \subset \Omega$ and $v\in C_0^\infty((0,1)\times V)$, as long as $\tau$ is sufficiently small there holds
\begin{align}
	\tilde{N}_\tau(S_\tau v) & \leqc \tilde{M}(v), \label{eq:Ntbound1}  \\  
	\tilde{N}_\tau(e^{s X_I} v) & \leqc \tilde{N}_{\tau}(v), \qquad 0 \le |s| \le \tau^{m(I)}, \quad I \in \mathcal{J}. \label{eq:Ntbound2}
\end{align}		
	Because $[\partial_t,X_0] = 0$ and $e^{\pm s\epsilon \partial_t}$  is bounded with respect to $L^2 \mathscr{X}$ we have 
	\begin{align*}
	\|(e^{s X_{0,\epsilon}}S_\tau)^*(e^{s X_{0,\epsilon}}S_\tau u(t) - S_\tau u(t))\|_{L^2 \mathscr{X}}^2 \leqc \tilde{N}_\tau((e^{sX_0}S_\tau)^* e^{sX_0}S_\tau u) + \tilde{N}_\tau((e^{sX_0}S_\tau)^* S_\tau u). 
	\end{align*}
	Now, (\ref{eq:Ntbound1}) and \eqref{eq:Ntbound2} along with the form of $S_\tau$ imply that $(e^{sX_0}S_\tau)^*$ is bounded with respect to $\tilde{N}_\tau$. Hence, we have
	$$ \|(e^{s X_{0,\epsilon}}S_\tau)^*(e^{s X_{0,\epsilon}}S_\tau u(t) - S_\tau u(t))\|_{L^2\mathscr{X}}^2 \leqc \tilde{N}_\tau(e^{sX_0}S_\tau u) + \tilde{N}_\tau(S_\tau u) \leqc \tilde{M}(u), $$
	which completes the proof. 
	\end{proof}

\section{Uniform estimates on the stationary measure}\label{sec:Stationary}

\subsection{Uniform $L^2$ estimate for $f_\epsilon$} \label{sec:StationaryL2}
The purpose of this section is to prove Lemma \ref{lem:L2bound}. 

Recall from Section~\ref{sec:outline} that in order to more easily justify formal calculations, we introduce the following regularization: $\forall \delta,\epsilon > 0$, define $f_{\epsilon,\delta} \geq 0$ with $\int f_{\epsilon,\delta}  = 1$ to be the unique solution to the problem  
\begin{align}
\delta \Delta f_{\epsilon,\delta} + \frac{1}{\epsilon}L_\epsilon^\ast f_{\epsilon,\delta} = 0. \label{eq:fedeq}
\end{align}
In Appendix~\ref{sec:basics}, we sketch the proof that this problem is well-posed and that $f_{\epsilon,\delta} \in L^2$ $\forall \delta > 0$. Note also that $f_{\epsilon,\delta}$ satisfies the moment bound (\ref{eq:momentbound}) uniformly in $\epsilon, \delta \in (0,1)$ (see \eqref{eq:AppDistrMoment}) and from classical elliptic theory \cite{GT} there holds $\forall R > 0$, 
\begin{align}
\norm{f_{\epsilon,\delta}}_{H^k(B_R)} & \lesssim_{R,k,\delta,\epsilon} 1 \label{ineq:edfinHk}.
\end{align}
Now, we are interested in obtaining estimates uniform in both $\epsilon$ and $\delta$, and then passing to the limit $\delta \to 0$.

Define $\bar{\chi}(x) \in C^\infty_0(B_1)$ radially symmetric such that $\bar{\chi} = 1$ for $\abs{x} < 1/2$. Define $\chi(x) = \bar{\chi}(x/2) - \bar{\chi}(x)$ which is now a $C^\infty_0(B_2\setminus B_{1/2})$ function. Further, define $\chi_R  = \chi(x/R)$ and note that 
\begin{align}
1 = \bar{\chi} + \sum_{R = 2^j: j\ge 0} \chi_R. \label{ineq:chipartu}
\end{align}

\begin{proof}[Proof of Lemma \ref{lem:L2bound}]
  We proceed by proving that
  \begin{align}
  \sup_{\epsilon,\delta \in (0,1)} \norm{f_{\epsilon,\delta}}_{L^2} \lesssim 1.
\end{align}
  This implies the second inequality in (\ref{eq:H1hypuniform}) by passing $\delta \to 0$ and using lower-semicontinuity along with the uniqueness of $f_{\epsilon}$ for each $\epsilon > 0$ (see Lemma~\ref{lem:MarkovSemigroup} and the remark following it).
  The first inequality in (\ref{eq:H1hypuniform}) then follows from Lemma \ref{lem:AprioriBounds}.  

\textbf{Step 1:  estimates on $\bar{\chi}f_{\epsilon,\delta}$}:
Multiplying \eqref{eq:fedeq} by $\bar{\chi}$ and using the energy property $N(x)\cdot x = Bx \cdot x = 0$ together with the radial symmetry of $\bar{\chi}$, we obtain
\begin{align}
\left(\epsilon \delta \Delta + L_\epsilon^*\right)\bar\chi f_{\epsilon,\delta} = [\epsilon \delta\Delta,\bar\chi] f_{\epsilon,\delta} + \epsilon \sum_{j=1}^r [Z_j^2,\bar{\chi}] f_{\epsilon,\delta} + \epsilon[Ax\cdot \grad,\bar\chi]f_{\epsilon,\delta}.    
\end{align}
Pairing with $\bar{\chi} f_{\epsilon,\delta}$ gives the a priori estimate
\begin{align}
  \delta \norm{\grad (\bar{\chi}f_{\epsilon,\delta})}_{L^2}^2 + \sum_{j=1}^r\norm{Z_j (\bar{\chi} f_{\epsilon,\delta})}_{L^2}^2  \lesssim \norm{f_{\epsilon,\delta}}_{L^2}^2. \label{ineq:fedH1d} 
\end{align}
Similarly, we pair with a test function $v \in C_0^\infty(\R^d)$ satisfying $\|v\|_{\mathscr{X}_\delta} \le 1$ and obtain, using \eqref{ineq:fedH1d},
\begin{align}
  \abs{\int v Z_{0,\epsilon} \bar{\chi} f_{\epsilon,\delta} dx} & \lesssim \epsilon \delta \norm{\grad (\bar{\chi} f_{\epsilon,\delta})}_{L^2} \norm{\grad v}_{L^2} + \epsilon\sum_{j=1}^r\norm{Z_j (\bar{\chi} f_{\epsilon,\delta})}_{L^2} \norm{Z_j v}_{L^2} \\
& \quad + \epsilon \delta \norm{f_{\epsilon,\delta}}_{L^2} \norm{\grad v}_{L^2} + \epsilon\sum_{j=1}^r\norm{Z_j v}_{L^2} \norm{f_{\epsilon,\delta}}_{L^2} \\
& \quad + \epsilon \norm{f_{\epsilon,\delta}}_{L^2} \norm{v}_{L^2} \\ 
& \leqc \epsilon \|f_{\epsilon,\delta}\|_{L^2}.
\end{align}
Combining with \eqref{ineq:fedH1d} we then have, uniformly in $\delta,\epsilon$, 
\begin{align}
\norm{\bar{\chi} f_{\epsilon,\delta}}_{\Hhypd} \lesssim \norm{f_{\epsilon,\delta}}_{L^2}. \label{eq:chibarHhypbyL2}
\end{align}
Thus, by Lemma \ref{lem:HineqBall/Ann} and Sobolev embedding, $\exists \theta \in (0,1)$ (depending on dimension but not $\epsilon$) such that 
\begin{align} \label{ineq:bchif}
  \norm{\bar{\chi} f_{\epsilon,\delta}}_{L^2} \lesssim \norm{\bar{\chi} f_{\epsilon,\delta}}_{L^1}^{1-\theta} \norm{\bar{\chi}f_{\epsilon,\delta}}_{H^s}^\theta \lesssim \norm{\bar{\chi} f_{\epsilon,\delta}}_{L^1}^{1-\theta} \norm{\bar{\chi}f_{\epsilon,\delta}}_{\Hhypd}^\theta \lesssim \norm{f_{\epsilon,\delta}}_{L^2}^\theta.
\end{align}

\textbf{Step 2:  estimates on $\chi_R f_{\epsilon,\delta}$}:
For any $R \geq 1$, by applying the same arguments as in the case of $\bar{\chi}$ and using $\|\grad^j \chi_R\|_{L^\infty} \leqc R^{-j}$ to control the commutator error terms, we similarly obtain
\begin{align} \label{eq:H1hypByL2}
\norm{\chi_R f_{\epsilon,\delta}}_{\Hhypd} \lesssim \norm{f_{\epsilon,\delta}}_{L^2}. 
\end{align}
Therefore, again by Lemma \ref{lem:HineqBall/Ann} and Sobolev embedding, $\exists \theta \in (0,1)$ such that
\begin{align}
\norm{\chi_R f_{\epsilon,\delta}}_{L^2} \lesssim R\norm{\chi_R f_{\epsilon,\delta}}_{L^1}^{1-\theta} \norm{f_{\epsilon,\delta}}_{L^2}^\theta. \label{ineq:cRf}
\end{align}

\textbf{Step 3: $L^2$ estimates:}
By \eqref{ineq:chipartu}, Young's inequality, \eqref{ineq:bchif}, and \eqref{ineq:cRf}, we have
\begin{align}
\norm{f_{\epsilon,\delta}}_{L^2} & \leq \norm{\bar{\chi} f_{\epsilon,\delta}}_{L^2} + \sum_{2^j : j \geq 0} \norm{ \chi_{2^j} f_{\epsilon,\delta}}_{L^2} \\
& \leqc  \norm{f_{\epsilon,\delta}}_{L^2}^\theta + \norm{f_{\epsilon,\delta}}_{L^2}^\theta \sum_{j \geq 0} 2^j \norm{\chi_{2^j} f_{\epsilon,\delta}}_{L^1}^{1-\theta} \\ 
& \leqc \|f_{\epsilon,\delta}\|^\theta\left(1 + \sum_{j \ge 0}2^{-\frac{j}{\theta}} + \sum_{j \ge 0} 2^\frac{2j}{1-\theta} \|\chi_{2^j} f_{\epsilon,\delta}\|_{L^1} \right)  \\
& \leqc \norm{f_{\epsilon,\delta}}_{L^2}^\theta \left(1  + \sum_{2^j : j \geq 0} \norm{ \brak{\cdot}^{\frac{2}{1-\theta}} \chi_{2^j} f_{\epsilon,\delta}}_{L^1} \right). 
\end{align}
Applying \eqref{eq:AppDistrMoment} with $V$ as in \eqref{eq:momentbound} then implies
\begin{align}
\norm{f_{\epsilon,\delta}}_{L^2} & \leqc \norm{f_{\epsilon,\delta}}_{L^2}^\theta.
\end{align}
The desired result follows from $\theta < 1$ and $\|f_{\epsilon,\delta}\|_{L^2} < \infty$.
\end{proof} 

\begin{remark}
	An important consequence of the proof above which we require later is that
	\begin{equation} \label{eq:H1hypduniform}
	\sup_{\epsilon,\delta \in (0,1)}\|f_{\epsilon,\delta}\|_{\Hhypd} \leqc 1.
	\end{equation}
\end{remark}

\subsection{Hypoelliptic Moser iteration} \label{sec:StationaryMoser}

Next we apply a Moser iteration to obtain the local $L^\infty$ estimate in Lemma~\ref{lem:Moser}. 
\begin{proof}[Proof of Lemma~\ref{lem:Moser}]
Let $f \in C^\infty(\R^d)$ satisfy $f \geq 0$ and $(\epsilon \delta \Delta + L_\epsilon^*)f \ge 0$. By replacing $f$ with $f + \epsilon'$ and then sending $\epsilon' \to 0$ we may assume without loss of generality that $f > 0$. Fix $R \ge 1$ and for each $k \ge 0$ define $R_k = R(1+2^{-k})$. With $s$ as given in Lemma~\ref{lem:HineqBall/Ann}, let $\alpha>1$ be such that $H^s \hookrightarrow  L^{2\alpha}$ and define $w_k = f^{\alpha^k}$. We prove that $\exists C > 0$ (depending only on $R$ and dimension) such that  for $k \geq 0$, 
\begin{align}
\norm{w_k}_{L^{2\alpha}(B_{R_{k+1}})} \leq  C^{k} \norm{w_k}_{L^2(B_{R_{k}})}. \label{ineq:iterate}
\end{align}
By the convexity of $z \mapsto z^{\beta}$, 
\begin{align} \label{eq:MoserConvex}
\delta \Delta w_k + \sum_{j=1}^r Z_j^2 w_k + \frac{1}{\epsilon}Z_{0,\epsilon}w_k + \alpha^k \text{Tr} A w_k \geq 0. 
\end{align}
Let $\chi_k \in C_0^\infty(B_{R_{k}})$ be a radially-symmetric, smooth cutoff function satisfying $\chi_k(x) = 1$ for $|x| \le R_{k+1}$ and $|D^\beta \chi_k| \leqc R^{-1}2^{|\beta|k}$ for every multi-index $\beta$ with $|\beta| \le 2$. Denoting $v_{k} = \chi_{k} w_{k}$ and using \eqref{eq:MoserConvex} we obtain
\begin{align}
\delta \Delta v_{k} + \sum_{j=1}^r Z_j^2 v_k + \frac{1}{\epsilon}Z_{0,\epsilon}v_{k} + \alpha^k \text{Tr} A v_{k} - \mathcal{C} \geq 0, 
\end{align}
where
\begin{align}
\mathcal{C} = [\delta\Delta, \chi_k]w_k + \sum_{j=1}^r [Z_j^2,\chi_{k}] w_k + [Ax\cdot \grad,\chi_{k} ]w_k. 
\end{align}
Pairing with $v_k = \chi_k w_k$ we obtain the a priori estimate
\begin{align} \label{eq:subsolXdeltabound}
\delta \|\grad v_k\|_{L^2}^2 + \sum_{j=1}^r \|Z_j v_k\|_{L^2}^2 \lesssim \alpha^k \norm{v_k}_{L^2}^2 + 2^{2k} \norm{w_k}_{L^2(B_{R_{k}})}^2. 
\end{align}
Let $g$ be the unique solution to the Dirichlet problem
\begin{align}
  \left\{
  \begin{array}{l}
\delta \Delta g +\sum_{j=1}^r Z_j^2 g + \frac{1}{\epsilon}Z_{0,\epsilon} g + \alpha^{k} \text{Tr} A v_k  - \mathcal{C} = 0 \\
  g|_{\partial B_{2R+1}} = 0.
  \end{array}
  \right.
\end{align}
By the weak elliptic maximum principle we have $v_{k} \leq g$ and, in particular, for all $L^p$,  we have $\norm{v_k}_{L^p} \leq \norm{g}_{L^p}$. 
Moreover, we have the a priori estimate
\begin{align}
\delta\norm{\grad g}_{L^2}^2 + \sum_{j=1}^r \|Z_j g\|_{L^2}^2 + \norm{g}_{L^2}^2 \lesssim_R \alpha^{2k} \norm{v_k}_{L^2}^2 + 2^{4k} \norm{w_k}_{L^2(B_{R_{k}})}^2.
\end{align}
Multiplying by a radially-symmetric, smooth cutoff $\chi \in C_0^\infty(B_{2R + 1/2})$ with $\chi(x) = 1$ for $|x| \le 2R$ and applying the arguments we used in the proof of Lemma \ref{lem:L2bound} we obtain 
\begin{equation} 
\|\chi g\|_{\Hhypd} \leqc_R \alpha^{2k}\|v_k\|_{L^2} + 2^{2k}\|w_k\|_{L^2(B_{R_k})},
\end{equation}
and so by Lemma \ref{lem:HineqBall/Ann} we have
\begin{align}
\|w_k\|_{L^{2\alpha}(B_{R_{k+1}})}\le \norm{v_k}_{L^{2\alpha}} \le \norm{\chi g}_{L^{2\alpha}} \leqc_R \alpha^{2k}\norm{v_k}_{L^2} + 2^{2k} \norm{w_k}_{L^2(B_{R_k})}. 
\end{align}
This completes the proof of the iteration \eqref{ineq:iterate}. 

The bound \eqref{ineq:iterate} implies that for some $C > 0$ (depending only on $R$ and dimension) there holds  
\begin{align}
\norm{f}_{L^{2\alpha^{k+1}}(B_{R_{k+1}})} \leq  C^{k \alpha^{-k}} \norm{f}_{L^{2\alpha^k}(B_{R_k})}, 
\end{align}
which by iteration gives
\begin{align}
\norm{f}_{L^{2\alpha^{k+1}}(B_{R_{k+1}})} \leq C^{\sum_{j=0}^k j \alpha^{-j}} \norm{f}_{L^2(B_{2R})} 
\end{align}
for every $k \ge 0$. Using that $\alpha > 1$, we pass to the limit $k \to \infty$ and obtain the desired result. 
\end{proof}

As above, we use the regularization by $\delta$ and pass to the limit to deduce the final estimate on $f_\epsilon$.

\begin{proof}[Proof of (\ref{eq:ubd})]
	 Combining Lemmas~\ref{lem:Moser} and~\ref{lem:L2bound} we have, for every $R \ge 1$,
	 \begin{equation}
	 \sup_{\epsilon,\delta \in (0,1)} \|f_{\epsilon,\delta}\|_{L^\infty(B_{R})} \leqc \sup_{\epsilon,\delta \in (0,1)}\|f_{\epsilon,\delta}\|_{L^2(B_{2R})} \leqc 1.
	 \end{equation}
	 Sending $\delta \to 0$, the bound (\ref{eq:ubd}) follows from lower semicontinuity and the uniqueness of $f_\epsilon$.
\end{proof}

\subsection{Intermediate value lemma and proof of Lemma~\ref{lem:IVTtolbd}} \label{sec:StatIVT}

\begin{proof}[Proof of Lemma~\ref{lem:IVT}]
As discussed in Section~\ref{sec:Outlinelower}, we follow a compactness-rigidity scheme as in \cite{GIMV16}, obtaining the necessary compactness by H\"ormander inequalities and the necessary rigidity from Lemma \ref{lem:rigidity}. 

If the lemma fails, then there exists a sequence $\{(\delta_n,\epsilon_n)\}_{n=1}^\infty \subseteq (0,1) \times (0,1)$ with $\lim_n \epsilon_n = 0$ and $\{w_n\}_{n=1}^\infty \subseteq C^\infty(B_{2R})$ satisfying the following properties:
	
	\begin{itemize}
		\item $0 \le w_n \le 1$
		\item $|\{w_n = 0\} \cap B_R| \ge \alpha_1$
		\item $|\{w_n \ge 1-\frac{1}{n}\} \cap B_R| \ge \alpha_2$
		\item $|\{0<w_n<1 - \frac{1}{n}\} \cap B_R| < \frac{1}{n}$; 
	\end{itemize}
	and moreover
	\begin{align}
	0 \le \delta_n \epsilon_n \Delta w_n + L_{\epsilon_n}^* w_n  \le \sqrt{\epsilon_n}\left( 1 + \delta_n |\grad f_{\epsilon_n,\delta_n}|^2 + \sum_{j=1}^r |Z_j f_{\epsilon_n,\delta_n}|^2\right). \label{ineq:wneq}
	\end{align}	
	By the uniform estimate $\|w_n\|_{L^\infty} \le 1$ and the Banach-Alaoglu theorem, $\exists w \in L^\infty$ such that $$ w_n \rightharpoonup_* w$$ in $L^\infty$ up to extracting a subsequence (not relabled). 
	
	Now we obtain the needed compactness. Let $\chi \in C_0^\infty(B_{2R})$ be radially symmetric with $0 \le \chi \le 1$ and $\chi(x) = 1$ for $|x| \le R$. From the lower bound in (\ref{ineq:wneq}) and the arguments that led to (\ref{eq:subsolXdeltabound}), we have 
	\begin{equation} \label{eq:wnH1hyp}
	\delta_n \|\grad (\chi w_n)\|_{L^2}^2 + \sum_{j=1}^r\|Z_j (\chi w_n)\|_{L^2}^2 \leqc \int_{B_{2R}}|w_n|^2dx \leqc \abs{B_{2R}},
	\end{equation}
	where the constant is independent of $n$ using $0\le w_n \le 1$. Moreover, pairing (\ref{ineq:wneq}) with $\chi \varphi$ for $\varphi\in C_0^\infty$ yields 
	\begin{align} 
	\left|\int Z_{0,\epsilon_n} (\chi w_n)\varphi \right| & \leqc \sum_{j=1}^r \|Z_j \varphi\|_{L^2}\|Z_j(\chi w_n)\|_{L^2} + \delta_n\|\grad \varphi\|_{L^2}\|\grad(\chi w_n)\|_{L^2} \\ 
	& \quad + \|w_n\|_{L^2(B_{2R})}\left(\|\varphi\|_{L^2} + \delta_n\|\grad \varphi\|_{L^2} + \sum_{j=1}^r \|Z_j \varphi\|_{L^2}\right) \\ 
	& \quad + \|\varphi\|_{L^\infty}\left(1 + \|f_{\epsilon_n,\delta_n}\|_{H^1_{\text{hyp},\delta_n}}^2\right).
	\end{align}
	Combining with (\ref{eq:wnH1hyp}) and (\ref{eq:H1hypduniform}) it follows that 
	\begin{equation} \label{eq:wndual}
		\left|\int Z_{0,\epsilon_n} (\chi w_n)\varphi \right| \leqc \|\varphi\|_{L^\infty} + \|\varphi\|_{L^2} + \sum_{j=1}^r \|Z_j \varphi\|_{L^2} + \sqrt{\delta_n} \|\grad \varphi\|_{L^2}.
	\end{equation}
In the notations $\tilde{\mathscr{X}_\delta}$ and $\tilde{\mathscr{X}^*_\delta}$ from \eqref{eq:defXdeltainf}, the bounds \eqref{eq:wnH1hyp} and \eqref{eq:wndual} together imply 
\begin{equation}
\|\chi w_n\|_{\tilde{\mathscr{X}_{\delta_n}}} + \|\chi w_n\|_{\tilde{\mathscr{X}^*_{\delta_n}}} \leqc 1
\end{equation}
uniformly in $n$. Applying Lemma~\ref{lem:HineqOrigQuant} we conclude that for some $s > 0$, 
$$\sup_{n \ge 1}\|\chi w_n\|_{H^s} \leqc 1.$$

Therefore, by compact embedding (up to extracting another subsequence) $w_n \to w$ strongly in $L^p(B_{R})$ for some $p > 2$. In particular, $w_n \to w$ in measure. Moreover, using (\ref{eq:H1hypduniform}), passing $n\to \infty$ in the sense of distributions in (\ref{ineq:wneq}) we obtain that $w \in L^2(B_{R})$ is a distributional solution to
$$N \cdot \grad w = 0$$
on $B_{R}$. Convergence in measure and lower semicontinuity moreover provide
\begin{itemize}
	\item $0 \le w \le 1$ 
	\item $|\{w = 0\} \cap B_{R}| \ge \alpha_1$
	\item $|\{w = 1\} \cap B_{R}| \ge \alpha_2$
	\item $|\{0 < w < 1\} \cap B_{R}| = 0.$
\end{itemize}
Now, by the Banach-Alaoglu theorem, lower semicontinuity, and (\ref{eq:wnH1hyp}) we have 
\begin{equation} \label{eq:wXjL2}
Z_j w \in L^2(B_{R}), \quad 1 \le j \le r.
\end{equation}
Since $w$ is a characteristic function on $B_{R}$ and $H^1$ functions cannot have jump discontinuities, (\ref{eq:wXjL2}) implies that $Z_j w = 0$ for $j = 1,\ldots,r$. By Assumption~\ref{Assumption:spanning}, the collection $\{N,Z_1,\ldots,Z_r\}$ satisfies H\"{o}rmander's condition on $\R^d$. Thus, Lemma~\ref{lem:rigidity} and the second bullet imply that $w \equiv 0$ on $B_{R}$. However, this contradicts the third bullet, and so no such $w$ can exist. This completes the proof of Lemma~\ref{lem:IVT}.
\end{proof}

We will need a regularized version of the function $z \to z_+$ that smooths out the kink at the origin in such a way so that the signed term that appears when passing solutions through the resulting convex function does not blow up too fast.
\begin{lemma} \label{lem:phieps}
For all $\epsilon > 0$  $\exists \phi_\epsilon:\R \to \R$ that satisfies the following properties: 
\begin{itemize}
    \item $\phi_\epsilon$ is smooth with $\|\phi_\epsilon''\|_{L^\infty(\R)} \leqc \epsilon^{-1/4}$
    \item $\phi_\epsilon'' \ge 0$
    \item $\phi_\epsilon(x) = x$ when $x \ge \epsilon^{1/4}$
    \item $\phi_\epsilon(x) = 0$ when $x \le -\epsilon^{1/4}$
    \item $\phi_\epsilon(x)$ is nondecreasing with $\|\phi_\epsilon'\|_{L^\infty} \leqc 1$ and $\phi_\epsilon(x) > 0$ for $x > -\epsilon^{1/4}$
\end{itemize}
\end{lemma}
\begin{proof} 
Define a symmetric, smooth function $\varphi \in C_0^\infty([-\epsilon^{1/4}, \epsilon^{1/4}])$ with $\varphi(x) > 0$ for $|x| < \epsilon^{1/4}$, $\int \varphi = 1$, and $\|\varphi\|_{L^\infty} \leqc \epsilon^{-1/4}$. Define $\phi_\epsilon$ to be the mollification 
\begin{equation}
    \phi_\epsilon(x) = \int_{-\epsilon^{1/4}}^{\epsilon^{1/4}} \varphi(y) (x+y)_+ dy. 
\end{equation}
Note:
\begin{equation}
    \phi_\epsilon(x) = 
    \begin{cases}
    0 & x \le - \epsilon^{1/4} \\ 
    x\int_{-x}^{\epsilon^{1/4}} \varphi(y)dy +\int_{-x}^{\epsilon^{1/4}} y\varphi(y)dy & -\epsilon^{1/4} < x < \epsilon^{1/4} \\ 
    x & x \ge \epsilon^{1/4}.
    \end{cases}.
\end{equation}
The properties asserted above follow directly. 
\end{proof} 

\begin{proof}[Proof of Lemma~\ref{lem:IVTtolbd}]
	A byproduct of proving Lemmas~\ref{lem:Moser} and ~\ref{lem:L2bound} is that 
	\begin{equation} \label{eq:regubd}
	\sup_{\epsilon \in (0,1)}\sup_{\delta \in (0,1)} \|f_{\epsilon,\delta}\|_{L^\infty(B_R)} \leqc_R 1, \quad R > 0.
	\end{equation}
	Combining this with $\int f_{\epsilon,\delta} = 1$ and \eqref{eq:AppDistrMoment}, we see that there exist constants $c_1, c_2, R_0 > 0$ independent of $\epsilon,\delta$ such that
	\begin{equation} \label{eq:wk0}
	|\{f_{\epsilon,\delta} \ge c_1\} \cap B_{R_0}| \ge c_2.
	\end{equation}
	Let $\tilde{f}_{\epsilon,\delta} = f_{\epsilon, \delta}/c_1$. For any $\theta \in (0,1)$ we define the sequence of functions
	\begin{equation} \label{eq:defwk}
	\tilde{w}_{k,\theta}^{\epsilon,\delta} = 1 - \left(\frac{4}{\theta}\right)^{k}\tilde{f}_{\epsilon,\delta}, \quad w_{k,\theta}^{\epsilon,\delta} = \phi_\epsilon(\tilde{w}^{\epsilon,\delta}_{k,\theta}),
	\end{equation}
	where $\phi_\epsilon$ is the function guaranteed by Lemma~\ref{lem:phieps}. When $\epsilon$ and $\delta$ are clear from context, we suppress them from the notation and simply write $w_{k,\theta}$. A direct consequence of the construction and (\ref{eq:wk0}) is that for any $\theta \in (1/2)$, $\delta \in (0,1)$, $\epsilon \in (0,1/16)$, and $k, \ell \in \N$ there holds
	\begin{align}
	& |\{w_{k,\theta} = 0\} \cap B_{R_0}| \ge c_2,  \label{eq:wk1} \\ 
	& \{w_{k+1,\theta} > 0\} \subseteq \{w_{k,\theta} \ge 1-\theta\}, \label{eq:wk2}\\ 
	&\{0 < w_{k,\theta} < 1-\theta\} \cap \{0 < w_{\ell,\theta} < 1-\theta\} = \emptyset,  \quad k \neq \ell. \label{eq:wk3} 
	\end{align}
	Moreover, we have the following lemma, which says that for $\epsilon$ and $\theta$ fixed, the sequence $\{w_{k,\theta}\}$ satisfies the inequalities in Lemma~\ref{lem:IVT} as long as $k$ is not too large.
	\begin{lemma} \label{lem:wkprop}
		Let $\theta_* \in (0,1)$,  $k_* \in \N$, and $R > 0$. There exists $\epsilon_*(k_*,\theta_*,R)$ so that whenever $\epsilon \in (0,\epsilon_*)$, $\delta \in (0,1)$, and $k \in \{1,\ldots,k_*\}$ the following is satisfied pointwise for $|x| < 2R$: 
		\begin{equation}
		0 \le \delta \Delta w_{k,\theta_*} + \frac{1}{\epsilon}L_\epsilon^* w_{k,\theta_*} \le \frac{1}{\sqrt{\epsilon}}\left(1+ \delta|\grad f_{\epsilon,\delta}|^2 +\sum_{j=1}^r |Z_j f_{\epsilon,\delta}|^2\right).  \label{eq:wk4}
		\end{equation}		
	\end{lemma}
\begin{proof}
	A direct computation reveals that 
	\begin{align*} 
	\delta \Delta w_{k,\theta_*} + \frac{1}{\epsilon}L_\epsilon^* w_{k,\theta_*} & = \phi_\epsilon^{''}(\tilde{w}_{k,\theta_*})\left(\frac{4}{\theta_*}\right)^{2k}\left(\delta|\grad \tilde{f}_{\epsilon,\delta}|^2 +\sum_{j=1}^r |Z_j \tilde{f}_{\epsilon,\delta}|^2\right) \\ 
	& \quad + \text{Tr}(A)\phi'_\epsilon(\tilde{w}_{k,\theta_*}) \left(\frac{4}{\theta_*}\right)^k \tilde{f}_{\epsilon,\delta} + \text{Tr}(A) w_{k,\theta_*}. 
	\end{align*}
	The lower bound in (\ref{eq:wk4}) is then immediate for any $k \in \N$ due to $\phi_\epsilon'', \phi_\epsilon' \ge 0$.  As for the upper bound, by (\ref{eq:regubd}), $0 \le w_{k,\theta_*} \le 1$, $\|\phi_\epsilon''\|_{L^\infty} \leqc \epsilon^{-1/4}$, and $\|\phi_\epsilon'\|_{L^\infty} \leqc 1$, there exists a constant $C(R)$ such that for any $k \le k_*$ we have	
		$$  \delta \Delta w_{k,\theta_*} + \frac{1}{\epsilon}L_\epsilon^* w_{k,\theta_*} \le C(R) \epsilon^{-1/4} \left(\frac{4}{\theta_*}\right)^{2k_*}\left(1 + \delta|\grad f_{\epsilon,\delta}|^2 +\sum_{j=1}^r |Z_j f_{\epsilon,\delta}|^2\right). $$
	Choosing
	$$ \epsilon_*^{1/4} \le \frac{1}{C(R)} \left(\frac{\theta_*}{4}\right)^{2k_*}$$
	yields (\ref{eq:wk4}).
\end{proof}	

	 The main application of Lemma~\ref{lem:IVT} is the following.
	\begin{lemma} \label{lem:wkiteration}
		Let $R \ge R_0$, $\kappa > 0$, and $\delta \in (0,1)$. There exists $\theta_* \in (0,1/2)$, $\epsilon_* \in (0,1/16)$, and $K \in \N$, all depending only on $\kappa$ and $R$, so that whenever $\epsilon \in (0,\epsilon_*)$ there exists $k_* \in \N$ with $k_* \le K$ such that
		\begin{equation} \label{eq:wkiteration}
		\int_{B_R} |w_{k_*,\theta_*}|^2 \le \kappa.
		\end{equation}
	\end{lemma}
	\begin{proof}
		Let $\epsilon_0 > 0$, $\theta_* \in (0,1/2)$, and $\mu > 0$ denote the parameters guaranteed by applying Lemma~\ref{lem:IVT} at radius $R$ with $\alpha_1 = c_2$ and $\alpha_2 = \kappa$. This fixes $\theta_*$ from the lemma statement. Note that since $c_2$ is universal, $\epsilon_0$, $\theta_*$, and $\mu$ depend only on $\kappa$ and $R$. 
		
		Let $K$ be the first natural number that exceeds $1 + 2|B_R|/\mu$ and observe that $K$ depends only on $R$ and $\kappa$. By Lemma~\ref{lem:wkprop} there exists $\bar{\epsilon}(K,\theta_*,R) < 1/16$ such that (\ref{eq:wk4}) holds whenever $\epsilon \in (0,\bar{\epsilon})$ and $k \le K$. Let $\epsilon_* = \min(\epsilon_0,\bar{\epsilon})$. To complete the proof, it suffices to show that for every $\epsilon \in (0,\epsilon_*)$ there exists $k_* \le K$ such that 
		$$\int_{B_R} |w_{k_*,\theta_*}|^2 \le \kappa.$$
		If this is not the case, then there exists
		$\epsilon' \in (0,\epsilon_*)$ such that
		\begin{equation} \label{eq:wkcontr}
		\int_{B_R} |w^{\epsilon',\delta}_{k,\theta_*}|^2 > \kappa 
		\end{equation}
		for all $k \le K$. In the remainder of this proof we write $w_{k,\theta_*} = w_{k,\theta_*}^{\epsilon',\delta}$.
		Since $0 \le w_{k,\theta_*} \le 1$, it follows from (\ref{eq:wk2}) and (\ref{eq:wkcontr}) that for every $k \le K-1 $ we have
		\begin{equation} 
		|\{w_{k,\theta_*} \ge 1-\theta_*\} \cap B_R| \ge \int_{B_R}|w_{k+1,\theta_*}|^2 \ge \kappa.
		\end{equation}
		Combining with (\ref{eq:wk1}) and (\ref{eq:wk4}), we see that for every $k \le K - 1$, the function $w_{k,\theta_*}$ satisfies the hypotheses of Lemma~\ref{lem:IVT} at radius $R$ with $\alpha_1 = c_2$ and $\alpha_2 = \kappa$. Since $\epsilon' < \epsilon_0$, we obtain that for every $k \le K - 1$ there holds
		\begin{equation}
		|\{0 < w_{k,\theta_*} < 1-\theta_*\} \cap B_R| \ge \mu,
		\end{equation}
		which along with (\ref{eq:wk3}) implies that $|B_R| \ge (K - 1) \mu \ge 2|B_R|$, a contradiction. 
	\end{proof}
	
	We are now ready to complete the proof of Lemma~\ref{lem:IVTtolbd}. For $R \ge R_0$ and $\kappa(R)$ to be chosen sufficiently small we app Lemma~\ref{lem:wkiteration} to obtain $\epsilon_*$, $\theta_*$, and $K$, all depending only on $R$, so that whenever $\epsilon \in (0,\epsilon_*)$ there exists $k_* \le K$ such that 
	\begin{equation}
	\int_{B_R}\left|w^{\epsilon,\delta}_{k_*,\theta_*}\right|^2 \le \kappa.
	\end{equation}	
	From the lower bound in Lemma~\ref{lem:wkprop} (which $w_{k,\theta_*}^{\epsilon,\delta}$ satisfies for any $\epsilon,\delta \in (0,1)$ and $k \in \N$) it follows by Lemma~\ref{lem:Moser} that there exists a constant $C(R)$ such that 
	\begin{equation}
	\|w_{k_*,\theta_*}\|_{L^\infty(B_{R/2})}\le C(R)\|w_{k_*,\theta_*}\|_{L^2(B_R)} \le C(R)\sqrt{\kappa}.
	\end{equation}
	Let $\kappa$ be small enough so that $C(R)\sqrt{\kappa} \le 1/2$. Since $\phi_\epsilon$ is monotone increasing with $\phi(1/2) = 1/2$ this implies that $\sup_{|x| \le R/2}\tilde{w}_{\theta_*,k_*}(x) \le 1/2$. Directly from (\ref{eq:defwk}) we get that 
	\begin{equation} \label{eq:lbdsmalleps}
	\inf_{\epsilon \in (0,\epsilon_*)}\inf_{\delta \in (0,1)} \inf_{|x| \le R/2}f_{\epsilon,\delta}(x) \ge \frac{c_1}{2}\left(\frac{\theta_*}{4}\right)^{K} \gtrsim_R 1.  
	\end{equation}
The bound (\ref{eq:IVTtolbd}) then follows from Lemma~\ref{lem:deltatozero}. If Assumption~\ref{Assumption:positivity} is satisfied, then Lemma~\ref{lem:largeps} and \eqref{eq:IVTtolbd} together yield (\ref{eq:lbd}), which completes the proof.
\end{proof}

Under Assumption~\ref{Assumption:positivity}, the arguments of this section yield
\begin{equation} \label{eq:reglbd}
\inf_{\epsilon, \delta \in (0,1), |x| \le R}f_{\epsilon,\delta}(x) \gtrsim_{R} 1
\end{equation}
for every $R > 0$. Indeed, this follows immediately from \eqref{eq:lbdsmalleps} and Lemma~\ref{lem:largeps}.

\subsection{Global bounds from local ones} \label{sec:StationaryGaussian}
The purpose of this section is to upgrade (\ref{eq:ubd}) to the Gaussian upper bound (\ref{eq:Gaussubd}). By Lemma~\ref{lem:deltatozero} it suffices to prove the following.
\begin{lemma} \label{lem:ubd2Gauss}
	There exists $\lambda > 0$ so that 
	\begin{equation}
	\sup_{\epsilon, \delta \in (0,1)} f_{\epsilon, \delta}(x) \leqc e^{-\lambda \abs{x}^2/2}.
	\end{equation}
\end{lemma}

\begin{proof}
Let $G_\lambda(x) = \exp(-\lambda \abs{x}^2/2)$ for $\lambda > 0$. Note that $Bx \cdot \grad G_\lambda = N \cdot \grad G_\lambda = 0$ because $G_\lambda$ is radially symmetric. Hence, denoting $Z_j = (Z_j^{(1)},\ldots, Z_j^{(r)})$, we have 
\begin{equation} \label{eq:LstarG}
\begin{aligned}
(L^*_\epsilon + \epsilon \delta \Delta)G_\lambda(x) &= \left(\lambda^2\sum_{j=1}^r |Z_j \cdot x|^2 - \lambda \sum_{k,j} (Z_j^{(k)})^2 + \text{Tr}(A) - \lambda Ax\cdot x \right)\epsilon G_\lambda(x) \\ 
& \quad + \left(\lambda^2 \abs{x}^2 - d\lambda\right) \epsilon \delta G_\lambda(x). 
\end{aligned}
\end{equation}
Since $A$ is positive definite, there exists $\lambda_0 > 0$ sufficiently small and $R_0 \ge 1$ sufficiently large, both depending only on $A$, $\{Z_j\}_{j=1}^r$, and $d$, so that $(L_\epsilon^* + \epsilon \delta \Delta)G_{\lambda_0}(x) < 0$ whenever $|x| \ge R_0$. With $R_0$ fixed, we have from Lemmas~\ref{lem:Moser} and~\ref{lem:L2bound} that there exists a constant $C_0(A,\{Z_j\}_{j=1}^r,d)$ such that 
\begin{equation}\label{eq:regubdbarrier}
\sup_{\epsilon, \delta \in (0,1)}\|f_{\epsilon,\delta}\|_{L^\infty(B_{2R_0})} \le C_0.
\end{equation}
We then define the upper barrier function 
\begin{equation}
G^+(x) = 2C_0 e^{4\lambda_0 R_0^2/2}  G_{\lambda_0}(x).
\end{equation}
To prove the lemma it suffices to show for all $x \in \R^d$ there holds 
\begin{equation} \label{eq:Gaussubd0}
\sup_{\epsilon, \delta \in (0,1)} f_{\epsilon, \delta}(x) \leqc G^+(x).
\end{equation}

Let $\mathcal{J}_\eta$ denote the standard mollification at scale $\eta > 0$ and define 
\begin{equation}
g_{\eta} = \mathcal{J}_\eta\left(\mathbf{1}_{|x| \le R_0 + \frac{1}{2}} f_{\epsilon,\delta}  \right).
\end{equation}
Since $f_{\epsilon,\delta}$ is smooth, it follows from \eqref{eq:reglbd}, (\ref{eq:regubdbarrier}), and the definition of $G^+$ that there exists $\eta_0$ so that
\begin{equation} \label{eq:Gaussubd1}
\frac{1}{2}f_{\epsilon,\delta}(x) < g_{\eta_0}(x) \quad \forall|x| \le R_0,
\end{equation}
\begin{equation} \label{eq:Gaussubd2}
g_{\eta_0}(x) < \min\left(G^+(x),\frac{3}{2}f_{\epsilon,\delta}(x)\right) \quad \forall x \in \R^d.
\end{equation}
Due to \eqref{eq:Gaussubd1} and \eqref{eq:AppDistrMoment} we can assume that $R_0$ is large enough so that 
\begin{equation} \label{eq:Gaussubd3}
\int g_{\eta_0} \ge \frac{1}{4}.
\end{equation}

Let $\Pt_t^{\epsilon,\delta}$ denote the Markov semigroup generated by $L_\epsilon + \epsilon \delta \Delta$. For details on the construction and properties of $\Pt_t^{\epsilon,\delta}$ we refer to Appendix~\ref{sec:basics}. Let $\mu$ be the measure on $\R^d$ with density $(\int g_{\eta_0})^{-1}g_{\eta_0}$ and let $\tilde{g}_t: \R^d \to \R$ denote the density of  $(\Pt_t^{\epsilon,\delta})^*(\mu)$. Then, $g_t : = (\int g_{\eta_0}) \tilde{g}_t$ is a global smooth solution to the problem
\begin{equation}
\begin{cases}
\partial_t g_t = L_\epsilon^*g_t + \epsilon \delta \Delta g_t, \\ 
g_0(x) = g_{\eta_0}(x).
\end{cases}
\end{equation}
Define 
$$ t_* = \inf\{t \ge 0: \text{there exists }x\text{ such that }g_{t}(x) = G^+(x)\}  $$
with the convention that $t_* = \infty$ if $g_t(x) < G^+(x)$ for every $x \in \R^d$ and $t > 0$.
By the convergence 
$$\limsup_{t\to \infty}\|\tilde{g}_t - f_{\epsilon,\delta}\|_{L^1} \le \limsup_{t\to \infty}\|(\Pt_{t}^{\epsilon,\delta})^* \mu - \mu_{\epsilon,\delta}\|_{TV} = 0$$ for any $\epsilon, \delta > 0$ and (\ref{eq:Gaussubd3}) there exists a sequence $\{t_k\}_{k=1}^\infty$ such that
$$ \lim_{t \to \infty}g_{t_k}(x) = \left(\int g_{\eta_0}\right) f_{\epsilon,\delta}(x) \ge \frac{1}{4}f_{\epsilon,\delta}(x) $$
for almost every $x \in \R^d$. Thus, since $f_{\epsilon,\delta}$ is smooth, to prove \eqref{eq:Gaussubd0} it is enough to show that $t_* = \infty$.

As in the proof of Lemma~\ref{lem:AppL2density}, it is not hard to verify that $\forall k > 0$, $\exists \gamma' > 0$ such that for all $T < \infty$ and $\delta > 0$, 
\begin{align*}
\sup_{t \in [0,T]}\|e^{\gamma' |x|^2} g_t\|_{H^{k}} \lesssim_{T,\delta,k} 1,
\end{align*}
where $\gamma' < \gamma$ (with $\gamma$ as in Lemma~\ref{lem:AprioriBounds}) does not depend on $\epsilon$ or $\delta$.
Moreover, this implies similar estimates on $\partial_t g_t$ with $\gamma'$ replaced with $\gamma'/2$, and so $g_t$ takes values continuously in such spaces. 
Hence, if $\lambda_0 < \gamma'$ and $t_* < \infty$ there exists a ``first crossing time'' $t_* > 0$; i.e., $(t_*,x_*)$ is such that $g_{t_*}(x_*) = G^+(x_*)$ and $g_t(x) \le G^+(x)$ for all $t \le t_*$ and $x \in \R^d$. 
Suppose for the sake of contradiction that $t_* < \infty$. We have two cases. 

\textbf{Case 1:} $|x_*| < R_0$:
By (\ref{eq:Gaussubd2}) and the fact $(\Pt_t^{\epsilon,\delta})^*$ preserves positivity, we have $g_t(x) \le (3/2)f_{\epsilon,\delta}$ for all $t \ge 0$. Combining with \eqref{eq:regubdbarrier} we obtain
$$\sup_{t \ge 0, |x| < R_0}g_t(x) \le \frac{3C_0}{2}.$$
Since $G^+(x) \ge 2C_0$ whenever $|x| \le 2R_0$ we conclude that $|x_*| < R$ is impossible. 

\textbf{Case 2:} $|x_*| \ge R_0$:
Since $g \in C_t^1C_x^2((0,\infty) \times \R^d)$, it follows from (\ref{eq:Gaussubd2}) and a classical barrier function argument that $(L_\epsilon^* + \epsilon\delta \Delta)G^+(x_*) \ge 0.$ This is a contradiction because we chose $\lambda_0$ and $R_0$ so that  $(L_\epsilon^* + \epsilon\delta \Delta)G^+(x) < 0$ whenever $|x| \ge R_0$.
\end{proof}

\section{Geometric ergodicity} \label{sec:GeoErg}
In this section we prove (\ref{eq:PtVbound}),  Lemmas~\ref{lem:ParReg} and~\ref{lem:Linfty2L2decay}, and the optimality result Theorem~\ref{thrm:optimal}.

We begin with \eqref{eq:PtVbound}, which is a consequence of the following lemma. 
\begin{lemma} \label{lem:PtVbound}
	Let $V: \R^d \to \R$ be a uniform Lyapunov function for $\Pt_t^\epsilon$ with constants $\kappa$ and $b$. Then, for every measurable function $f:\R^d \to \R$ with $\|f\|_{V} < \infty$ there holds 
	$$ \sup_{t \ge 0}\|\Pt_t^\epsilon f\|_{V} \leqc_{\kappa, b} \|f\|_{V}.$$
\end{lemma}

\begin{proof}
	Recall from (\ref{eq:drift}) that 
	$$\Pt_t^\epsilon V(x) \leqc_{\kappa, b}1 + V(x).$$ 
	Hence, 
	\begin{align*}
	\|\Pt_t^\epsilon f\|_{V} = \sup_{x \in \R^d} \frac{|\Pt_t^\epsilon f|(x)}{1+V(x)} \le \|f\|_{V}\sup_{x \in \R^d}\frac{1 + \Pt_t^\epsilon V(x)}{1+V(x)} \leqc_{\kappa, b}\|f\|_{V}.
	\end{align*}
\end{proof}

\subsection{$L^\infty \to L^2_{\mu_\epsilon}$ decay for $\Pt^\epsilon_t$} 

In this section we prove Lemma~\ref{lem:Linfty2L2decay}. A key ingredient is the following hypoelliptic weak Poincar\'{e} type inequality. Recall the notations defined in (\ref{eq:defXparabolic}) and (\ref{eq:defXstarparabolic}).

\begin{lemma} \label{lem:poincinterp}
	Let $R > 0$. For every $\delta > 0$ there exists a constant $C_\delta$ such that for all $\epsilon \in [0,1]$, $t_0 \ge 0$, and $f \in C^\infty((t_0,t_0+1) \times B_{R+1})$ there holds 
	\begin{align*} \|f-\bar{f}\|_{L^2((t_0 + 1/4,t_0 + 3/4)\times B_R)} &\le \delta \|f\|_{L^\infty((t_0,t_0+1) \times B_{R+1})} + C_\delta \sum_{j=1}^r\|Z_j f\|_{L^2((t_0,t_0 + 1)\times B_{R+1})} \\ 
	& \quad + C_\delta \|(\partial_t + Z_{0,\epsilon})f\|_{L^2((t_0,t_0+1), \mathscr{X^*}(B_{R+1}))},
	\end{align*}
	where 
	$$\bar{f} = 2|B_R|^{-1}\int_{t_0 + 1/4}^{t_0 + 3/4}\int_{B_R}f$$
	is the average value of $f$ on $(t_0 + 1/4,t_0 + 3/4) \times B_R$.
\end{lemma}

\begin{proof}[Proof of Lemma~\ref{lem:poincinterp}]
  It suffices to prove the inequality for $t_0 = 0$. Suppose for the sake of contradiction that the result is false. Then, there exists $\delta > 0$ and a sequence $\{(f_n, \epsilon_n)\}_{n=1}^\infty \subseteq C^\infty((0,1) \times B_{R+1}) \times [0,1]$ such that 
	\begin{equation} \label{eq:Poinclem1}
	\begin{aligned}
	\|f_n - \bar{f}_n\|_{L^2((1/4,3/4)\times B_R)} &\ge \delta \|f_n\|_{L^\infty((0,1)\times B_{R+1})} + n\sum_{j=1}^r\|Z_j f_n\|_{L^2((0,1)\times B_{R+1})} \\ 
	& \quad + n\|(\partial_t + Z_{0,\epsilon_n}) f_n\|_{L^2((0,1);\mathscr{X^*}(B_{R+1}))}    
	\end{aligned}
	\end{equation}
	for every $n \in \N$. Let 
	$$g_n = \frac{f_n - \bar{f}_n}{\|f_n - \bar{f}_n\|_{L^2((1/4,3/4)\times B_R)}}.$$
	Dividing (\ref{eq:Poinclem1}) by $\|f_n - \bar{f}_n\|_{L^2((1/4,3/4)\times B_R)}$ and using that 
	$$\|g_n\|_{L^\infty} \le \frac{2}{\|f_n - \bar{f}_n\|_{L^2((1/4,3/4)\times B_R)}}\|f_n\|_{L^\infty((0,1)\times B_{R+1})}$$
	we obtain
	\begin{equation} \label{eq:Poinclem2}
	1 \ge \frac{\delta}{2} \|g_n\|_{L^\infty((0,1)\times B_{R+1})} + n\sum_{j=1}^r\|Z_j g_n\|_{L^2((0,1)\times B_{R+1})} + n\|(\partial_t + Z_{0,\epsilon_n}) g_n\|_{L^2((0,1);\mathscr{X^*}(B_{R+1}))}.
	\end{equation}
	Let $\chi \in C_0^\infty((0,1)\times B_{R+1})$ be a smooth cutoff function with $0 \le \chi \le 1$ and $\chi \equiv 1$ on $(1/8,7/8) \times B_{R+1/2})$. From (\ref{eq:Poinclem2}) it follows readily that 
	\begin{equation} 
	\|\chi g_n\|_{L^2((0,1)\times B_{R+1})} + \sum_{j=1}^r \|Z_j (\chi g_n)\|_{L^2((0,1)\times B_{R+1})} + \|(\partial_t + Z_{0,\epsilon_n})(\chi g_n)\|_{L^2((0,1);\mathscr{X}^*(B_{R+1}))} \leqc \delta^{-1}.
	\end{equation}
	By Assumption~\ref{Assumption:spanning}, $\{\partial_t + Z_{0,\epsilon_n}, Z_1, \ldots Z_r\}$ satisfies the uniform H\"{o}rmander condition on $(0,1) \times B_{R+1}$ with constants that do not depend on $\epsilon_n$ (depending on $R$ however). Thus, Lemma~\ref{lem:HineqOrigQuant} implies that there exists $s > 0$ such that 
	$$ \sup_{n \in \N}\|\chi g_n\|_{H^s(\R \times \R^d)} \leqc \delta^{-1}. $$
	By compact embedding there exists $g_\infty \in L^2((1/8,7/8)\times B_{R+1/2})$ such that (up to a subsequence that we do not relabel) $g_n \to g_\infty$ strongly in $L^2((1/8,7/8)\times B_{R+1/2})$. Moreover,
	\begin{align}
	\int_{1/4}^{3/4}\int_{B_R} g_\infty & = 0, \label{eq:Poinclem3}\\ 
	\sum_{j =1}^r\int_{1/8}^{7/8}\int_{B_{R+1/2}}|Z_j g_\infty|^2 & = 0, \label{eq:Poinclem4} \\ 
	\int_{1/4}^{3/4}\int_{B_R}|g_\infty|^2 & = 1. \label{eq:Poinclem5}
	\end{align}
	By extracting a subsequence can ensure that $\epsilon_n \to \epsilon_\infty \in [0,1].$ Since any function $\varphi\in C_0^\infty((1/8,7/8)\times B_{R+1/2})$ can be extended by zero to a function $\varphi \in C_0^\infty((0,1)\times B_{R+1})$ with $\|\varphi\|_{\mathscr{X}} < \infty$, we must have $(\partial_t + Z_{0,\epsilon_\infty})g_\infty = 0$ in the sense of distributions on $(1/8,7/8) \times B_{R+1/2}$ by \eqref{eq:Poinclem2}. Thus, due to Lemma~\ref{lem:rigidity} we have that $g_\infty$ is constant on $(1/8,7/8) \times B_{R+1/2}$, which contradicts the combination of (\ref{eq:Poinclem3}) and (\ref{eq:Poinclem5}).
\end{proof}

We use Lemma~\ref{lem:poincinterp}  along with Theorem~\ref{thrm:stationary} to prove the following decay estimate, which is the main step in the proof of Lemma~\ref{lem:Linfty2L2decay}.

\begin{lemma} \label{lem:weakpoinc}
	Suppose that Assumption~\ref{Assumption:positivity} holds. Then, there exists a nonincreasing function $\beta:(0,\infty) \to [1,\infty)$ so that for every $s > 0$ and bounded, measurable $f:\R^d \to \R$ there holds, uniformly in $t_0 \ge 1$ and $\epsilon \in (0,1)$, 
	\begin{equation} \label{eq:WPI}
	\int_{t_0+1/4}^{t_0+3/4}\int |\Pt^\epsilon_t f - \mu_\epsilon(f)|^2 d\mu_\epsilon dt \le \beta(s)\sum_{j=1}^r \int_{t_0}^{t_0+1}\int |Z_j \Pt^\epsilon_t f|^2 d\mu_\epsilon dt + s\|f - \mu_\epsilon(f)\|^2_{L^\infty}.
	\end{equation}
\end{lemma}

\begin{proof}
	For simplicity we omit the $\epsilon$ dependence in the notation. 
	
	If $s \ge 1/2$ then the claimed inequality is trivial. Fix $s < 1/2$ and let $g(t) = \Pt_t f - \mu(f)$. By the moment bound (\ref{eq:momentbound}), there exists $R(s)$ sufficiently large so that $\mu(B_R^c) \le s/2$ uniformly in $\epsilon$. Using that $\Pt_t$ propagates $L^\infty$ bounds we have then 
	\begin{equation}  \label{eq:weakpoinc1}
	\int_{t_0+1/4}^{t_0+3/4}\int |g(t)|^2 d\mu dt  \le \int_{t_0+1/4}^{t_0+3/4}\int_{|x| \le R}|g(t)|^2 d\mu dt + \frac{s}{4}\|f - \mu(f)\|_{L^\infty}^2.
	\end{equation}
	
	The goal now is to bound the first term on the right-hand side of (\ref{eq:weakpoinc1}). Let 
$$ \mu_R(C) = \frac{\mu(C \cap B_R)}{\mu(B_R)}, \quad C \in \mathcal{B}(\R^d)$$
and 
	$$g_R = 2 \int_{t_0 + 1/4}^{t_0 + 3/4} \int g(t)d\mu_R dt.$$ 
By adding and subtracting $g_R$ we have 
	\begin{align*}
	\int_{t_0+1/4}^{t_0+3/4}&\int_{|x| \le R} |g(t)|^2 d\mu dt = \mu(B_R)\int_{t_0+1/4}^{t_0+3/4}\int_{|x| \le R} |g(t)|^2 d\mu_{R} dt \\
	& \le 2 \mu(B_R) \int_{t_0+1/4}^{t_0+3/4} \int \left|g(t) - g_R\right|^2 d\mu_R dt  + \mu(B_R)|g_R|^2.
	\end{align*}
	Now, for each $t \ge 0$ we have $\int g(t)d\mu = 0$, and so 
	\begin{align*} 
|g_R|^2 = \frac{4}{\mu(B_R)^2} \left(\int_{t_0 + 1/4}^{t_0 + 3/4} \int_{|x| > R} g(t) d\mu dt\right)^2 \le \left( \frac{\mu(B_R^c)}{\mu(B_R)} \right)^2 \|f - \mu(f)\|_{L^\infty}^2 \le \frac{s}{4}\|f- \mu(f)\|_{L^\infty}^2,
	\end{align*}
	where in the last inequality we used that 
	$$ \left( \frac{\mu(B_R^c)}{\mu(B_R)} \right)^2 \le \left( \frac{s/2}{3/4}\right)^2 = \frac{4s^2}{9} \le \frac{s}{4}. $$
	Combining our estimates thus far and using that $\text{Var}_\nu h \le \mathbb{E}_\nu (h-c)^2$ for any $\nu \in \mathcal{M}(\R^d)$, $c \in \R$, and $h \in L^2_\nu$ we obtain
	\begin{equation} \label{eq:weakpoinc2}
	\int_{t_0+1/4}^{t_0+3/4}\int_{|x| \le R} |g(t)|^2 d\mu dt \le 2\mu(B_R) \int_{t_0+1/4}^{t_0+3/4} \int\left|g(t) - \bar{g}_R \right|^2 d\mu_R dt + \frac{s}{4} \|f - \mu(f)\|_{L^\infty}^2,
	\end{equation}
	where we have introduced 
	$$ \bar{g}_R = \frac{2}{|B_R|}\int_{t_0 + 1/4}^{t_0 + 3/4} \int_{|x| \le R} g(t)dx dt.$$
	Using (\ref{eq:Gaussubd}), there exists a constant $c > 0$ that does not depend on $\epsilon$ such that 
	$$ 2\mu(B_R)\int_{t_0 + 1/4}^{t_0 + 3/4} \int_{|x| \le R} |g(t) - \bar{g}_R|^2 d\mu_R dt \le c \int_{t_0 + 1/4}^{t_0 + 3/4} \int_{|x| \le R} |g(t) - \bar{g}_R|^2 dx dt.$$
By Lemma~\ref{lem:poincinterp} applied with $\delta = \sqrt{s/(8c)}$ and the fact that 
	$$\partial_t g + Z_{0,\epsilon}g = \epsilon \sum_{j=1}^r Z_j^2 (\Pt_t f),$$
	 there exist constants $C_s$ and $C_s'$ such that 
	\begin{align*}
	c\int_{t_0+1/4}^{t_0+3/4} &\int_{|x| \le R}\left|g(t) - \bar{g}_R \right|^2 dx dt \le C_s \sum_{j=1}^r\int_{t_0}^{t_0 + 1}\int_{|x| \le R+1} |Z_j \Pt_t f|^2 dx dt + +  \frac{s}{2}\|f-\mu(f)\|_{L^\infty}^2 \\ 
	& + \epsilon^2 C_s \sup_{\varphi \in C_0^\infty((t_0,t_0+1) \times B_{R+1}), \|\varphi\|_{L^2(t_0, t_0+1; \mathscr{X})} \le 1} \left|\sum_{j=1}^r \int_{t_0}^{t_0+1}\int_{|x| \le R+1} \varphi Z_j^2(\Pt_t f) dx dt\right|^2 \\ 
	& \le C_s' \sum_{j=1}^r\int_{t_0}^{t_0 + 1}\int_{|x| \le R+1} |Z_j \Pt_t f|^2 dx dt +  \frac{s}{2}\|f-\mu(f)\|_{L^\infty}^2.
	\end{align*}
	From (\ref{eq:lbd}) and the fact that $R$ depends only $s$ we have
$$
	C_s' \sum_{j=1}^r \int_{t_0}^{t_0+1}\int_{|x| \le R+1} |Z_j \Pt_t f|^2 dx dt \le C_s'' \sum_{j=1}^r \int_{t_0}^{t_0+1}\int_{|x| \le R+1} |Z_j \Pt_t f|^2 d\mu dt,
    $$
	which along with the estimates proceeding (\ref{eq:weakpoinc2}) yields 
	\begin{equation} \label{eq:weakpoinc3}
	\int_{t_0 + 1/4}^{t_0 + 3/4}\int_{|x| \le R} |g(t)|^2 d\mu dt \le C_s''\sum_{j=1}^r \int_{t_0}^{t_0+1}\int_{|x| \le R+1} |Z_j \Pt_t f|^2 d\mu dt + \frac{3s}{4}\|f- \mu(f)\|_{L^\infty}^2.
	\end{equation}
	Combining (\ref{eq:weakpoinc3}) and (\ref{eq:weakpoinc1}) completes the proof.
\end{proof}

\begin{proof}[Proof of Lemma~\ref{lem:Linfty2L2decay}]
Let $f: \R^d \to \R$ be bounded and Borel measurable. Computing with $\partial_t (\Pt_t^\epsilon f) = L (\Pt_t^\epsilon f)$ we obtain the identity
\begin{equation} \label{eq:lem2.8.1}
    \frac{1}{2}\frac{d}{d\tau}\int|\Pt^\epsilon_\tau f - \mu_\epsilon(f)|^2 d\mu_\epsilon = -\epsilon \sum_{j=1}^r \int |Z_j \Pt_\tau^\epsilon f|^2 d\mu_\epsilon.
\end{equation}
Integrating over $\tau \in (t,t+1)$ for $t \ge 1$ and using Lemma~\ref{lem:weakpoinc} gives, for any $s > 0$, 
\begin{equation} \label{eq:lem2.8.2}
\|\Pt_{t+1}^\epsilon f - \mu_\epsilon(f)\|^2_{L^2_{\mu_\epsilon}} - \|\Pt_{t}^\epsilon f - \mu_\epsilon (f)\|^2_{L^2_{\mu_\epsilon}} \le -\frac{2\epsilon}{\beta(s)} \int_{t+1/4}^{t+3/4}\int \|\Pt_{\tau}^\epsilon f - \mu_\epsilon(f)\|^2_{L^2_{\mu_\epsilon}} d\tau + \frac{2\epsilon s}{\beta(s)}\|f - \mu_\epsilon(f)\|_{L^\infty}^2.
\end{equation}
Let $E(t) = \|\Pt_t^\epsilon f - \mu_\epsilon(f)\|_{L^2_{\mu_\epsilon}}^2$. By (\ref{eq:lem2.8.1}), the energy $E(t)$ is nonincreasing, and so the previous estimate implies
\begin{equation}
    E(t+1) \le \left(1 + \frac{\epsilon}{\beta(s)} \right)^{-1}E(t) + \left(1 + \frac{\epsilon}{\beta(s)} \right)^{-1}\frac{2\epsilon s}{\beta(s)}\|f - \mu_\epsilon(f)\|_{L^\infty}^2 \quad \forall\text{ }t\ge 1.
\end{equation}
Iterating over $t = 1,2,\ldots, n-1$ yields
\begin{align}
    E(n) & \le \left(1 + \frac{\epsilon}{\beta(s)} \right)^{-(n-1)}E(1) + \frac{2\epsilon s}{\beta(s)}\|f - \mu_\epsilon(f)\|_{L^\infty}^2 \sum_{k = 1}^{n-1}\left(1 + \frac{\epsilon}{\beta(s)} \right)^{-k} \\ 
    & \le \left(1 + \frac{\epsilon}{\beta(s)} \right)^{-(n-1)}\|f - \mu_\epsilon(f)\|_{L^\infty}^2 + 2 s \|f- \mu_\epsilon(f)\|_{L^\infty}^2.
\end{align}
This implies that there exists a universal constant $\delta > 0$ such that for all $\epsilon \in (0,1)$ there holds
\begin{equation} \label{eq:lem2.8.3}
E(n) \le e^{-n \delta \epsilon/\beta(s)}\|f- \mu_\epsilon(f)\|_{L^\infty}^2 + 2 s \|f- \mu_\epsilon(f)\|_{L^\infty}^2, \quad \N \ni n \ge 2.
\end{equation}

Let $\bar{\psi}:[0,\infty) \to (0,1]$ be defined by
\begin{equation}
    \bar{\psi}(t) = \inf\{s > 0: e^{-t\delta/\beta(s)} \le s\}.
\end{equation}
It is clear that $\bar{\psi}$ is non-increasing with $\lim_{t\to \infty}\bar{\psi}(t) = 0$, and moreover by (\ref{eq:lem2.8.3}) we have proven that 
$$E(n) \le 3\bar{\psi}(n\epsilon)\|f - \mu_\epsilon(f)\|_{L^\infty}^2, \quad \N \ni n \ge 2.$$
Because $E(t)$ is nonincreasing it follows that 
$$\|\Pt_t^\epsilon f - \mu_\epsilon(f)\|_{L^2_{\mu_\epsilon}}^2 \le 3 \bar{\psi}(\lfloor t \rfloor \epsilon) \|f - \mu_\epsilon(f)\|_{L^\infty}^2 \le 3\bar{\psi}(\epsilon t - 1)\|f-\mu_\epsilon(f)\|_{L^\infty}^2, \quad t \ge 2.$$
Since $\|\Pt_t^\epsilon f - \mu_\epsilon(f)\|_{L^2_{\mu_\epsilon}} \le \|f- \mu_\epsilon(f)\|_{L^\infty}$ for any $t \ge 0$, the proof of (\ref{eq:Linfty2L2decay}) is complete by setting 
$$\psi(t) = \begin{cases} 
3 & t\le 2, \\ 
3\bar{\psi}(t-1) & t > 2.
\end{cases}
$$
\end{proof}

\subsection{$L^2 \to L^\infty$ regularization for $\Pt_t$} \label{sec:ParRegul}

In this section we prove Lemma \ref{lem:ParReg}, which proceeds by a parabolic version of the arguments in Section \ref{sec:StationaryMoser}. 

\begin{proof}[Proof of Lemma~\ref{lem:ParReg}]

Since $\Pt_{t}^\epsilon$ is strong Feller (see Lemma~\ref{lem:MarkovSemigroup}), it follows from the semigroup property and the monotonicity \eqref{eq:lem2.8.1} that it suffices to prove the result for continuous $f$. As above, it is convenient to regularize the problem with $\delta \Delta$ and pass to the limit. Let $\tilde{\Pt}_t^{\epsilon,\delta}$ denote the Markov semigroup generated by $\delta \Delta + \epsilon^{-1}L_\epsilon$ and as before write $\mu_{\epsilon,\delta}$ for its unique invariant measure. For $k \geq 0$, define $R_k = R(1+2^{-k})$ and $t_k = \frac{1}{4} - 2^{-k-3}$. Let $\alpha \in (1,4)$ be such that with $s$ given as in Lemma \ref{lem:Horparabolic} there holds
\begin{align}
\norm{g}_{L^{2\alpha}_t L^{2\alpha}_x} \lesssim \norm{g}_{L^\infty_t L^2_x} + \norm{g}_{L^2_t H^s_x} \label{def:alphap},
\end{align}
and for $k \ge 0$ define $w_k = (\tilde{\Pt}_t^{\epsilon,\delta} f)^{\alpha^k}$. We will show that $\exists C > 0$ such that for every $k \ge 0$ there holds 
\begin{equation} \label{ineq:wk}
\|w_{k}\|_{L_t^{2\alpha}L_x^{2\alpha}((t_{k+1},2-t_{k+1}) \times B_{R_{k+1}})} \le C^k \|w_k\|_{L_t^2 L_x^2((t_k,2-t_k)\times B_{R_k})}.
\end{equation}

Let $\chi_k \in C_0^\infty((t_{k},2-t_{k}) \times B_{R_{k}})$ be a time-dependent, radially-symmetric in space, smooth
cutoff function satisfying $\chi_k(t,x) = 1$ for $\abs{x} \leq R_{k+1}$ and  $2-t_{k+1} \ge t \ge t_{k+1}$. Moreover, we may choose $\chi_k$ so that $|\partial_t \chi_k| \leqc 2^k$ and $|D_x^\beta \chi| \leqc 
R^{-1}2^{|\beta|k}$ for every multi-index with $|\beta| \le 2$. Let $v_k = \chi_k w_k$. By splitting $f= \max(f,0) - \max(-f,0)$ and regularizing with a small constant we may assume without loss of generality that $f > 0$. From the convexity and smoothness of $z \mapsto z^\beta$ away from the origin, for all $k \geq 0$ we then have
\begin{align} \label{eq:v_kPar}
  \partial_s v_k \leq \delta \Delta v_k + \sum_{j=1}^r Z_j^2 v_k - \frac{1}{\epsilon}Z_{0,\epsilon} v_k + S_k
\end{align}
where
\begin{align}
S_k = -[\chi_k,\partial_s]w_k + [\chi_k,\delta \Delta +  \sum_{j=1}^r Z_j^2 - Ax\cdot \grad]w_k. 
\end{align}
Let $g$ be a solution to the Dirichlet problem 
\begin{align} \label{def:gP}
\left\{
\begin{array}{l}
  \partial_s g  = \delta\Delta g + \sum_{j=1}^r Z_j^2 g - \frac{1}{\epsilon}Z_{0,\epsilon} g + S_k \qquad (t,x) \in (0,2) \times B_{2R+1} \\ 
  g|_{t = 0} = 0 \\ 
  g|_{\abs{x} = 2R+1} = 0.
\end{array}
\right.
\end{align}
By the weak parabolic maximum principle, there holds $v_k \leq g$. Pairing \eqref{def:gP} with $g$ and using Gr\"{o}nwall's lemma we obtain
\begin{align}
\norm{g}_{L^\infty_t L^2_x} + \norm{g}_{L^2_t \mathscr{X}_\delta} \lesssim 2^{2k}\|w_k\|_{L^2_t L^2_x((t_{k},2-t_k)\times B_{R_k})}. 
\end{align}
Introducing a radially-symmetric in space cutoff $\chi \in C_0^\infty((1/16,31/16)\times B_{2R+1/2})$ with $\chi(t,x) = 1$ for $(t,x) \in (1/8,15/8) \times 2R$ and using \eqref{def:gP} again we then deduce
\begin{equation}
\|\chi g\|_{L^2 \mathscr{X}_\delta} + \|(\epsilon \partial_t + Z_{0,\epsilon})(\chi g)\|_{L^2 \mathscr{X}_\delta^*} \leqc 2^{2k}\|w_k\|_{L^2_t L^2_x((t_{k},2-t_k)\times B_{R_k})}.
\end{equation}
Therefore, by Assumption~\ref{Assumption:spanning} and the parabolic H\"ormander inequality, Lemma \ref{lem:Horparabolic}, we obtain the bound
\begin{align}
\norm{\chi g}_{L^\infty_t L^2_x} + \norm{\chi g}_{L^2_t H^s_x} \lesssim_R 2^{2k}\|w_k\|_{L^2_t L^2_x((t_{k},2-t_k)\times B_{R_k})}.
\end{align}
By $v_k \leq g$, recalling \eqref{def:alphap}, and the definition of $\chi_k$, there is a constant $C > 0$ such that
\begin{align}
\|w_k\|_{L_{t,x}^{2\alpha}((t_{k+1}, 2-t_{k+1})\times B_{R_{k+1}})}\le \norm{v_k}_{L^{2\alpha}_{t,x}} \leq C^k\|w_k\|_{L^2_{t,x}((t_{k},2-t_k)\times B_{R_k})}, 
\end{align}
which completes the proof of \eqref{ineq:wk}. 

By definition, \eqref{ineq:wk} implies  
\begin{align}
\norm{\tilde{\mathcal{P}}_s^{\epsilon,\delta}f }_{L^{2\alpha^{k+1}}_{t,x} ((t_{k+1} , 2-t_{k+1}) \times B_{R_{k+1}})}
& \leq C^{\sum_{j=0}^k j \alpha^{-j}} \norm{\tilde{\mathcal{P}}^{\epsilon,\delta}_sf}_{L^2_t\left(\frac{1}{8},\frac{15}{8};L^2_x(B_{2R})\right)}. 
\end{align}
Passing to the limit and using $\sum_{j=0}^\infty j \alpha^{-j} < \infty$ along with the definitions of $t_k, R_k$ yields (passing to the limit also in $\delta \to 0$)
\begin{align}\label{eq:ParRegThrm2}
\norm{\tilde{\mathcal{P}}_s^{\epsilon}f}_{L^\infty_t\left(\frac{1}{4},\frac{7}{4};L^\infty_x(B_R)\right)} \leqc_R \norm{\tilde{\mathcal{P}}^{\epsilon}_sf}_{L^2_t\left(\frac{1}{8},\frac{15}{8};L^2_x(B_{2R})\right)}. 
\end{align}
Finally by the uniform lower bound \eqref{eq:lbd} followed by the monotonicity \eqref{eq:lem2.8.1} we have 
\begin{align} \label{eq:ParRegThrm3}
\norm{\tilde{\mathcal{P}}^{\epsilon}_sf}_{L^2_t\left(\frac{1}{8},\frac{15}{8};L^2_x(B_{2R})\right)} \leq \int_{0}^2 \norm{\tilde{\mathcal{P}}_s^{\epsilon} f}_{L^2(B_{2R})}^2 ds \lesssim_R \int_{0}^2 \norm{\tilde{\mathcal{P}}^\epsilon_s f}_{L^2_{\mu_\epsilon}}^2 ds \lesssim \norm{f}_{L^2_{\mu_\epsilon}}^2.  
\end{align}
Combining \eqref{eq:ParRegThrm2} and \eqref{eq:ParRegThrm3} completes the proof of Lemma \ref{lem:ParReg}.
\end{proof}

\subsection{Optimality of Theorem~\ref{thrm:GeoErg}}

In this section we prove Theorem~\ref{thrm:optimal}. The idea is essentially that if one starts the process $(x_t^\epsilon)_{t\ge 0}$ at the origin, then the expected value of the energy $\E|x_t^\epsilon|^2$ must take at least time $t \gtrsim \epsilon^{-1}$ to reach equilibrium. For the basic properties of $(x_t^\epsilon)_{t\ge 0}$, see Lemma~\ref{lem:StochasticWP}.

\begin{proof}[Proof of Theorem~\ref{thrm:optimal}]
		We will only prove the statement about the case $s < 1$, since the other is treated in the same way. Suppose that the claim is false. Then, there exists $s < 1$ and $K,\delta > 0$ so that for all measurable functions $f: \R^d \to \R$ with $\|f\|_{V} < \infty$ there is a sequence $\{\epsilon_n\}_{n=1}^\infty$ with $\lim_n \epsilon_n = 0$ such that for all $t \ge 0$ and $n \in \N$ there holds 
	\begin{equation} \label{eq:optimalcontr}
	\|\Pt_t^{\epsilon_n} f - \mu_{\epsilon_n}(f)\|_{V} \le Ke^{-\delta \epsilon_n^s t}\|f - \mu_{\epsilon_n}(f)\|_{V}.
	\end{equation}
	We will derive a contradiction by considering $f:\R^d \to \R$ defined by $f(x) = x^2$, which clearly satisfies $\|f\|_{V} < \infty$.
	
	By It\^{o}'s formula we have
	\begin{equation} \label{eq:Ito}
	\frac{1}{2}\E |x_t^\epsilon|^2 = \frac{1}{2}\E |x_0^\epsilon|^2 - \epsilon \int_0^t \E (Ax_s^\epsilon \cdot x_s^\epsilon)ds + \epsilon t \sum_{j=1}^r |Z_j|^2.
	\end{equation}
	In statistical steady state this reduces to 
	\begin{equation} \label{eq:EquilEnerg}
	\mu_{\epsilon_n}(f) \approx \int_{\R^d} (Ax \cdot x) \mu_\epsilon(dx) = \sum_{j=1}^r |Z_j|^2,
	\end{equation}
	Next, applying (\ref{eq:Ito}) with $x_0^\epsilon \equiv 0$ gives 
	\begin{equation}
	(\Pt_t^{\epsilon_n}f)(0)  \leqc \epsilon_n t.
	\end{equation}
	Combining the previous two equations we see that there are constants $c, \eta > 0$ sufficiently small so that 
	\begin{equation}
|(\Pt_{c\epsilon_n^{-1}}^{\epsilon_n}f)(0) - \mu_{\epsilon_n}(f)| \ge \eta, \quad \forall n\in \N.
	\end{equation}
	Hence, by (\ref{eq:optimalcontr}) and the upper bound in (\ref{eq:EquilEnerg}) we have
	\begin{equation} 
	\frac{\eta}{2} \le \frac{|(\Pt_{c\epsilon_n^{-1}}^{\epsilon_n}f)(0) - \mu_{\epsilon_n}(f)| }{1+V(0)} \le \|\Pt_{c\epsilon_n^{-1}}^{\epsilon_n}f - \mu_{\epsilon_n}(f)\|_{V} \leqc e^{-\delta c \epsilon_n^{s-1}},
	\end{equation}
	where the implicit constant does not depend on $n$. Since  $s < 1$, sending $n \to \infty$ yields the desired contradiction. 
\end{proof}

\appendix

\section{Qualitative properties and basic well-posedness theorems} \label{sec:basics}

In this section we give the basic well-posedness and regularity results that justify the computations in the paper. We also discuss the qualitative results for $\epsilon \gtrsim 1$ that we need. 

We begin with well-posedness of (\ref{eq:SDE}). In what follows, $\set{W_t^{(j)}}_{j=1}^r$, $\set{\tilde{W}_t^{(k)}}_{k=1}^d$ are independent one-dimensional Wiener processes on a complete probability space $(\Omega, \mathcal{F}, \P)$, $\tilde{W}_t = (\tilde{W}_t^{(1)},\ldots, \tilde{W}_t^{(d)})$, and $\mathcal{F}_t$ denotes the $\sigma$-algebra generated by $\{W_s^{(j)}, \tilde{W}_s: 1 \le j \le r, 0 \le s \le t\}$ and the $\P$-null sets of $\mathcal{F}$. Also, we write $W_t = (W_t^{(1)}, \ldots, W_t^{(r)}, \tilde{W}_t^{(1)}, \ldots, \tilde{W}_t^{(d)})$. 

\begin{lemma} \label{lem:StochasticWP}
	Suppose that Assumption~\ref{Assumption:VF} holds (Assumption~\ref{Assumption:spanning} is not needed). Let $X_0 \in L^2(\Omega;\P)$ be a random variable independent of the $\sigma$-algebra generated by $\cup_{t \ge 0} \mathcal{F}_t$, and let $\mathcal{F}_t^{X_0}$ denote the $\sigma$-algebra generated by $\mathcal{F}_t$ and $X_0$. For $\epsilon \in (0,1)$ and $\delta \in [0,1]$, consider the SDE
	\begin{equation} \label{eq:AppSDE}
	\begin{cases}
	dX_{t}^{\epsilon,\delta} = -\epsilon A X_{t}^{\epsilon,\delta} dt - \epsilon^\alpha B X_{t}^{\epsilon,\delta} dt -  N(X_{t}^{\epsilon,\delta})dt + \sqrt{2\epsilon}\sum_{j=1}^r Z_j dW_t^{(j)} + \sqrt{2\epsilon \delta}d\tilde{W}_t, \\ 
	X_{0}^{\epsilon,\delta} = X_0.
	\end{cases}
	\end{equation} 
	There exists a unique (up to indistinguishability), globally defined $\mathcal{F}_t^{X_0}$-adapted process $(X_{t}^{\epsilon,\delta})_{t\ge 0}$ with continuous sample paths solving the integral form of (\ref{eq:AppSDE}) $\P$-a.s. and such that
	$ \int_0^T\E|X_t|^2 dt < \infty$ for every $T \ge 0$. Let $X_{t,x}^{\epsilon,\delta}$ denote the unique solution with $X_0 = x \in \R^d$. If $x_n \to x$ in $\R^d$ then $X_{t,x_n}^{\epsilon,\delta}$ converges to $X_{t,x}^{\epsilon,\delta}$ $\P$-a.s. uniformly on compact time intervals. Moreover, the solution is continuous with respect to the Wiener trajectory in the sense that there exists a set $\Omega' \subseteq \Omega$ with full measure so that for every fixed $0 < T < \infty$ and $\omega_1,\omega_2 \in \Omega'$ one has that
\begin{equation}
	 \sup_{0 \le t \le T}|X_{t,x}^{\epsilon,\delta}(\omega_1) - X_{t,x}^{\epsilon,\delta}(\omega_2)| \to 0
	  \quad\text{as}\quad \sup_{0 \le t \le T}|W_t(\omega_1) - W_t(\omega_2)|  \to 0.
\end{equation}
	 Similarly, if $\{\delta_n\}_{n=1}^\infty$ is a sequence with $\delta_n \to 0$ then 
	 \begin{equation} \label{eq:ParabolicDeltanToZero}
	 \lim_{n \to \infty}\sup_{0 \le t \le T}|X_{t,x}^{\epsilon,\delta_n}(\omega)-X_{t,x}^{\epsilon}(\omega)| = 0 \quad \P\text{-a.s.}
	 \end{equation}
	  Lastly, if $V$ is any uniform Lyapunov function (see Definition~\ref{def:uniformLyapunov}) with $\kappa, b > 0$ as in \eqref{eq:LyapFuncGen}, then uniformly in $\epsilon \in (0,1)$, $\delta \in [0,1]$, and $x \in \R^d$ there holds
	\begin{equation} \label{eq:AppMoment}
	\E V(X_{t,x}^{\epsilon,\delta}) \le \frac{b}{\kappa} + e^{-\epsilon \kappa t}V(x).
	\end{equation} 
\end{lemma}

\begin{proof}
	Since the noise is additive and the drift is smooth, uniqueness follows from the usual ODE argument using Gr\"{o}nwall's lemma. Due to the energy conservation property $N(x) \cdot x = 0$, global existence can be proven with an approximation scheme that relies on standard energy estimates and a routine stopping time argument. The details needed to carry out the procedure can all be found in \cite{Oksendal2003} and [section 3, \cite{Flandoli2008}]. 
		
	To prove the moment bound (\ref{eq:AppMoment}) we begin by applying It\^{o}'s formula to obtain
	\begin{align}
	e^{\kappa \epsilon t}V(X_{t,x}^{\epsilon,\delta}) &= V(x) + \int_0^t e^{\kappa \epsilon s}(LV(X_{s,x}^{\epsilon,\delta}) + \epsilon \delta \Delta V(X_{s,x}^{\epsilon,\delta}) + \epsilon\kappa V(X_{s,x}^{\epsilon,\delta}))ds \label{eq:AppLyapIto} \\ 
	& \quad + \sqrt{2\epsilon} \sum_{j=1}^r \sum_{k=1}^d \int_0^t \frac{\partial V}{\partial x_k}(X_{s,x}^{\epsilon,\delta})Z_j^{(k)}dW_s^{(j)} + \sqrt{2\epsilon \delta}\sum_{k=1}^d \int_0^t \frac{\partial V}{\partial x_k}(X_{s,x}^{\epsilon,\delta}) d\tilde{W}^{(k)}_s ds \label{eq:AppStochasticInt}.
	\end{align}
	Let $\tau_n(\omega) = \inf\{s \in [0,t]:|X_{s,x}^{\epsilon,\delta}| = n\}$. Applying (\ref{eq:LyapFuncGen}) to estimate (\ref{eq:AppLyapIto}) and then localizing with $\tau_n$ (so that the stochastic integral becomes a martingale) we obtain, uniformly in $n \in \N$,
	\begin{equation}
	\E V(X^{\epsilon,\delta}_{t\wedge \tau_n,x}) \le V(x)\E e^{-\kappa \epsilon t \wedge \tau_n} + \frac{\beta}{\kappa}.
	\end{equation}
	Sending $n \to \infty$ the desired result follows from Fatou's lemma and the fact that $\tau_n \uparrow t$ $\P$-a.s. 
	
	Now we turn to continuity with respect to the Wiener trajectory. For notational convenience we define $Z_0(x) = -\epsilon Ax - \epsilon^\alpha Bx - N(x)$.  Fix $T > 0$ and $x \in \R^d$. For $j = 1,2$ let $F_j:[0,T] \to \R^d$ be continuous and suppose that $x_j:[0,T] \to $ is a continuous solution to the integral equation 
	$$ x_j(s) = x + \int_0^t Z_0(x_j(s))ds + F_j(t), \quad j=1,2.$$
	Since we consider additive noise, it is enough to show that 
	\begin{equation} \label{eq:AppNoiseContGoal}
	\lim_{\delta' \to 0}\sup_{F_2:\|F_1 - F_2\|_{C([0,T];\R^d)} \le \delta'}\sup_{0 \le t \le T}|x_1(t) - x_2(t)| = 0.	\end{equation}
	Since $x_1$ is continuous, there exists $C > |x| + 1$ so that $\sup_{0 \le t \le T}|x_1(t)| \le C$. For $\epsilon' > 0$ fixed and $F_2$ to be chosen close to $F_1$, let $T_*$ be the maximal time so that $|x_1(t) - x_2(t)| \le \epsilon'$ for all $t \in [0,T_*]$. By continuity we have $T_* > 0$, and moreover by a simple Gr\"{o}nwall argument there holds 
	\begin{equation}
	\sup_{0 \le t \le T_*} \leqc_{C,T} \|F_1 - F_2\|_{C([0,T];\R^d)}.
	\end{equation}
	Hence, as long as $\|F_1 - F_2\|_{C([0,T];\R^d)}$ is small in terms of $T$, $C$, and $\epsilon'$, it follows from a bootstrap argument that $T_* = T$. This yields (\ref{eq:AppNoiseContGoal}). Both (\ref{eq:ParabolicDeltanToZero}) and continuity with respect to the initial condition follow from a similar argument. This completes the proof. 
\end{proof}

Recall that for a Polish space $\mathcal{X}$ we write $\mathcal{M}(\mathcal{X})$ for the space of Borel probability measures on $\mathcal{X}$. Also, we denote the space of bounded, Borel measurable function $f:\mathcal{X} \to \R$ by $B_b(\mathcal{X})$. In the setting of Lemma~\ref{lem:StochasticWP} the unique, global solution $X_{t,x}$ is a Markov process with respect to the filtration $\mathcal{F}_t$, and $(x,\omega) \to X_{t,x}(\omega)$ is measurable for fixed $t \ge 0$. This allows one to define the transition probabilities $\Pt^{\epsilon,\delta}_t(x,A) = \P(X^{\epsilon,\delta}_{t,x} \in A)$ and the associated Markov semigroup $\Pt_t^{\epsilon,\delta}: B_b(\R^d) \to B_b(\R^d)$ by $\Pt_t^{\epsilon,\delta} f(x) = \E f(X_{t,x}^{\epsilon,\delta})$. The next lemma is about the regularizing properties of $\Pt_t^{\epsilon,\delta}$ and the uniqueness of its invariant measure. 

\begin{lemma} \label{lem:MarkovSemigroup}
	Suppose that Assumptions~\ref{Assumption:VF} and~\ref{Assumption:spanning} both hold. Then, the Markov semigroup $\Pt_t^{\epsilon,\delta}: B_b(\R^d) \to B_b(\R^d)$ is smoothing in the sense that if $f \in B_b(\R^d)$, then $\Pt_t^{\epsilon,\delta}f$ is smooth in space for each $t > 0$. Similarly, $(\Pt_t^{\epsilon,\delta})^* \mu$ has a smooth density with respect to Lebesgue measure for any $\mu \in \mathcal{M}(\R^d)$ and $t > 0$. Moreover, $\Pt_t^{\epsilon,\delta}$ admits a unique invariant measure $\mu_{\epsilon,\delta}$ and it has a smooth density $f_{\epsilon,\delta}$ satisfying
	\begin{equation} \label{eq:AppDistrMoment}
	\sup_{\epsilon \in (0,1),\delta \in [0,1]}\int V(x)f_{\epsilon,\delta} dx \leqc 1
	\end{equation}
	for any uniform Lyapunov function $V$.
\end{lemma}

\begin{proof}
		For $\mu \in \mathcal{M}(\R^d)$, the measure $\mu_t = (\Pt_{t}^{\epsilon,\delta})^* \mu$ is a distributional solution to Kolmogorov forward equation
		\begin{equation} \label{eq:AppDistrMut}
		(\partial_t - \epsilon \delta \Delta - L_\epsilon^*)\mu_t = 0. 
		\end{equation}
		Since $\{\epsilon Ax + \epsilon^\alpha Bx + N, Z_1,\ldots, Z_r\}$ satisfies the parabolic H\"{o}rmander condition, it is easy to see that $\partial_t - (\epsilon \delta \Delta + L^*_{\epsilon})$ satisfies H\"{o}rmander's condition on $\R^{d+1}$. Hence, the fact that $\mu_t$ has a smooth density with respect to Lebesgue measure for $t > 0$ is a direct consequence of H\"{o}rmander's theorem \cite{H67}. Similarly, it is classical consequence of It\^{o}'s formula that if $f \in C(\R^d)$, then $\Pt_t^{\epsilon,\delta}f$ is a distributional solution to the backward equation 
		\begin{equation} \label{eq:AppDistrf}
		(\partial_t - \epsilon \delta \Delta - L_\epsilon)\Pt_t^{\epsilon,\delta}f = 0.
		\end{equation}
		Using that $X_{t,x}^{\epsilon,\delta}$ has a smooth density with respect to Lebesgue measure for fixed $x$, one can show with a standard approximation argument that (\ref{eq:AppDistrf}) holds when $f$ is just bounded and measurable. The regularity of $\Pt_t^{\epsilon,\delta}f$ for $f\in B_b(\R^d)$ then follows again by H\"{o}rmander's theorem. 
	
	Existence of an invariant measure follows from the moment bound (\ref{eq:AppMoment}) and the Krylov-Boguliubov theorem (see e.g. [Theorem 3.1.1, \cite{DaPratoZabczyk1996}]). Since $\Pt_t^{\epsilon,\delta}$ is strong Feller, to prove uniqueness it suffices to show that any invariant measure contains the origin in its support. This is a standard consequence of the dissipative structure of (\ref{eq:AppSDE}), the continuity with respect to the Wiener trajectory proven in Lemma~\ref{lem:StochasticWP}, and the fact that $\P(\sup_{0 \le t \le T}|W_t| \le \epsilon') > 0$ for any $\epsilon', T > 0$. Lastly, the moment bound (\ref{eq:AppDistrMoment}) is proven in the usual way by approximating $V$ with $\min(V,n)$ for $n \in \N$, iteratively applying (\ref{eq:AppMoment}), and then sending $n \to \infty$. 
\end{proof} 

\begin{remark} \label{rem:uniqueness}
	As a consequence of the uniqueness described in Lemma~\ref{lem:MarkovSemigroup}, for any $\epsilon \in (0,1)$ and $\delta \in [0,1]$ the only probability measure $\mu$ solving 
	$$(L_\epsilon^* + \epsilon \delta \Delta)\mu = 0$$
	in the sense of distributions is $\mu_{\epsilon,\delta}$.
\end{remark}

An important qualitative result used in the proof of Theorem~\ref{thrm:stationary} is that the density of $\mu_{\epsilon,\delta}$ is in $L^2$. Since this is a distinctively PDE type estimate, it requires an argument beyond the classical probabilistic ones used above. The goal is essentially to make rigorous the computation in Remark~\ref{rem:ellipticL2}. 

\begin{lemma}\label{lem:AppL2density}
	In the setting and assumptions of Lemma~\ref{lem:MarkovSemigroup}, the smooth density $f_{\epsilon,\delta}$ of $\mu_{\epsilon,\delta}$ is in $L^2$ whenever $\delta > 0$.
\end{lemma}

	\begin{proof}
		Fix $\epsilon, \delta > 0$. Let  $X_{t,x}^{(n)}$ denote the unique, global solution with initial condition $x \in \R^d$ to the SDE \eqref{eq:AppSDE} with $B$ and $N$ multiplied by a radially symmetric cutoff $\chi_n \in C_0^\infty(B_{2n})$ with $\chi(x) = 1$ for $|x| \le n$. Let $\Pt_t^{n}$ denote the semigroup generated by $(X_t^{(n)})_{t \ge 0}$. Note that both $(\Pt_t^{(n)})^*$ and $(\Pt_t^{\epsilon,\delta})^*$ are well-posed on $L^1$ and preserve positivity by the well-posedness of the underlying stochastic flows.
		
		Let $\rho \in C_0^\infty(B_1)$ be a probability density function with $\mu(dx) = \rho(x)dx$, $(\Pt_{t}^{\epsilon,\delta})^*\mu =  \rho_t(x)dx$, and $(\Pt_t^{n})^*\mu = \rho_t^{(n)}(x)dx$. For all $\delta > 0$ and $n < \infty$, the Kolmogorov equation for $\rho_t^{(n)}$ is a compact perturbation of a Fokker-Plank operator, and is thus well-posed on $L^2$ spaces with inverse Gaussian weights, in particular, for $\gamma$ sufficiently small, we have that $e^{\gamma \abs{x}^2}\rho \in L^2 \Rightarrow e^{\gamma \abs{x}^2} \rho_t^{(n)} \in L^2$ and that the norm can be estimated above independently of $n$. 
Using standard energy estimates, for all $\delta > 0$, $n < \infty$ one can further show finite-time propagation of the following norms (using the standard multi-index notation, $\alpha \in \mathbb N^d$), 
\begin{align*}
\norm{\rho}_{N_k} = \sum_{\abs{\alpha} \leq k} \norm{e^{ (1+\abs{\alpha})^{-1} \gamma \abs{x}^2} D^\alpha \rho }_{L^2},  
\end{align*}
again with an $n$-independent upper bound. 
Passing to the limit in $n \to \infty$ (using uniqueness for $\rho_t$) we see that $\rho_t^{(n)} \to \rho_t$ (up to extraction of a subsequence) strongly in $H^k$ for any $k < \infty$. This allows us to justify energy estimates on the equation for $\rho_t$, namely, 
		\begin{equation} \label{eq:AppParabolic}
		\begin{cases}
		\partial_t \rho_t(x) = L_\epsilon^* \rho_t(x) + \epsilon \delta \Delta \rho_t(x) & (t,x) \in (0,\infty) \times \R^d \\ 
		\rho_0(x) = \rho(x) & x \in \R^d.
		\end{cases}
		\end{equation}
In particular, we have 
		\begin{equation}
		\frac{d}{dt}\| \rho_t\|_{L^2}^2 + \|\grad \rho_t \|_{L^2}^2 \leqc_{\delta} \|\rho_t\|_{L^2}^2.
		\end{equation}
		Applying the Gagliardo-Nirenberg inequality and $\int \rho_t = 1$ we can then show with an estimate analogous to the one in Remark~\ref{rem:ellipticL2} that 
		\begin{equation} \label{eq:rhoH1est}
		\int_0^t \|\rho_s\|_{H^1}ds \leqc t + \int_0^t \|\grad \rho_s\|_{L^2}^2 ds \leqc_{\delta} t.
		\end{equation}
		For $n \ge 1$ we define the probability density function
		$$\rho_{n,\text{KB}} = \frac{1}{n} \int_0^n \rho_s ds.$$
		By \eqref{eq:rhoH1est}, the sequence $\{\rho_{n,\text{KB}}\}_{n=1}^\infty$ is uniformly bounded in $H^1$, and so passing to a subsequence (which we do not relabel) we obtain a limit $\rho_{\infty} \in H^1$ with $\lim_{n \to \infty}\rho_{n,\text{KB}} = \rho_\infty$ weakly in $H^1$ and strongly in $L^2$ on compact subsets. We may also assume that $\rho_{n,\text{KB}} \to \rho_\infty$ pointwise a.e., so $\rho_\infty \ge 0$ a.e. Moreover, due to (\ref{eq:AppMoment}) there holds $\sup_{t \ge 0} \int e^{\gamma |x|^2} \rho_t(x)dx < \infty$ for $\gamma > 0$ small enough, which when combined with the strong $L^2_{\text{loc}}$ convergence $\rho_{n,\text{KB}} \to \rho_\infty$ implies that $\int \rho_\infty = 1$. Similar to the proof of the  Krylov-Bogoliubov theorem we can show that $\rho_\infty$ solves $(L_{\epsilon}^* + \epsilon \delta \Delta)\rho_\infty = 0$ in the sense of distributions. By the uniqueness described in Lemma~\ref{lem:MarkovSemigroup} we conclude that $f_{\epsilon,\delta} = \rho_\infty \in H^1$, which completes the proof.
\end{proof}

Next, we have a lemma regarding the elliptic regularization, which justifies our approximation arguments with $f_{\epsilon,\delta}$.
\begin{lemma} \label{lem:deltatozero}
	For all $\epsilon > 0$, $k \geq 0$, and $R> 0,$ 
	\begin{align}
	\sup_{\delta \in [0,1]}\norm{f_{\epsilon,\delta}}_{H^k(B_R)} & \lesssim_{k,\epsilon,R} 1. \label{ineq:efinHk}
	\end{align}
	For each fixed $\epsilon > 0$ there holds, for all $k \geq 0$ and $R > 0$, 
	\begin{align}
	\lim_{\delta \to 0} \norm{f_{\epsilon,\delta} - f_{\epsilon}}_{H^k(B_R)} & = 0. \label{ineq:delconv}
	\end{align}
\end{lemma} 
\begin{proof}
	Let $s \in (0,1)$ be given as in Lemma \ref{lem:HineqOrigQuant}.
	Let $k \leq sJ$ for $J \in \N$ fixed and define a decreasing sequence of radially-symmetric, smooth
	cutoff functions $\chi_j$ which satisfy $\chi_j(x) = 1$ for $\abs{x} \leq R + J - j  $ and $\chi_j(x) = 0$ for $\abs{x} > R + J - j + 1$.
    Define $\brak{\grad}^s$ as the Fourier multiplier
\begin{align*}
\widehat{\brak{\grad}^s u}(\xi) = \left(1 + \abs{\xi}^2\right)^{s/2} \widehat{u}(\xi). 
\end{align*}
    Let $v_0 = \chi_0 f_{\epsilon,\delta}$ and $v_j = \brak{\grad}^{sj}\chi_j f_{\epsilon,\delta}$. Then, 
	\begin{align}
	\epsilon \delta \Delta v_j + L^*_\epsilon v_j + \brak{\grad}^{sj}[\chi_j,\epsilon \delta \Delta + \epsilon \sum_{k=1}^r Z_k^2]f_{\epsilon,\delta} + \brak{\grad}^{sj}[\chi_j,Ax\cdot \grad]f_{\epsilon,\delta} + \mathcal{C}_j = 0, \label{eq:vjtriv}
	\end{align}
	where we denote
	\begin{align}
	\mathcal{C}_j = [\brak{\grad}^{sj}, Z_{0,\epsilon}\cdot \grad]\chi_j f_{\epsilon,\delta}.    
	\end{align}
	Note that
	\begin{align}
	\epsilon \delta\norm{\grad v_j}_{L^2}^2 + \epsilon\sum_{k=1}^r\norm{Z_k v_j}_{L^2}^2 \lesssim_{R,k} \norm{v_j}_{L^2}^2 + \mathbf{1}_{j \ge 1} \|\brak{\grad}^s v_{j-1}\|_{L^2}^2 + \|f_{\epsilon,\delta}\|_{L^2}^2 + \abs{\int v_j \mathcal{C}_j dx}.  
	\end{align}
	To bound the term involving $C_j$, we first rewrite it on the Fourier side to obtain
	\begin{equation}
	\left| \int v_j C_j dx\right| \leqc \int \int |\hat{v}_j(\xi)| |\brak{\xi}^{sj} - \brak{\eta}^{sj}| |\widehat{\chi Z_{0,\epsilon}}(\xi - \eta)| |\eta| |\widehat{\chi_j f_{\epsilon,\delta}}(\eta)| d\eta d\xi,
	\end{equation}
	where $\chi \in C_0^\infty(\R^d)$ is a smooth cutoff with $\chi(x) = 1$ for all $|x| \le R+J+2$. By splitting the integral between the regions $|\xi - \eta| > |\eta|/2$, $|\xi - \eta| \le |\eta|/2$ and using the mean value theorem in the latter piece to deduce $|\brak{\xi}^{sj} - \brak{\eta}^{sj}| \leqc \brak{\xi - \eta}\brak{\eta}^{sj-1}$ we can show
	\begin{align}
	\abs{\int v_j \mathcal{C}_j dx} \lesssim \norm{v_j}_{L^2} \left(\norm{v_{j}}_{L^2} + \norm{f_{\epsilon,\delta}}_{L^2} \right). 
	\end{align}
	Pairing \eqref{eq:vjtriv} with test functions similarly gives
	\begin{align}
	\norm{Z_{0,\epsilon} v_j}_{\mathscr{X}_\delta^*} \lesssim \norm{v_{j}}_{L^2} + \norm{f_{\epsilon,\delta}}_{L^2} + \mathbf{1}_{j \ge 1}\|\brak{\grad}^s v_{j-1}\|_{L^2}.
	\end{align}
	Therefore, by Lemma \ref{lem:HineqOrigQuant}, we have, independent of $\delta$, 
	\begin{align}
	\|v_j\|_{H^s} \lesssim \|v_j\|_{L^2} + \mathbf{1}_{j\ge 1}\|\brak{\grad}^s v_{j-1}\|_{L^2} + \|f_{\epsilon,\delta}\|_{L^2} \leqc \|f_{\epsilon,\delta}\|_{L^2} + \mathbf{1}_{j \ge 1}\|v_{j-1}\|_{H^s}.
	\end{align}
	Iterating gives \eqref{ineq:efinHk}. From there, to deduce \eqref{ineq:delconv} we first use compact embedding to extract a subsequence $\{f_{\epsilon,\delta_n}\}_{n=1}^\infty$ with $\delta_n \to 0$ and a limit $f_{\epsilon,0} \in C^\infty$ with $\lim_{n \to \infty}f_{\epsilon,\delta_n} = f_{\epsilon,0}$ in $H^k_{\text{loc}}$ for every $k$. Clearly, $f_{\epsilon,0} \ge 0$ and $L_\epsilon^* f_{\epsilon,0} = 0$. Moreover by (\ref{eq:AppDistrMoment}) we have $\int f_{\epsilon,0} = 1$. Hence, $f_{\epsilon,0} = f_\epsilon$ by uniqueness, which completes the proof.
\end{proof} 

We conclude with a qualitative lower bound for $f_{\epsilon,\delta}$ that holds for $\epsilon \gtrsim 1$. 

\begin{lemma} \label{lem:largeps}
Suppose that Assumption~\ref{Assumption:positivity} holds. Then, for any $R \ge 1$ and $\epsilon_* \in (0,1)$ there exists $C(\epsilon_*,R) > 0$ such that 
\begin{equation}
\inf_{\epsilon \in [\epsilon_*,1], \delta \in [0,1]} \inf_{|x| \le R} f_{\epsilon,\delta}(x) \ge C.
\end{equation}
\end{lemma}

\begin{proof}
	First, note that $f_{\epsilon,\delta}$ is strictly positive for all $\epsilon \in (0,1)$, $\delta \in [0,1]$. Indeed, $f_{\epsilon} > 0$ by assumption, and the fact that $f_{\epsilon,\delta} > 0$ when $\delta > 0$ follows from the classical elliptic Harnack inequality. Now, if the claim is false, then using \eqref{ineq:efinHk} and the argument used to prove \eqref{ineq:delconv} we can obtain $(\epsilon_0, \delta_0, x_0) \in [\epsilon_*,1]\times [0,1] \times \bar{B}_R$ such that $f_{\epsilon_0,\delta_0}(x_0) = 0$, which contradicts $f_{\epsilon_0,\delta_0} > 0$.
\end{proof}

\phantomsection
\addcontentsline{toc}{section}{References}
\bibliographystyle{abbrv}

\end{document}